\documentclass[12pt]{article}
\pdfoutput=1
\usepackage[authoryear]{natbib}
\usepackage{amsmath,amssymb,amsfonts,amsthm,bbm}
\usepackage{epic,eepic,epsfig,longtable}
\usepackage{multirow,verbatim}
\usepackage{array}
\usepackage{verbatim}
\usepackage{epsfig}
\usepackage{setspace}
\usepackage{float}
\usepackage{CJK}
\usepackage{color}
\usepackage{algorithm}

\makeatletter
\def\singlespace{\def\baselinestretch{1}\@normalsize}

\makeatletter
\def\singlespace{\def\baselinestretch{1}\@normalsize}

\numberwithin{equation}{section}

\renewcommand{\hat}{\widehat}

\newcommand{\bfm}[1]{\ensuremath{\mathbf{#1}}}
     \def\bA{\bfm A}          
     \def\bB{\bfm B}     \def\cB{{\cal  B}}     
\def\bc{\bfm c}     \def\bC{\bfm C}     \def\cC{{\cal  C}}     
     \def\bD{\bfm D}          
\def\be{\bfm e}     \def\bE{\bfm E}          
    \def\bF{\bfm F}

     \def\bI{\bfm I}

     \def\bL{\bfm L}     \def\cL{{\cal  L}}     
     \def\bM{\bfm M}     \def\cM{{\cal  M}}     
          \def\cN{{\cal  N}}     
               
          \def\cP{{\cal  P}}     
               
     \def\bR{\bfm R}          
     \def\bS{\bfm S}     \def\cS{{\cal  S}}     
               
\def\bu{\bfm u}     \def\bU{\bfm U}          
\def\bv{\bfm v}     \def\bV{\bfm V}          
     \def\bW{\bfm W}     \def\cW{{\cal  W}}     
\def\bx{\bfm x}     \def\bX{\bfm X}          
\def\by{\bfm y}     \def\bY{\bfm Y}          
\def\bz{\bfm z}     \def\bZ{\bfm Z}          
\def\bzero{\bfm 0}

\newcommand{\bfsym}[1]{\ensuremath{\boldsymbol{#1}}}
       \def \bbeta    {\bfsym{\beta}}
       
\def \bepsilon {\bfsym{\epsilon}}     
         \def \btheta   {\bfsym{\theta}}

\def \brho     {\bfsym{\rho}}

\def \bGamma   {\bfsym{\Gamma}}       \def \bDelta   {\bfsym{\Delta}}
\def \bTheta   {\bfsym{\Theta}}       \def \bLambda  {\bfsym{\Lambda}}
          
\def \bSigma   {\bfsym{\Sigma}}

\renewcommand{\hat}{\widehat}

\def \heps     {\hat{\heps}}

\DeclareMathOperator*{\argmin}{argmin}

\DeclareMathOperator{\diag}{diag}
\DeclareMathOperator{\E}{E}

\DeclareMathOperator{\rank}{rank}

\DeclareMathOperator{\sgn}{sgn}

\DeclareMathOperator{\Var}{Var}

\DeclareMathOperator{\tr}{tr}

\def \sgn   {\mbox{sgn}}

 at 8truept

\def\today{\ifcase\month\or
  January\or February\or March\or April\or May\or June\or
  July\or August\or September\or October\or November\or December\fi
  \space\number\day, \number\year}

\def \newpage {\vfill\eject}

\newdimen\biblioindent\biblioindent=30pt
\newcommand{\beq}  {\begin{equation}}
\newcommand{\eeq}  {\end{equation}}
\newcommand{\beqn} {\begin{eqnarray}}
\newcommand{\eeqn} {\end{eqnarray}}
\newcommand{\beqnn}{\begin{eqnarray*}}
\newcommand{\eeqnn}{\end{eqnarray*}}

\allowdisplaybreaks
\renewcommand{\baselinestretch}{1.33}
\newtheorem{lem}{Lemma}

\newtheorem{thm}{Theorem}
\newtheorem{rem}{Remark}
\newcounter{CondCounter}

\newcommand{\lonenorm}[1]{\lVert#1\rVert_1}
\newcommand{\ltwonorm}[1]{\lVert#1\rVert_2}

\newcommand{\opnorm}[1]{\lVert#1\rVert_{op}}
\newcommand{\fnorm}[1]{\lVert#1\rVert_F}
\newcommand{\nnorm}[1]{\lVert#1\rVert_N}
\newcommand{\supnorm}[1]{ \lVert#1  \rVert_{\max}}
\newcommand{\inn}[1]{\langle #1 \rangle}
\newcommand{\truncate}[1]{\sgn(#1)(|#1|\wedge \tau)}
\newcommand{\shrunk}[1]{#1 (\|#1\|_2 \wedge \tau)/\ltwonorm{#1}}

\def \Rd       {\mathbb{R}^d}
\def \Rdd	{\mathbb{R}^{d_1\times d_2}}
\def \RR	{\mathbb{R}}

\def \vec	{\text{vec}}
\def \mat {\text{mat}}

\def \ind {1}
\def \Tr {\text{Tr}}
\def \bbX {\mathbb{X}}
\def \bbY {\mathbb{Y}}

\begin{document}

	\title{\vspace*{-0.5 in} A Shrinkage Principle for Heavy-Tailed Data:\\
    High-Dimensional Robust Low-Rank Matrix Recovery\thanks{
    The research was supported by NSF grants DMS-1206464 and DMS-1406266 and NIH grant R01-GM072611-12. }}
	\author{Jianqing Fan, Weichen Wang, Ziwei Zhu\\
     \normalsize
    Department of Operations Research and Financial Engineering\\
    \normalsize
    Princeton University.
    \date{}}

	\maketitle
	
	\begin{abstract}
	This paper introduces a simple principle for robust high-dimensional statistical inference via an appropriate shrinkage on the data.  This widens the scope of high-dimensional techniques, reducing the moment conditions from sub-exponential or sub-Gaussian distributions to merely bounded second or fourth moment. As an illustration of this principle, we focus on robust estimation of the low-rank matrix $\Theta^*$ from the trace regression model $Y=\Tr (\Theta^{*T}X) +\epsilon$.  It encompasses four popular problems: sparse linear models, compressed sensing, matrix completion and multi-task regression.  We propose to apply penalized least-squares approach to appropriately truncated or shrunk data. Under only bounded $2+\delta$ moment condition on the response, the proposed robust methodology yields an estimator that possesses the same statistical error rates as previous literature with sub-Gaussian errors. For sparse linear models and multi-tasking regression, we further allow the design to have only bounded fourth moment and obtain the same statistical rates, again, by appropriate shrinkage of the design matrix. As a byproduct, we give a robust covariance matrix estimator and establish its concentration inequality in terms of the spectral norm when the random samples have only bounded fourth moment. Extensive simulations have been carried out to support our theories.
	\end{abstract}
	
	\textbf{Keywords:} Robust Statistics, Shrinkage, Heavy-Tailed Data, Trace Regression, Low-Rank Matrix Recovery, High-Dimensional Statistics.
	
\newpage	
	\section{Introduction}
		
		Heavy-tailed distributions are ubiquitous in modern statistical analysis and machine learning problems. They are stylized features of high-dimensional data.  By chance alone, some of observable variables in high-dimensional datasets can have heavy or moderately heavy tails (see right panel of Figure~\ref{fig1}).  It has been widely known that financial returns and macroeconomic variables exhibit heavy tails, and large-scale imaging datasets in biological studies are often corrupted by heavy-tailed noises due to limited measurement precisions. Figure~\ref{fig1} provides some empirical evidence on this which is pandemic to high-dimensional data.
These stylized features and phenomena contradict the popular assumption of sub-Gaussian or sub-exponential noises in the theoretical analysis of standard statistical procedures.  They also have adverse impacts on the methods that are popularly used. Simple and effective principles are needed for dealing with moderately heavy or heavy tailed data.

\begin{figure}[thbp]
\begin{center}
{\includegraphics[height=1.8 in]{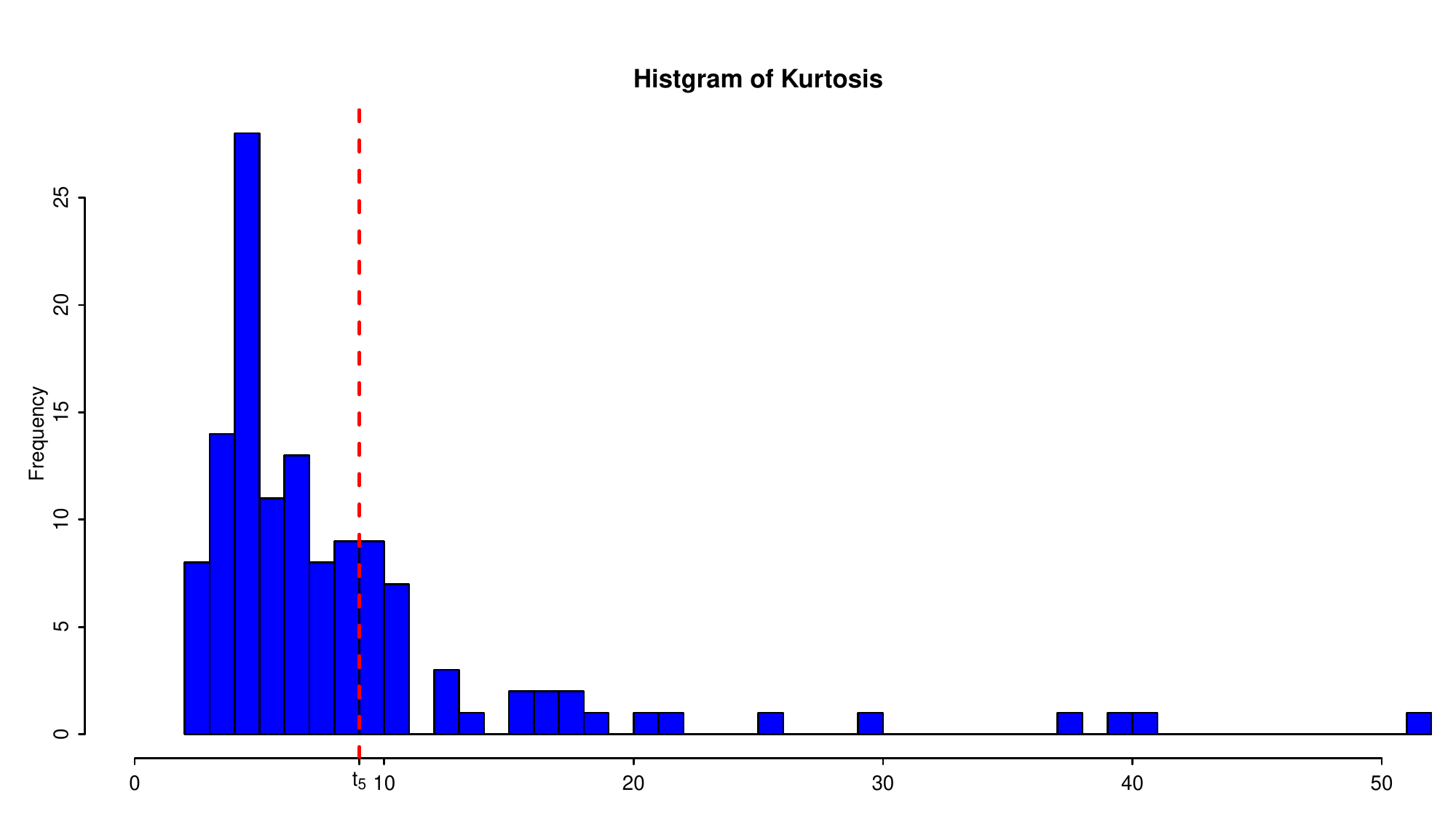} \includegraphics[height=1.7 in]{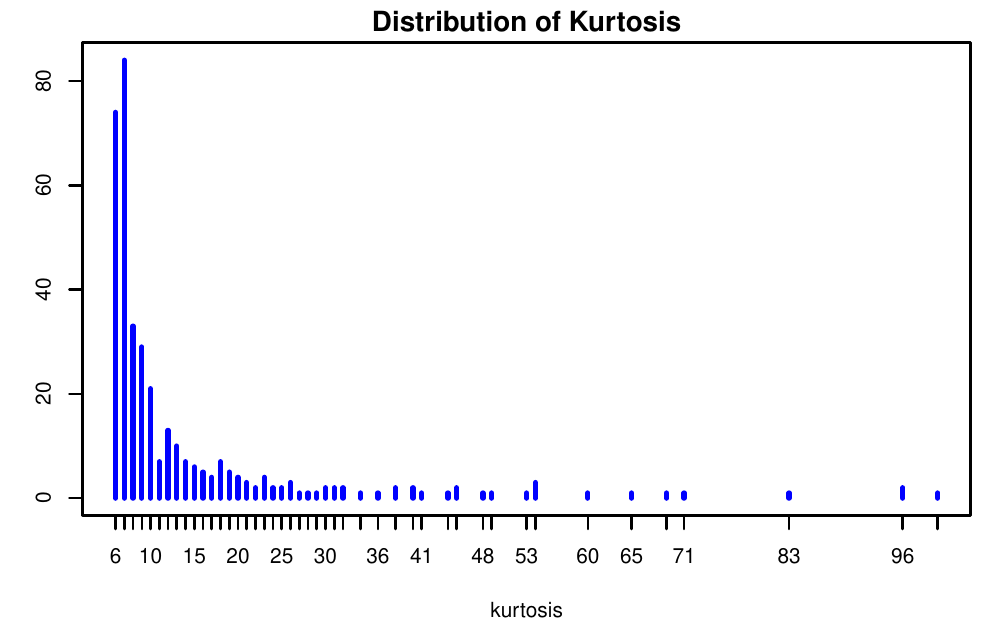}}
\vspace{-0.2 in}
\end{center}
\caption{\small {\bf Distributions of kurtosis of macroeconomic variables and gene expressions}. Red dashline marks variables with empirical kurtosis equals to that of $t_5$-distribution.  Left panel: For 131 macroeconomics variables in \cite{stock2002macroeconomic}.  Right panel: For logarithm of expression profiles of 383 genes based on RNA-seq for autism data \citep{gupta2014transcriptome}, whose kurtosis is bigger than that of $t_5$ among 19122 genes. \label{fig1} }
\end{figure}

		Recent years have witnessed increasing literature on the robust mean estimation when the population distribution is heavy-tailed. \cite{Cat12} proposed a novel approach that is through minimizing a robust empirical loss. Unlike the traditional $\ell_2$ loss, the robust loss function therein penalizes large deviations, thereby making the correspondent M-estimator  insensitive to extreme values. It turns out that when the population has only finite second moment, the estimator has exponential concentration around the true mean and enjoys the same rate of statistical consistency as the sample average for sub-Gaussian distributions. \cite{BJL15} pursued the Catoni's mean estimator further by applying it to empirical risk minimization. \cite{FLW16} utilized the Huber loss with diverging threshold, called robust approximation to quadratic (RA-quadratic), in a sparse regression problem and showed that the derived M-estimator can also achieve the minimax statistical error rate. \cite{Loh15} studied the statistical consistency and asymptotic normality of a general robust $M$-estimator and provided a set of sufficient conditions to achieve the minimax rate in the high-dimensional regression problem.
		
		Another  effective approach to handle heavy-tailed distribution is the so-called ``median of means" approach, which can be traced back to \cite{NYD82}. The main idea is to first divide the whole samples into several parts and take the median of the means from all  pieces of sub-samples as the final estimator. This ``median of means" estimator also enjoys exponential large deviation bound around the true mean. \cite{HSa16} and \cite{Min15} generalized this idea to multivariate cases and applied it to robust PCA, high-dimensional sparse regression and matrix regression, achieving minimax optimal rates up to logarithmic factors.
		

In this paper, we propose a simple and effective principle: truncation of univariate data and more generally shrinkage of multivariate data to achieve the robustness. We will illustrate our ideas through a general model called the trace regression
$$
    Y=\Tr (\bTheta^{*T}\bX)+\epsilon,
$$
which embraces linear regression, matrix or vector compressed sensing, matrix completion and multi-tasking regression as specific examples. The goal is to estimate the coefficient matrix $\bTheta^*\in \RR^{d_1\times d_2}$, which is assumed to have a nearly low-rank structure in the sense that its Schatten norm is constrained: $\sum\limits_{i=1}^{\min(d_1, d_2)} \sigma_i(\bTheta^*)^q\le \rho$ for $0\le q<1$, where $\sigma_i(\bTheta^*)$ is the $i^{th}$ singular value of $\bTheta^*$, i.e., the square-root of the $i^{th}$ eigenvalue of $\bTheta^*{}^T \bTheta^*$. In other words, the singular values of $\bTheta^*$ decay fast enough so that $\bTheta^*$ can be well approximated by a low-rank matrix. We always consider the high-dimensional setting where the sample size $n\ll d_1d_2$. As we shall see, appropriate data shrinkage allows us to recover $\bTheta^*$ with only bounded moment conditions on noise and design.

As the most simple and important example of low-rank trace regression, sparse linear regression and compressed sensing have become a hot topic in statistics research in the past two decades. See, for example, \cite{Tib96}, \cite{CDS01}, \cite{FLi01}, \cite{Don06}, \cite{CTa06}, \cite{CTa07}, \cite{candes2008restricted}, \cite{nowak2007gradient}, \cite{fan2008sure}, \cite{zou2008one}, \cite{BRT09}, \cite{zhang2010nearly}, \cite{NWa12}, \cite{donoho2013accurate}. These pioneering works explore the sparsity to achieve accurate signal recovery in high dimensions.

Recently significant progresses have been made on low-rank matrix recovery under high-dimensional settings. One of the most well-studied approaches is the penalized least-squares method. \cite{NWa11} analyzed the nuclear norm penalization in estimating nearly low-rank matrices under the trace regression model. Specifically, they derived non-asymptotic estimation error bounds in terms of the Frobenius norm when the noise is sub-Gaussian. \cite{RTs11} proposed to use a Schatten-$p$ quasi-norm penalty where $p\le 1$, and they derived non-asymptotic bounds on the prediction risk and Schatten-$q$ risk of the estimator, where $q\in [p,2]$. Another effective method is through nuclear norm minimization under affine fitting constraint. Other important contributions include \cite{RFP10}, \cite{CPl11}, \cite{CZh14}, \cite{CZh15}, etc. When the true low-rank matrix $\bTheta^*$ satisfies certain restricted isometry property (RIP) or similar properties, this approach can exactly recover $\bTheta^*$ under the noiseless setting and enjoy sharp statistical error rate with sub-Gaussian and sub-exponential noise.

		There has also been great amount of work on matrix completion.
\cite{candes2009exact} considered matrix completion under noiseless settings and gave conditions under which exact recovery is possible.
\cite{CPl10} proposed to fill in the missing entries of the matrix by nuclear-norm minimization subject to data constraints, and showed that $rd\log^2 d$ noisy samples suffice to recover a $d\times d$ rank-$r$ matrix with error that is proportional to the noise level.
\cite{recht2011simpler} improves the results of \cite{candes2009exact} on the number of observed entries required to reconstruct an unknown low-rank matrix.
\cite{NWa12} instead used nuclear-norm penalized least squares to recover the matrix. They derived the statistical error of the corresponding M-estimator and showed that it matched the information-theoretic lower bound up to logarithmic factors.

		Our work aims to handle the presence of heavy-tailed, asymmetrical and heteroscedastic noises in the general trace regression. Based on the shrinkage of data, we developed a new loss function called the robust quadratic loss, which is constructed by plugging robust covariance estimators in the $\ell_2$ risk function. Then we obtain the estimator $\widehat\bTheta$ by minimizing this new robust quadratic loss plus nuclear-norm penalty. By tailoring the analysis of \cite{NRW12} to this new loss, we can establish statistical rates in estimating the matrix $\bTheta^*$ that are the same as those in \cite{NRW12} for the sub-Gaussian distributions, while allowing the noise and design to have much heavier tails.  This result is very generic and applicable to all four specific afforementioned examples.
		
		Our robust approach is particularly simple:  it truncates or shrinks appropriately the response variables, depending on whether the responses are univariate or multivariate.    Under the setting of sub-Gaussian design, unusually large responses are very likely to be due to the outliers of noises. This explains why we need to truncate the responses when we have light-tailed covariates. Under the setting of heavy-tailed covariates, we need to truncate the designs as well. It turns out that appropriate truncation does not induce significant bias or hurt the restricted strong convexity of the loss function. With these data robustfications, we can then apply penalized least-squares method to recover sparse vectors or low-rank matrices.  Under only bounded moment conditions for either noise or covariates, our robust estimator achieves the same statistical error rate as that under the case of the sub-Gaussian design and noise. The crucial component in our analysis is the sharp spectral-norm convergence rate of robust covariance matrices based on data shrinkage. Of course, other robustifications of estimated covariance matrices, such as the RA-covariance estimation in \cite{FLW16}, are also possible to enjoy similar statistical error rates, but we will only focus on the shrinkage method, as it is easier to analyze and always semi-positive definite.
		

		
		It is worth emphasis that the successful application of the shrinkage sample covariance in multi-tasking regression inspires us to also study its statistical error in covariance estimation. It turns out that as long as the random samples $\{\bx_i\in \RR^d\}_{i=1}^N$ have bounded fourth moment in the sense that $\sup_{\bv\in \cS^{d-1}} \E (\bv^T\bx_i)^4 \le R \allowbreak <\infty$, where $\cS^{d-1}$ is the $d$-dimensional unit sphere, our $\ell_4$-norm shrinkage sample covariance $\widetilde\bSigma_n$ achieves the statistical error rate of order $O_P(\sqrt{d\log d/n})$ in terms of the spectral norm. This rate is the same, up to a logarithmic term, as that of the standard sample covariance matrix $\overline\bSigma_n$ with sub-Gaussian samples under the low-dimensional regime. Under the high-dimensional regime, $\widetilde\bSigma$ even outperforms $\overline\bSigma_n$ for sub-Gaussian random samples, since now the error rate of $\overline\bSigma_n$ deteriorates to $O_P(d/n)$ while the error rate of $\widetilde \bSigma$ is still $O_P(\sqrt{d\log d/n})$. This means even with light-tailed data, standard sample covariance can be inadmissible in terms of convergence rate when dimension is high. Therefore, shrinkage not only overcomes heavy-tailed corruption, but also mitigates curse of dimensionality. In terms of the elementwise max-norm, it is not hard to show that appropriate elementwise truncation of the data delivers a sample covariance with statistical error rate of order $O_P(\sqrt{\log d/n})$. This estimator can further be regularized if the true covariance has sparsity and other structure.  See, for example, \cite{MBu06}, \cite{BLe08}, \cite{LFa09}, \cite{CLi11}, \cite{CZh12}, \cite{FLM13}, among others.
		
		The  paper is organized as follows. In Section 2, we introduce the trace regression model and its four well-known examples: the linear model, matrix compressed sensing, matrix completion and multi-tasking regression. Then we develop the generalized $\ell_2$ loss, the truncated and shrinkage sample covariance and corresponding M-estimators. In Section 3, we present our main theoretical results. We first demonstrate through Theorem \ref{thm:1} the conditions on the robust covariance inputs to ensure the statistical error rate of the M-estimator. Then we apply this theorem to all the four specific aforementioned problems and derive explicitly the statistical error rate of our M-estimators. Section 4 derives the statistical error of the shrinkage covariance estimator in terms of the spectral norm. Finally we present simulation studies in Section 5, which demonstrate the advantage of our robust estimator over the standard one. The associated optimization algorithms are also discussed there. All the proofs are relegated to the Appendix A.

	\section{Models and methodology}
	
	We first collect the general notation before formulating the model and methodology.
	
	\subsection{Generic Notations}
	
	We follow the common convention of using boldface letters for vectors and matrices and using regular letters for scalars. For a vector $\bx$, define $\|\bx\|_q$ to be its $\ell_q$ norm; specifically, $\lonenorm{\bx}$ and $\ltwonorm{\bx}$ denote the $\ell_1$ norm and $\ell_2$ norm of $\bx$ respectively. We use $\RR^{d_1d_2}$ to denote the space of $d_1d_2$-dimensional real vectors, and use $\RR^{d_1\times d_2}$ to denote the space of $d_1$-by-$d_2$ real matrices. For a matrix $\bX\in \RR^{d_1\times d_2}$, define $\opnorm{\bX}$,  $\nnorm{\bX}$,  $\fnorm{\bX}$ and $\|\bX\|_{\max}$  to be its operator norm, nuclear norm, Frobenius norm and elementwise max norm respectively. We use $\vec(\bX)$ to denote vectorized version of $\bX$, i.e., $\vec(\bX)=(\bX_1^T, \bX_2^T, ..., \bX_{d_2}^T)^T$, where $\bX_j$ is the $j^{th}$ column of $\bX$. Conversely, for a vector $\bx \in \RR^{d_1 d_2}$, we use $\mat(\bx)$ to denote the $d_1$-by-$d_2$ matrix constructed by $\bx$, where $(x_{(j-1)d_1+1}, ..., x_{jd_1})^T$ is the $j^{\text{th}}$ column of $\mat(\bx)$.
	For any two matrices $\bA, \bB\in \Rdd$, define the inner product $\langle \bA, \bB \rangle:=\Tr (\bA^T\bB)$ where $\Tr$ is the trace operator.
	We denote $\diag(\bM_1, \cdots, \bM_n)$ to be the block diagonal matrix  with  the  diagonal  blocks  as $\bM_1, \cdots, \bM_n$.
	For two Hilbert spaces $\bA$ and $\bB$, we write $\bA \perp \bB$ if $\bA$ and $\bB$ are orthogonal to each other. 
	For two scalar series $\{a_n\}_{n=1}^{\infty}$ and $\{b_n\}_{n=1}^\infty$, we say $a_n\asymp b_n$ if there exist constants $0<c_1<c_2$ such that $c_1a_n \le b_n \le c_2a_n$ for $1\le n <\infty$.
	For a random variable $X$, define its sub-Gaussian norm $\|X\|_{\psi_2}:=\sup_{p\ge 1} (\E |X|^p)^{\frac{1}{p}}/\sqrt{p}$ and its sub-exponential norm $\|X\|_{\psi_1}:=\sup_{p\ge 1} (\E |X|^p)^{\frac{1}{p}}/p$. For a random vector $\bx\in \RR^d$, we define its sub-Gaussian norm $\|\bx\|_{\psi_2}:=\sup_{\bv\in \cS^{d-1}} \|\bv^T \bx\|_{\psi_2}$ and sub-exponential norm $\|\bx\|_{\psi_1}:=\sup_{\bv\in \cS^{d-1}} \|\bv^T \bx\|_{\psi_1}$. Given $x,y\in \RR$, we denote $\max(x, y)$ and $\min(x, y)$ by $x \vee y$ and $x\wedge y$ respectively. Let $\be_j$ be the unit vector with the $j^{th}$ element $1$ and other elements $0$.

	\subsection{Trace Regression}
	
	In this paper, we consider the trace regression, a general model that encompasses the linear regression, compressed sensing, matrix completion, multi-tasking regression, etc. Suppose we have $N$ matrices $\{\bX_i\in \Rdd\}_{i=1}^N$ and responses $\{Y_i \in \RR\}_{i=1}^N$.  We say $\{(Y_i, \bX_i)\}_{i=1}^N$ follow the trace regression model if
	\begin{equation}
		\label{eq:tcreg}
		Y_i=\inn{\bX_i, \bTheta^*} +\epsilon_i,
	\end{equation}
	where $\bTheta^*\in \Rdd$ is the true coefficient matrix, $\E \bX_i=\bzero$ and $\{\epsilon_i\}_{i=1}^N$ are independent noises satisfying $\E(\epsilon_i| \bX_i)=0$. Note that here we do not assume $\{\bX_i\}_{i=1}^N$ are independent to each other nor assume $\epsilon_i$ is independent to $\bX_i$. Model (\ref{eq:tcreg}) includes the following specific cases.
	\begin{itemize}
		\item {\bf Linear regression}: $d_1=d_2=d$, and $\{\bX_i\}_{i=1}^N$ and $\bTheta^*$ are diagonal. Let $\bx_i$ and $\btheta^*$ denote the vectors of diagonal elements of $\bX_i$ and $\bTheta^*$ respectively in the context of linear regerssion, i.e., $\bx_i = \diag(\bX_i)$ and $\btheta^* = \diag(\bTheta^*)$. Then, \eqref{eq:tcreg} reduces to familiar linear model:  $Y_i = \bx_i^T \btheta + \varepsilon_i$. Having a low-rank $\bTheta^*$ is then equivalent to having a sparse $\btheta^*$.

		\item {\bf Compressed sensing}: For matrix compressed sensing,
entries of $\bX_i$ jointly follow the Gaussian distribution or other ensembles.  For vector compressed sensing, we can take $\bX$ and $\bTheta^*$ as diagonal matrices.

		\item {\bf Matrix completion}: $\bX_i$ is a singleton, i.e., $\bX_i=\be_{j(i)}\be_{k(i)}^T$ for $1 \le j(i) \le d_1$ and $1 \le k(i) \le d_2$.  In other words, a random entry of the matrix $\bTheta$ is observed along with noise for each sample.

		\item {\bf Multi-tasking regression}:
The multi-tasking (reduced-rank) regression model is
	\begin{equation}
		\label{eq:2.2}
		\by_j=\bTheta^{*T}\bx_j+\bepsilon_j,\quad j=1,\cdots,n,
	\end{equation}
	where $\bx_j\in \RR^{d_1}$ is the covariate vector, $\by_j\in \RR^{d_2}$ is the response vector, $\bTheta^*\in \RR^{d_1\times d_2}$ is the coefficient matrix and $\bepsilon_j\in \RR^{d_2}$ is the noise with each entry independent to each other. See, for example,  \cite{kim2010tree} and \cite{velu2013multivariate}.  Each sample $(\by_j, \bx_j)$ consists of $d_2$ responses and is equivalent to $d_2$ data points in (\ref{eq:tcreg}), i.e., $\{(Y_{(j-1)d_2 + i} = y_{ji}, \bX_{(j-1)d_2 + i} = \bx_j\be_i^T)\}_{i=1}^{d_2}$. Therefore $n$ samples in (\ref{eq:2.2}) correspond to $N=nd_2$ observations in (\ref{eq:tcreg}).
\end{itemize}
	
	In this paper, we impose rank constraint on the coefficient matrix $\bTheta^*$. Rank constraint can be viewed as a generalized sparsity constraint for two-dimensional matrices. For linear regression, rank constraint is equivalent to the sparsity constraint since $\bTheta^*$ is diagonal. The rank constraint reduces the effective number of parameters in $\bTheta^*$ and arises frequently in many applications.
Consider the Netflix problem for instance, where $\bTheta^*_{ij}$ is the intrinsic score of film $j$ given by customer $i$ and we would like to recover the entire $\bTheta^*$ with only partial observations. Given that movies of similar types or qualities should receive similar scores from viewers, columns of $\bTheta^*$ should share colinearity, thus delivering a low-rank structure of $\bTheta^*$. The rationale of the model can also be understood from the celebrated factor model in finance and econometrics \citep{FY15}, which assumes that several market risk factors drive the returns of a large panel of stocks.  Consider $N \times T$ returns $\bY$ of $N$ stocks (like movies) over $T$ days (like viewers).  These financial returns are driven by $K$ factors $\bF$ ($K\times T$ matrix, representing $K$ risk factors realized on $T$ days) with a loading matrix $\bB$ ($N\times K$ matrix), where $K$ is much smaller than $N$ or $T$.  The factor model admits the following form:
$$
    \bY = \bB\bF + \bE
$$
where $\bE$ is idiosyncratic noise. Since $\bB\bF$ has a small rank $K$, $\bB\bF$ can be regarded as the low-rank matrix $\bTheta^*$ in the matrix completion problem. If all movies were rated by all viewers in the Netflix problem, the ratings should also be modeled as a low-rank matrix plus noise, namely, there should be several latent factors that drive ratings of movies.  The major challenge of the matrix completion problem is that there are many missing entries.

Being exactly low-rank is still too stringent to model the real-world situations. We instead consider $\bTheta^*$ satisfying
	\beq
		\label{eq:2.3}
		\cB_q(\bTheta^*):=\sum\limits_{i=1}^{d_1\wedge d_2} \sigma_i(\bTheta^*)^q \le \rho,
	\eeq
	where $0\le q\le 1$. Note that when $q=0$, the constraint \eqref{eq:2.3} is an exact rank constraint. Restriction on $\cB_q(\bTheta^*)$ ensures that the singular values decay fast enough; it is more general and natural than the exact low-rank assumption. In the analysis, we can allow $\rho$ to grow with dimensionality and sample size.
	
A popular method for estimating $\bTheta^*$ is the penalized empirical loss that solves $\widehat\bTheta\in\argmin_{\bTheta\in \cS} \cL(\bTheta)+\lambda_N\cP(\bTheta)$, where $\cS$ is a convex set in $\Rdd$, $\cL(\bTheta)$ is a loss function, $\lambda_N$ is a tuning parameter and $\cP(\bTheta)$ is a rank penalization function.  Most of the previous work, e.g., \cite{KLT11} and \cite{NWa11}, chose $\cL(\bTheta)=\sum_{1\le i\le N} (Y_i-\inn{\bTheta, \bX_i})^2$ and $\cP(\bTheta)=\nnorm{\bTheta}$, and derived the rate for $\fnorm{\widehat\bTheta-\bTheta^*}$ under the assumption of sub-Gaussian or sub-exponential noise. However, the $\ell_2$ loss is sensitive to outliers and is unable to handle the data with moderately heavy or heavy tails.
	
	\subsection{Robustifying $\ell_2$ Loss}
	
	We aim to accomodate heavy-tailed noise and design for the nearly low-rank matrix recovery by robustifying the traditional $\ell_2$ loss. We first notice that the $\ell_2$ risk can be expressed as
	\beq
		\begin{aligned}
			R(\bTheta) & =\E\cL(\bTheta)=\E(Y_i-\inn{\bTheta, \bX_i})^2\\
			& =\E Y_i^2-2\inn{\bTheta, \E Y_i\bX_i}+\vec(\bTheta)^T\E\bigl(\vec(\bX_i)\vec(\bX_i)^T\bigr)\vec(\bTheta) \\
			& \equiv \E Y_i^2-2\inn{\bTheta, \bSigma_{Y\bX}}+\vec(\bTheta)^T\bSigma_{\bX\bX}\vec(\bTheta).
		\end{aligned}
	\eeq
Ignoring $\E Y_i^2$,  if we substitute $\bSigma_{Y\bX}$ and $\bSigma_{\bX\bX}$ by their corresponding sample covariances, we obtain the empirical $\ell_2$ loss. This inspires us to define a generalized $\ell_2$ loss via
	\beq
		\label{eq:2.5}
		\cL(\bTheta)=-\inn{\widehat\bSigma_{Y\bX}, \bTheta}+\frac{1}{2}\vec(\bTheta)^T\widehat\bSigma_{\bX\bX}\vec(\bTheta),
	\eeq
	where $\widehat\bSigma_{Y\bX}$ and $\widehat\bSigma_{\bX\bX}$ are estimators of $\E Y_i\bX_i$ and $\E \vec(\bX_i)\vec(\bX_i)^T$ respectively.

In this paper, we study the following M-estimator of $\bTheta^*$ with the generalized $\ell_2$ loss:
	\beq
		\label{eq:2.6}
		\widehat\bTheta\in\argmin_{\bTheta\in \cS} -\inn{\widehat\bSigma_{Y\bX}, \bTheta}+\frac{1}{2}\vec(\bTheta)^T\widehat\bSigma_{\bX\bX}\vec(\bTheta)+\lambda_N\nnorm{\bTheta},
	\eeq
	where $\cS$ is a convex set in $\RR^{d_1\times d_2}$.  To handle heavy-tailed noise and design, we need to employ robust estimators $\widehat \bSigma_{Y\bX}$ and $\widehat\bSigma_{\bX\bX}$. For ease of presentation, we always first consider the case where the design is sub-Gaussian and the response is heavy-tailed, and then further allow the design to have heavy-tailed distribution if it is appropriate for the specific problem setup.
	
	
We now introduce the robust covariance estimators to be used in \eqref{eq:2.6} by the principle of truncation, or more generally shrinkage. The intuition is that shrinkage reduces sensitivity of the estimator to the heavy-tailed corruption. However, shrinkage induces bias. Our theories revolve around finding appropriate shrinkage level so as to ensure the induced bias is not too large and the final statistical error rate is sharp. Different problem setups have different forms of $\widehat\bSigma_{Y\bX}$ and $\widehat\bSigma_{\bX\bX}$, but the principle of shrinkage of data is universal. For the linear regression, matrix compressed sensing and matrix completion, in which the response is univariate,  $\widehat\bSigma_{Y\bX}$ and $\widehat\bSigma_{\bX\bX}$ take the following forms:
\beq
	\label{eq:2.7}
	\widehat\bSigma_{Y\bX}=\widehat\bSigma_{\widetilde Y \widetilde \bX}= \frac{1}{N}\sum\limits_{i=1}^N \widetilde Y_i \widetilde\bX_i \quad \text{and} \quad \widehat\bSigma_{\bX\bX}=\widehat\bSigma_{\widetilde \bX\widetilde \bX}= \frac{1}{N}\sum\limits_{i=1}^N \vec(\widetilde\bX_i) \vec(\widetilde\bX_i)^T,
\eeq
where tilde notation means truncated versions of the random variables if they have heavy tails and equals the original random variables (truncation threshold is infinite) if they have light tails. Note that this  construction of generalized quadratic loss is equivalent to truncating the data first and then employing the usual quadratic loss.

For the multi-tasking regression, similar idea continues to apply. However, writing \eqref{eq:2.2} in the general form of \eqref{eq:tcreg} requires adaptation of more complicated notation.   We choose
\beq
\label{eq:2.8}
	\begin{aligned}
		& \widehat\bSigma_{Y\bX}= \frac{1}{N}\sum\limits_{i=1}^n  \sum\limits_{j=1}^{d_2} \widetilde Y_{ij}\widetilde \bx_i\be_j^T= \frac{1}{d_2}\hat\bSigma_{\widetilde \bx \widetilde \by} \quad \text{and} \\
		& \widehat\bSigma_{ \bX\bX} = \frac{1}{N}\sum\limits_{i=1}^n \sum\limits_{j=1}^{d_2} \vec(\widetilde \bx_i\be_j^T)\vec(\widetilde \bx_i\be_j^T)^T= \frac{1}{d_2} \diag(\;\underbrace{\widehat\bSigma_{\widetilde\bx \widetilde \bx}, \cdots, \widehat\bSigma_{\widetilde \bx\widetilde \bx}}_{d_2}\;)\,,
	\end{aligned}
\eeq
where
\[
	\widehat\bSigma_{\widetilde\bx \widetilde \by}=\frac{1}{n}\sum\limits_{i=1}^n \widetilde \bx_i \widetilde\by_i^T \quad \text{and} \quad \widehat\bSigma_{\widetilde \bx\widetilde \bx}= \frac{1}{n}\sum\limits_{i=1}^n \widetilde\bx_i \widetilde\bx_i^T
\]
and $\widetilde \by_i$ and $\widetilde \bx_i$ are again transformed versions of $\by_i$ and $\bx_i$. The tilde means shrinkage for heavy-tailed variables and means identity mapping (no shrinkage) for light-tailed variables. The factor $d_2^{-1}$ is due to the fact that $n$ independent samples under model (\ref{eq:2.2}) are treated as $nd_2$ samples in (\ref{eq:tcreg}).
As we shall see, under only bounded moment assumptions of the design and noise, the generalized $\ell_2$ loss equipped with the proposed robust covariance estimators yields a sharp $\cM$-estimator $\widehat\bTheta$, whose statistical error rates match those established in \cite{NWa11} and \cite{NWa12} under the setting of sub-Gaussian design and noise.

	\section{Main results}
	
Our goal is to derive the statistical error rate of $\widehat\bTheta$ defined by ($\ref{eq:2.6}$). In the theoretical results of this section, we always assume $d_1, d_2\ge 2$ and $\rho>1$ in \eqref{eq:2.3}. We first present the following general theorem that gives the estimation error $\fnorm{\widehat\bTheta-\bTheta^*}$.
	\begin{thm}
		\label{thm:1}
		Define $\widehat\bDelta= \widehat\bTheta- \bTheta^*$, where $\bTheta^*$ satisfies $\cB_q(\bTheta^*) \le \rho$. Suppose $\vec(\widehat\bDelta)^T\widehat\bSigma_{\bX\bX}\allowbreak\vec(\widehat\bDelta) \ge \kappa_{\cL}\fnorm{\widehat\bDelta}^2$, where $\kappa_{\cL}$ is a positive constant that does not depend on $\widehat\bDelta$. Choose $\lambda_N\ge 2\opnorm{\widehat\bSigma_{Y\bX} - \mat(\widehat\bSigma_{\bX\bX}\vec(\bTheta^*))}$. Then we have for some constants $C_1$ and $C_2$,
		\[
			\fnorm{\widehat\bDelta}^2 \le C_1\rho\Bigl(\frac{\lambda_N}{\kappa_{\cL}}\Bigr)^{2-q} \quad \text{and} \quad \nnorm{\widehat\bDelta} \le C_2\rho\Bigl(\frac{\lambda_N}{\kappa_{\cL}}\Bigr)^{1-q}.
		\]
	\end{thm}
	
		
		First of all, the above result is deterministic and nonasymptotic. As we can see from the theorem above, the statistical performance of $\widehat\bTheta$ relies on the restricted eigenvalue (RE) property of $\widehat\bSigma_{\bX\bX}$, which was first studied by \cite{BRT09}. When the design is sub-Gaussian, we choose $\widehat\bSigma_{\bX\bX}$ to be the traditional sample covariance, whose RE property has been well established (e.g.,  \cite{RZh13}, \cite{NWa11} and \cite{NWa12}). We will specify these results when we need them in the sequel. When the design only satisfies bounded moment conditions, we choose $\widehat\bSigma_{\bX\bX}= \widehat\bSigma_{\widetilde\bX\widetilde\bX}$, i.e., the sample covariance of shrunk $\bX$. We show that with appropriate level of shrinkage, $\widehat\bSigma_{\widetilde\bX \widetilde\bX}$ still retains the RE property, thus satisfying the conditions of the theorem. 
		
		Secondly, the conclusion of the theorem says that $\fnorm{\widehat\bDelta}^2$ and $\nnorm{\widehat\bDelta}$ are proportional to $\lambda_N^{2-q}$ and $\lambda_N^{1-q}$ respectively, but we require $\lambda_N\ge \opnorm{\widehat\bSigma_{Y\bX}-  \mat(\widehat\bSigma_{\bX\bX}\vec(\bTheta^*)}$. This implies that the rate of $\opnorm{\widehat\bSigma_{Y\bX}-  \mat(\widehat\bSigma_{\bX\bX}\vec(\bTheta^*)}$ is crucial to the statistical error of $\widehat\bTheta$. In the following subsections, we will derive the rate of $\opnorm{\widehat\bSigma_{Y\bX}-  \mat(\widehat\bSigma_{\bX\bX}\vec(\bTheta^*)}$ for all the aforementioned specific problems with only bounded moment conditions on the response and design. Under such weak assumptions, we show that the proposed robust M-estimator possesses the same rates as those presented in \cite{NWa11, NWa12} with sub-Gaussian assumptions on the design and noise.
		
		

	\subsection{Linear Model}
	
	For the linear regression problem, since $\bTheta^*$ and $\{\bX_i\}_{i=1}^N$ are all $d\times d$ diagonal matrices, we denote the diagonals of $\bTheta^*$ and $\{\bX_i\}_{i=1}^N$ by $\btheta^*$ and $\{\bx_i\}_{i=1}^N$ respectively for ease of presentation. The optimization problem in \eqref{eq:2.6} reduces to
	\beq
		\label{eq:3.1}
		\widehat\btheta \in \argmin_{\btheta\in \RR^d} -\widehat\bSigma_{Y\bx}^T\btheta+ \frac{1}{2}\btheta^T\widehat\bSigma_{\bx\bx}\btheta + \lambda_N\|\btheta\|_1\,,
	\eeq
	where $\widehat\bSigma_{Y\bx}= \widehat\bSigma_{\widetilde Y \widetilde \bx}= N^{-1}\sum\nolimits_{i=1}^N \widetilde Y_i \widetilde\bx_i$, $\widehat\bSigma_{\bx\bx}=\widehat\bSigma_{\widetilde \bx\widetilde \bx}= N^{-1}\sum\nolimits_{i=1}^N \widetilde\bx_i \widetilde\bx_i^T$. When the design is sub-Gaussian, we only need to truncate the response. Therefore, we choose $\widetilde Y_i=\widetilde Y_i(\tau) = sgn(Y_i)(|Y_i| \wedge \tau)$ and $\widetilde \bx_i= \bx_i$, for some threshold $\tau$. When the design is heavy-tailed, we choose $\widetilde Y_i(\tau) = sgn(Y_i)(|Y_i| \wedge \tau_1)$ and $\widetilde x_{ij}=sgn(x_{ij})(|x_{ij}|\wedge \tau_2)$, where $\tau_1$ and $\tau_2$ are both predetermined threshold values. To avoid redundancy, we will not repeat stating these choices in lemmas or theorems in this subsection.
	
	To establish the statistical error rate of $\widehat\btheta$ in \eqref{eq:3.1}, in the following lemma, we derive the rate of $ \opnorm{\widehat\bSigma_{Y\bx} - \mat(\widehat\bSigma_{\bX\bX}\vec(\bTheta^*))}$ in \eqref{eq:2.6} for the sub-Gaussian design and bounded-moment (polynomial tail) design respectively. Note here that
	$$\opnorm{\widehat\bSigma_{Y\bx} - \mat(\widehat\bSigma_{\bX\bX}\vec(\bTheta^*))} \allowbreak= \supnorm{\widehat\bSigma_{Y\bx}- \widehat\bSigma_{\bx\bx}\btheta^*}.$$
	
	\begin{lem}
		\label{lem1}
  Uniform convergence of cross covariance.
\begin{itemize}
\item [(a)] {\bf Sub-Gaussian design}. Consider the following conditions:
  		\begin{itemize} \itemsep -0.05in
  			\item [(C1)] $\{\bx_i\}_{i=1}^N$ are i.i.d. sub-Gaussian vectors with $\|\bx_i\|_{\psi_2}\le \kappa_0< \infty$, $\E \bx_i=\bzero$ and $\lambda_{\min}(\E \bx_1\bx_1^T) \ge \kappa_{\cL}>0$;
  			\item [(C2)] $\forall i=1,...,N$, $\E |Y_i|^{2k} \le M< \infty$ for some $k>1$.
  		\end{itemize}
		Choose $\tau \asymp \sqrt{N/ \log d}$. For any $\delta>0$, there exists a constant $\gamma_1>0$ such that as long as $\log d/N< \gamma_1$, we have
		\begin{equation}
			P\Big( \supnorm{\widehat\bSigma_{Y\bx}(\tau)-\widehat\bSigma_{\bx\bx}\btheta^*}\ge \nu_1\sqrt{\frac{\delta\log d}{N}}\Big)\le 2d^{1-\delta},
		\end{equation}
		where $\nu_1$ is a universal constant.

\item [(b)] {\bf (Bounded moment design)}
		Consider instead the following set of conditions:
		\begin{itemize} \itemsep -0.05in
			\item [(C1')] $\lonenorm{\btheta^*} \le R<\infty$;
			\item [(C2')] $\E |x_{ij_1}x_{ij_2}|^{2}\le M<\infty$, $1 \le j_1, j_2 \le d$;
			\item [(C3')] $\forall i=1,..., N$, $\E |Y_i|^4\le M< \infty$.
		\end{itemize}
		Choose $\tau_1,\tau_2 \asymp (N/\log d)^{\frac{1}{4}}$. For any $\delta>0$, it holds that 
		\[
			P\Bigl(\|\widehat\bSigma_{Y \bx}(\tau_1, \tau_2) - \widehat\bSigma_{\bx \bx}(\tau_2)\btheta^*\|_{\max}>\nu_2\sqrt{\frac{\delta\log d}{N}} \Bigr)\le 2d^{1-\delta},
		\]
		where $\nu_2$ is a universal constant.
\end{itemize}
 	\end{lem}
	
	\begin{rem}
		If we choose $\widehat\bSigma_{Y\bx}$ and $\widehat\bSigma_{\bx\bx}$ to be the sample covariance, i.e., $\widehat\bSigma_{Y\bx}= \overline\bSigma_{Y\bx}= \frac{1}{N}\sum\nolimits_{i=1}^N Y_i\bx_i$ and $\widehat\bSigma_{\bx\bx}= \overline\bSigma_{\bx\bx}= \frac{1}{N}\sum\nolimits_{i=1}^N \bx_i\bx_i^T$, Corollary 2 of \cite{NRW12} showed that under the sub-Gaussian noise and design,
		\[
			\supnorm{\overline\bSigma_{Y\bx}- \overline\bSigma_{\bx\bx}\btheta^*}=O_P\Bigl(\sqrt{\frac{\log d}{N}} \Bigr).
		\]
		This is the same rate as what we achieved under only the bounded moment conditions on response and design.
	\end{rem}
	
	Next we establish the restricted strong convexity of the proposed robust $\ell_2$ loss.
	\begin{lem}
		\label{lem:2}
Restricted strong convexity.
\begin{itemize}
\item [(a)] {\bf Sub-Gaussian design}. Under Condition (C1) of Lemma \ref{lem1}, it holds for certain constants $C_1$, $C_2$ and any $\eta_1>1$ that 		
		\beq
			\label{eq:3.3}
			P\Bigl(\bv^T\widehat \bSigma_{\bx\bx}\bv\ge \frac{1}{2}\bv^T\bSigma_{\bx\bx}\bv-\frac{C_1 \eta_1\log d}{N}\lonenorm{\bv}^2\,,  \;\; \forall \bv \in \RR^d\Bigr)\ge 1-\frac{d^{1-\eta_1}}{3}- 2d\exp(-C_2N).
		\eeq

\item [(b)] {\bf Bounded moment design}. If $\bx_i$ satisfies Condition (C2') of Lemma \ref{lem1}, then it holds for some constant $C_3>0$ and any $\eta_2>2$ that
		\beq
			\label{eq:3.4}
			P\Bigl(\bv^T\widehat\bSigma_{\bx \bx}(\tau_2)\bv \ge \bv^T\bSigma_{\bx\bx}\bv - C_3\eta_2\sqrt{\frac{\log d}{N}}\lonenorm{\bv}^2\,, \;\; \forall \bv\in \RR^d \Bigr) \le d^{2-\eta_2},
		\eeq
		as long as $\tau_2\asymp (N/\log d)^{\frac{1}{4}}$.
\end{itemize}
	\end{lem}
	
	\begin{rem}
		Comparing the results we get for sub-Gaussian design and heavy-tailed design, we can find that the coefficients before $\|\bv\|_1^2$ are different. Under the sub-Gaussian design, that coefficient is of the order $\log d/N$, while under the heavy-tailed design, the coefficient is of order $\sqrt{\log d/N}$. This difference will lead to different scaling requirements for $N, d$ and $\rho$ in the sequel. As we shall see, the heavy-tailed design requires stronger scaling conditions to retain the same statistical error rate as the sub-Gaussian design for the linear model.
	\end{rem}
	
	Finally we derive the statistical error rate of $\widehat\btheta$ as defined in \eqref{eq:3.1}.
	
\begin{thm}
		\label{thm:lm}
		Assume $\sum\limits_{i=1}^d |\theta^*_i|^q\le \rho$, where $0\le q\le 1$.
\begin{itemize}
\item [(a)]{\bf Sub-Gaussian design}: Suppose Conditions (C1) and (C2) in Lemma \ref{lem1} hold. For any $\delta>0$, choose $\tau\asymp \sqrt{N/\log d}$ and $\lambda_N = 2\nu_1 \sqrt{\delta\log d/N}$, where $\nu_1$ and $\delta$ are the same as in part (a) of Lemma \ref{lem1}. There exist positive constants $\{C_i\}_{i=1}^3$ such that as long as $\rho (\log d/N)^{1-\frac{q}{2}} \le C_1$, it holds that
		\[
			P\Big(\ltwonorm{\widehat\btheta(\tau, \lambda_N)-\btheta^*}^2 >C_2\rho\Bigl(\frac{\delta\log d}{N}\Bigr)^{1-\frac{q}{2}}\Big) \le 3d^{1-\delta}
		\]		
		and
		\[
			P\Big(\lonenorm{\widehat\btheta(\tau, \lambda_N)-\btheta^*} >C_3\rho\Bigl(\frac{\delta\log d}{N}\Bigr)^{\frac{1-q}{2}}\Big) \le 3d^{1-\delta}.
		\]
\item [(b)]{\bf Bounded moment design}: For any $\delta>0$ choose $\tau_1, \tau_2\asymp (N/\log d)^{\frac{1}{4}}$ and $\lambda_N=2\nu_2 \sqrt{\delta\log d/N}$, where $\nu_2$ and $\delta$ are the same as in part (b) of Lemma \ref{lem1}. Under Conditions (C1'), (C2') and (C3'), there exist constants $\{C_i\}_{i=4}^6$ such that as long as $\rho(\log d/N)^{\frac{1-q}{2}} \le C_4$, we have
		\[
			P\Big(\ltwonorm{\widehat\btheta(\tau_1, \tau_2, \lambda_N)- \btheta^*}^2> C_5\rho\Bigl(\frac{\delta\log d}{N}\Bigr)^{1-\frac{q}{2}} \Big)\le 3d^{1-\delta}
		\]
		and
		\[	
			P\Big(\lonenorm{\widehat\btheta(\tau_1, \tau_2, \lambda_N)- \btheta^*} > C_6\rho\Bigl(\frac{\delta\log d}{N}\Bigr)^{\frac{1-q}{2}} \Big)\le 3d^{1-\delta}.
		\]
\end{itemize}	
		
\end{thm}
	
	\begin{rem}
		Under both sub-Gaussian and heavy-tailed design, our proposed $\widehat\btheta$ achieves the minimax optimal rate of $\ell_2$ norm established by \cite{RWY11}. However, the difference lies in the scaling requirement on $N$, $d$ and $\rho$. For sub-Gaussian
design, we require $\rho (\log d/N)^{1-\frac{q}{2}} \le C_1$, whereas for heavy-tailed design we need $\rho (\log d/N)^{\frac{1-q}{2}} \le C_4$.
Under the high-dimensional regime that $d\gg N\gg \log d$, the former is weaker. Therefore, heavy-tailed design requires stronger scaling than sub-Gaussian design to achieve the optimal statistical rate.
	\end{rem}
	

	\subsection{Matrix Compressed Sensing}
	
	For the matrix compressed sensing problem, since the design is chosen by users, we consider solely the most popular design: the Gaussian design. We thus keep the original design matrix and only truncate the response. In \eqref{eq:2.7}, choose $\widetilde Y_i= sgn(Y_i)(|Y_i| \wedge \tau)$ and $\widetilde\bX_i= \bX_i$, then we have
\beq \label{eq3.5}
  \widehat\bSigma_{Y\bX}= \widehat\bSigma_{\widetilde Y\widetilde \bX}(\tau)= \frac{1}{N}\sum\limits_{i=1}^N sgn(Y_i)(|Y_i| \wedge \tau)\bX_i
  \quad \mbox{and} \quad
  \widehat\bSigma_{\bX\bX}= \frac{1}{N}\sum\limits_{i=1}^N \vec(\bX_i)\vec(\bX_i)^T.
\eeq
The following lemma quantifies the convergence rate of $\opnorm{\widehat\bSigma_{Y\bX} - \mat(\widehat\bSigma_{\bX\bX}\vec(\bTheta^*))}$. Note that here $\widehat\bSigma_{Y\bX} - \mat(\widehat\bSigma_{\bX\bX}\vec(\bTheta^*))=\widehat\bSigma_{ Y \bX}(\tau) - \frac{1}{N}\sum\limits_{i=1}^N \inn{\bX_i, \bTheta^*}\bX_i $.
	
	\begin{lem}
		\label{lem:3}
		Consider the following conditions:
  		\begin{itemize} \itemsep -0.05in
  			\item [(C1)] $\{\vec(\bX_i)\}_{i=1}^N$ are i.i.d. sub-Gaussian vectors with $\|\vec(\bX_i)\|_{\psi_2}\le \kappa_0< \infty$, $\E \bX_i=\bzero$ and $\lambda_{\min}(\E \vec(\bX_i)\vec(\bX_i)^T) \ge \kappa_{\cL}>0$.
  			\item [(C2)] $\forall i=1,...,N$, $\E |Y_i|^{2k} \le M< \infty$ for some $k>1$.		
  		\end{itemize}
		There exists a constant $\gamma>0$ such that as long as $(d_1+d_2)/N< \gamma$, it holds that 
		\begin{equation}
			P\Big(\opnorm{\widehat\bSigma_{ Y \bX}(\tau) - \frac{1}{N}\sum\limits_{i=1}^N \inn{\bX_i, \bTheta^*}\bX_i} \ge \nu\sqrt{\frac{d_1+d_2}{N}}\Big)\le \eta\exp(-(d_1+d_2)),
		\end{equation}
		where $\tau \asymp \sqrt{N/(d_1+d_2)}$ and $\nu$ and $\eta$ are constants.
	\end{lem}
	
	\begin{rem}
	For the sample covariance $\overline\bSigma_{Y\bX}=\frac{1}{N} \sum\nolimits_{i=1}^N Y_i\bX_i$,  \cite{NWa11} showed that when the noise and design are sub-Gaussian,
		\[
			 \opnorm{\overline\bSigma_{Y\bX} - \frac{1}{N}\sum\nolimits_{i=1}^N \inn{\bX_i, \bTheta^*}\bX_i}=\opnorm{\frac{1}{N}\sum\nolimits_{i=1}^N \epsilon_i\bX_i} = O_P(\sqrt{(d_1+d_2)/N}).
		\]	
		Lemma \ref{lem:3} shows that $\widehat\bSigma_{Y\bX}(\tau)$ achieves the same rate for response with just bounded moments.
	\end{rem}
	
	The following theorem gives the statistical error rate of $\widehat\bTheta$ in \eqref{eq:2.6}.
	
	\begin{thm}
		\label{thm:cs}
		Suppose Conditions (C1) and (C2) in Lemma \ref{lem:3} hold and $\cB_q(\bTheta^*)\le \rho$. We further assume that $\vec(\bX_i)$ is Gaussian. Choose $\tau \asymp \sqrt{N/(d_1+d_2)}$ and $\lambda_N=2\nu\sqrt{(d_1+d_2)/N}$, where $\nu$ is the same as in Lemma \ref{lem:3}. There exist constants $\{C_i\}_{i=1}^4$ such that once $\rho \bigl((d_1+d_2)/N\bigr)^{1-\frac{q}{2}} \le C_1$, we have
		\[
			P \Big(\fnorm{\widehat\bTheta(\tau, \lambda_N)-\bTheta^*}^2 \ge C_2\rho\Bigl(\frac{d_1+d_2}{N}\Bigr)^{1-\frac{q}{2}} \Big) \le \eta\exp(-(d_1+d_2))+2\exp(-N/32)\,,
		\]		
		and
		\[
			P \Big(\nnorm{\widehat\bTheta(\tau, \lambda_N)-\bTheta^*} \ge C_3\rho\Bigl(\frac{d_1+d_2}{N}\Bigr)^{\frac{1-q}{2}} \Big) \le \eta\exp(-(d_1+d_2))+2\exp(-N/32)\,,
		\]
		where $\eta$ is the same constant as in Lemma \ref{lem:3}.
	\end{thm}

	\begin{rem}	
		The Frobenius norm rate here is the same rate as established under sub-Gaussian noise in \cite{NWa11}.  When $q=0$, $\rho$ is the upper bound of the rank of $\bTheta^*$ and the rate of convergence depends only on $\rho (d_1+d_2) \asymp \rho (d_1 \vee d_2)$, the effective number of independent parameters in $\bTheta^*$, rather than the ambiem number of parameters $d_1*d_2$
	\end{rem}
			
	\subsection{Matrix Completion}
	
	In this section, we consider the matrix completion problem with heavy-tailed noises. Under a conventional setting, $\bX_i$ is a singleton, $\|\bTheta^*\|_{\max}=O(1)$ and $\fnorm{\bTheta^*}=O(\sqrt{d_1d_2})$. If we rescale the original model as
	\[
		Y_i=\inn{\bX_i, \bTheta^*}+\epsilon_i= \inn{\sqrt{d_1d_2}\bX_i, \bTheta^*/\sqrt{d_1d_2}}+\epsilon_i
	\]
	and treat $\sqrt{d_1d_2}\bX_i$ as the new design $\check\bX_i$ and $\bTheta^*/\sqrt{d_1d_2}$ as the new coefficient matrix $\check\bTheta^*$, then $\fnorm{\check\bX_i}=O(\sqrt{d_1d_2})$ and $\fnorm{\check\bTheta^*}=O(1)$. Therefore, by rescaling, we can assume without loss of generality that $\bTheta^*$ satisfies $\fnorm{\bTheta^*}\le 1$ and $\bX_i$ is uniformly sampled from $\{\sqrt{d_1d_2}\cdot \be_j\be_k^T\}_{1\le j\le d_1, 1\le k\le d_2}$.
	
	 For the matrix completion problem, in order to achieve consistent estimation, we require the true coefficient matrix $\bTheta^*$ not to be overly spiky, i.e., $\|\bTheta^*\|_{\max}\le R\fnorm{\bTheta^*}/\sqrt{d_1d_2}\le R/\sqrt{d_1d_2}$. We put a similar constraint in seeking the corresponding M-estimator:
	\beq
		\label{eq:3.6}
		\widehat\bTheta \in \argmin_{ \|{\bTheta}\|_{\max}\le R/\sqrt{d_1d_2}} -\inn{\widehat\bSigma_{Y\bX}(\tau), \bTheta}+\frac{1}{2}\vec(\bTheta)^T\widehat\bSigma_{\bX\bX}\vec(\bTheta)+\lambda_N\nnorm{\bTheta}.
	\eeq	
	This spikiness condition is proposed by \cite{NWa12} and it is required by the matrix completion problem per se instead of our robust estimation.
	
	To derive robust estimation in matrix completion problem, we choose $\widetilde Y_i= sgn(Y_i)(|Y_i| \wedge \tau)$ and $\widetilde\bX_i= \bX_i$ in \eqref{eq:2.7}. Then, $\widehat\bSigma_{Y\bX}$ and $\widehat\bSigma_{\bX\bX}$ are given by \eqref{eq3.5}. Note that the design $\bX_i$ here takes the singleton form, which leads to different scaling and consistency rates from the setting of matrix compressed sensing.


	\begin{lem}
		\label{lem:4}
		Under the following conditions:
		\begin{itemize} \itemsep -0.05in
			\item [(C1)] $\|\bTheta^*\|_{F}\le 1$ and $\|\bTheta^*\|_{\max}\le R/\sqrt{d_1d_2}$, where $0<R<\infty$;
			\item [(C2)] $\bX_i$ is uniformly sampled from $\{\sqrt{d_1d_2}\cdot \be_j\be_k^T\}_{1\le j\le d_1, 1\le k\le d_2}$;
			\item [(C3)] $\forall i=1,...,N$, $\E (\E (\epsilon_i^2| \bX_i))^k \le M< \infty$, where $k>1$;	
		\end{itemize}
		there exists a constant $\gamma>0$ such that for any $\delta>0$, as long as $(d_1\vee d_2)\log(d_1+d_2)/N <\gamma$,
		\beq
			P\Big( \opnorm{\widehat\bSigma_{ Y\bX}(\tau) - \frac{1}{N}\sum\limits_{i=1}^N \inn{\bX_i, \bTheta^*}\bX_i}>\nu\sqrt{\frac{\delta(d_1\vee d_2)\log (d_1+d_2)}{N}}\Big)\le 2(d_1+d_2)^{1-\delta},
		\eeq
		where $\tau\asymp(\sqrt{N/{((d_1\vee d_2)\log (d_1+d_2))}})$ and $\nu$ is a universal constant.
	\end{lem}
	
	\begin{rem}
		Again, for $\overline\bSigma_{Y\bX}= \frac{1}{N}\sum\nolimits_{i=1}^N Y_i\bX_i$, \cite{NWa12} proved that $\opnorm{\overline\bSigma_{Y\bX} - \frac{1}{N}\sum\nolimits_{i=1}^N \inn{\bX_i, \bTheta^*}\bX_i}=O_P(\sqrt{(d_1+d_2)\log (d_1+d_2)/N})$ for sub-exponential noise. Compared with this result, Lemma \ref{lem:4} achieves the same rate of convergence.  By Jessen's inequality, condition (C3) is implied by $E \epsilon_i^{2k} \leq M < \infty$.
	\end{rem}

%
	
	Now we present the following theorem on the statistical error of $\widehat\bTheta$ defined in \eqref{eq:3.6}.
	
	\begin{thm}
		\label{thm:mc}
		Suppose that the conditions of Lemma \ref{lem:4} hold. Consider $\cB_q(\bTheta^*)\le \rho$ with $\|\bTheta^*\|_{\max}/\fnorm{\bTheta^*} \le R/\sqrt{d_1d_2}$. For any $\delta>0$, choose $$\tau\asymp\sqrt{N/{((d_1\vee d_2)\log (d_1+d_2))}})\quad \text{and}\quad \lambda_N=2\nu\sqrt{\delta(d_1\vee d_2)\log (d_1+d_2)/N}$$ and assume $(d_1\vee d_2)\log(d_1+d_2)/N <\gamma$, where $\nu$ and $\gamma$ are the same as in Lemma \ref{lem:4}. There exist universal constants $\{C_i\}_{i=1}^4$ such that with probability at least $1-2(d_1+d_2)^{1-\delta}- C_1\exp(-C_2(d_1+d_2))$ we have
		\[
			\fnorm{\widehat\bTheta(\tau, \lambda_N)-\bTheta^*}^2 \le C_3\max\Bigl\{\rho\Bigl(\frac{\delta R^2(d_1+d_2)\log(d_1+d_2)}{N}\Bigr)^{1-\frac{q}{2}}, \frac{R^2}{N}\Bigr\}
		\]
		and
		\[
			\nnorm{\widehat\bTheta(\tau, \lambda_N)-\bTheta^*} \le C_4\max\Bigl\{\rho\Bigl(\frac{\delta R^2(d_1+d_2)\log(d_1+d_2)}{N}\Bigr)^{\frac{1-q}{2}}, \bigl(\frac{\rho R^{2-2q}}{N^{1-q}}\bigr)^{\frac{1}{2-q}}\Bigr\},
		\]
		where $\widehat\bTheta$ is defined in \eqref{eq:3.6}.
	\end{thm}	

	\begin{rem}
		Theorem \ref{thm:mc} achieves the same statistical error rate of Frobenius norm as established in \cite{NWa12}, which also matches the information-theoretic lower bound established in \cite{NWa12}.
	\end{rem}
	
	\subsection{Multi-Task Regression}
	
	Before presenting the theoretical results, we first simplify \eqref{eq:2.6} under the setting of multi-task regression. According to \eqref{eq:2.8}, \eqref{eq:2.6} can be reduced to the following form:
	\beq
		\widehat\bTheta\in \argmin_{\bTheta\in \cS} \frac{1}{d_2} \bigl(-\inn{\widehat\bSigma_{\widetilde \bx\widetilde \by}, \bTheta} + \frac{1}{n}\sum\limits_{i=1}^n \ltwonorm{\bTheta^T\widetilde \bx_i}^2 \bigr) + \lambda_N\nnorm{\bTheta}.
	\eeq
	Recall here that $n$ is the sample size in terms of \eqref{eq:2.2} and $N=d_2n$. We also have $\widehat\bSigma_{Y\bX}- \mat(\widehat\bSigma_{\bX\bX}\vec(\bTheta^*))= \bigl( \widehat\bSigma_{\widetilde \bx\widetilde\by} - \widehat\bSigma_{\widetilde\bx \widetilde\bx}\bTheta^{*}\bigr)/d_2$.
	
	Under the sub-Gaussian design, we only need to shrink the response vector $\by_i$. We choose for \eqref{eq:2.8} $\widetilde \bx_i= \bx_i$ and $\widetilde\by_i = (\|\by_i\|_2 \wedge \tau)\by_i/ \|\by_i\|_2$, where $\tau$ is some threshold value that depends on $n, d_1$ and $d_2$. In other words, we keep the original design, but shrink the Euclidean norm of the response. Note that when $\by_i$ is one-dimensional, the shrinkage reduces to the truncation $y_i(\tau) = \mbox{sgn}(y_i) (|y_i|\wedge \tau)$.
When the design is only of bounded moments, we need to shrink both the design vector $\bx_i$ and response vector $\by_i$ by their $\ell_4$  norm instead, i.e., we choose $\widetilde \bx_i=(\|\bx_i\|_4 \wedge \tau_1)\bx_i/ \|\bx_i\|_4$ and $\widetilde \by_i=(\|\by_i\|_4 \wedge \tau_2)\by_i/ \|\by_i\|_4$, where $\tau_1$ and $\tau_2$ are two thresholds.  Here shrinking by the fourth order moment in fact accelerates the convergence rate of the induced bias so that it will match the final statistical error rate. Again, we will not repeat stating these choices in the following lemmas and theorems for less redundancy.

	
	\begin{lem}
		\label{lem:5}
Convergence of gradients of robustified quadratic loss.
\begin{itemize}
\item [(a)] {\bf Sub-Gaussian design}.
Under the following conditions:
  		\begin{itemize} \itemsep -0.05in
			\item [(C1)] $\lambda_{\max}(\E \by_i\by_i^T)\le R <\infty$;
  			\item [(C2)] $\{\bx_i\}_{i=1}^n$ are i.i.d. sub-Gaussian vectors with $\|\bx_i\|_{\psi_2}\le \kappa_0< \infty$, $\E \bx_i=\bzero$ and $\lambda_{\min}(\E \bx_i\bx_i^T)\ge \kappa_{\cL}>0.$
  			\item [(C3)] $\forall i=1,...,n$, $j_1,j_2=1,...,d_2$ and $j_1\neq j_2$, $\epsilon_{ij_1}\perp \epsilon_{ij_2} |\bx_i$,  and $\forall j=1,...,d_1$, $\E \left(\E(\epsilon_{ij}^2|\bx_i) \right)^k \le M< \infty$, where $k>1$;			
  		\end{itemize}
		there exists some constant $\gamma>0$ such that if $(d_1+d_2)\log(d_1+d_2)/n<\gamma$, we have for any $\delta>0$,
		\[
			P\Big(\opnorm{\widehat\bSigma_{\widetilde \bx\widetilde\by}(\tau) - \widehat\bSigma_{\widetilde \bx \widetilde \bx}\bTheta^*} \ge \sqrt{\frac{(\nu_1+\delta)(d_1+d_2)\log (d_1+d_2)}{n}}\Big)\le 2(d_1+d_2)^{1-\eta_1\delta},
		\]
	where $\tau \asymp \sqrt{n/((d_1+d_2)\log (d_1+d_2))}$ and $\nu_1$ and $\eta_1$ are universal constants.
	
\item [(b)]{\bf Bounded moment design}. Consider Condition (C2') that for any $\bv\in \cS^{d_1-1}$, $\E (\bv^T\bx_i)^4 \le M<\infty$. Under Conditions (C1), (C2') and (C3), it holds for any $\delta>0$
		\[
			P\Bigl( \opnorm{\widehat\bSigma_{\widetilde \bx\widetilde \by}(\tau_1, \tau_2)- \widehat\bSigma_{\widetilde \bx \widetilde \bx}(\tau_1)\bTheta^*} \ge \sqrt{\frac{(\nu_2+\delta)(d_1+d_2)\log(d_1+d_2)}{n}}\Bigr) \le 2(d_1+d_2)^{1-\eta_2\delta},
		\]
		where $\tau_1, \tau_2\asymp \bigl(n/((d_1+d_2)\log (d_1+d_2)) \bigr)^{\frac{1}{4}}$ and $\nu_2$ and $\eta_2$ are universal constants.
\end{itemize}
		\end{lem}
		
		\begin{rem}
    		When the noise and design are sub-Gaussian, \cite{NWa11} used the covering argument to show that for regular sample covariance matrices $\overline\bSigma_{\bx\by}$ and $\overline\bSigma_{\bx\bx}$, $$\opnorm{\overline\bSigma_{\bx\by}- \overline\bSigma_{\bx\bx}\bTheta^*}= \opnorm{\overline\bSigma_{\bx\by}-  \frac{1}{n}\sum\limits_{j=1}^n \bx_j\bx_j^T\bTheta^{*}}= \opnorm{\frac{1}{n}\sum\limits_{j=1}^n \bepsilon_j\bx_j^T} = O_P(\sqrt{(d_1+d_2)/n}).$$ Lemma \ref{lem:5} shows that up to just a logarithmic factor, the shrinkage sample covariance achieves nearly the same rate of convergence for noise and design with only bounded moments.
		\end{rem}
	
		Finally we establish the statistical error rate for the low-rank multi-tasking regression.
		
		
		\begin{thm}
		\label{thm:mt}
Statistical error rate for multitask regression.  Assume
$\cB_q(\bTheta^*)\le \rho$.
\begin{itemize}
\item [(a)] {\bf Sub-Gaussian design}.
		Suppose that Conditions (C1), (C2) and (C3) in Lemma \ref{lem:5} hold. For any $\delta>0$, choose
		\[
			\tau \asymp \sqrt{n/((d_1+d_2)\log (d_1+d_2))} \quad\text{and}\quad \lambda_N=\frac{2}{d_2}\sqrt{(\nu_1+\delta)(d_1+d_2) \log (d_1+d_2)/n},
		\]
		where $\nu_1$ is the same as in Lemma \ref{lem:5}. There exist constants $\gamma_1, \gamma_2>0$ such that if $(d_1+d_2)\log(d_1+d_2)/n<\gamma_1$ and $d_1+d_2\ge \gamma_2$,  then with probability at least $1-3(d_1+d_2)^{1-\eta_1\delta}$ we have
		\[
			\fnorm{\widehat\bTheta(\tau, \lambda_N) -\bTheta^*}^2 \le C_1\rho \Bigl(\frac{(\nu_1+\delta)(d_1+d_2)\log(d_1+d_2)}{n}\Bigr)^{1-\frac{q}{2}}
		\]
		and
		\[
			\nnorm{\widehat\bTheta(\tau, \lambda_N) -\bTheta^*} \le C_2\rho \Bigl(\frac{(\nu_1+\delta)(d_1+d_2)\log(d_1+d_2)}{n}\Bigr)^{\frac{1-q}{2}},
		\]
		where $C_1$ and $C_2$ are universal constants and $\eta_1$ is the same as in Lemma \ref{lem:5}.
\item [(b)] {\bf Bounded moment design}.
		Suppose instead that Conditions (C1), (C2') and (C3) in Lemma \ref{lem:5} hold. For any $\delta>0$, choose
		\[
		\tau_1, \tau_2\asymp \bigl(n/((d_1+d_2)\log (d_1+d_2)) \bigr)^{\frac{1}{4}}\quad \text{and}\quad \lambda_N=\frac{2}{d_2}\sqrt{(\nu_2+\delta)(d_1+d_2) \log (d_1+d_2)/n},
		\]
		where $\nu_2$ is the same as in Lemma \ref{lem:5}. There exist constants $\gamma_3, \gamma_4>0$ such that if $(d_1+d_2)\log(d_1+d_2)/n<\gamma_3$ and $d_1+d_2\ge \gamma_4$,  then with probability at least $1-3(d_1+d_2)^{1-\eta_2\delta}$,
		\[
			\fnorm{\widehat\bTheta(\tau_1, \tau_2, \lambda_N) -\bTheta^*}^2 \le C_3\rho \Bigl(\frac{(\nu_2+\delta)(d_1+d_2)\log(d_1+d_2)}{n}\Bigr)^{1-\frac{q}{2}}
		\]
		and
		\[	
			\nnorm{\widehat\bTheta(\tau_1, \tau_2, \lambda_N) -\bTheta^*} \le C_4\rho \Bigl(\frac{(\nu_2+\delta)(d_1+d_2)\log(d_1+d_2)}{n}\Bigr)^{\frac{1-q}{2}},
		\]
		where $C_3$ and $C_4$ are universal constants and $\eta_2$ is the same as in Lemma \ref{lem:5}.
\end{itemize}
	\end{thm}

%
%
%
%
%

	\section{Robust Covariance Estimation} \label{sec4}
	
	In derivation of $\opnorm{\widehat\bSigma_{\widetilde\bx \widetilde\by}(\tau_1, \tau_2)- \widehat\bSigma_{\widetilde\bx\widetilde\bx}(\tau_1)\bTheta^*}$ in multi-tasking regression, we find that when the random sample has only bounded moments, the $\ell_4-$norm shrinkage sample covariance achieves the spectral norm convergence rate of order $O_P(\sqrt{d \log d/n})$ in estimating the true covariance. This is nearly the optimal rate with sub-Gaussian random samples, up to a logarithmic factor. Here we formulate the problem and the result, whose proof is relegated to the appendix.
	
	Suppose we have $n$ i.i.d. $d$-dimensional random vectors $\{\bx_i\}_{i=1}^n$ with $\E\bx_i=\bzero$. Our goal is to estimate the covariance matrix $\bSigma=\E(\bx_i\bx_i^T)$ when the distribution of $\{\bx_i\}_{i=1}^n$ has only fourth bounded moment. For any $\tau\in \RR^+$, let $\widetilde\bx_i:=(\|\bx_i\|_4 \wedge \tau)\bx_i/\|\bx_i\|_4$, where $\|\cdot\|_4$ is the $\ell_4-$norm. We propose the following shrinkage sample covariance to estimate $\bSigma$.
	\beq
		\label{eq:4.1}
		\widetilde\bSigma_n(\tau)=\frac{1}{n}\sum\limits_{i=1}^n \widetilde\bx_i\widetilde\bx_i^T.
	\eeq
	
	\begin{thm}
		\label{thm:6}
		Suppose $\E(\bv^T\bx_i)^4\le R$ for any $\bv\in \cS^{d-1}$, then it holds that for any $\delta>0$,
		\begin{equation}
			P\Big(\opnorm{\widetilde\bSigma_n(\tau) - \bSigma} \ge \sqrt{\frac{\delta Rd \log d}{n}} \Big)\le d^{1-C\delta},
		\end{equation}
		where $\tau\asymp \bigl(nR/(\delta\log d) \bigr)^{1/4}$ and $C$ is a universal constant.
	\end{thm}
		
		We have several comments for the above theorem. First of all, unlike the bounded moment conditions in the previous section, $R$ here can go to infinity with certain rates. Hence we also put this quantity into the rate of convergence. If indeed $R < \infty$, we recover the true covariance matrix with statistical error of $O_P(\sqrt{d\log d/n})$. In comparison with the robust covariance matrix estimator given in \cite{FLW16}, our estimator here is positive semidefinite and is very simple to implement.  Our concentration inequality here is on the operator norm, while their result is on the element-wise max norm.
		
		If we are concerned with statistical error in terms of elementwise max norm, then we need to apply elementwise truncation to the random samples rather than $\ell_4$ norm shrinkage. Let $\check \bx_i$ satisfy $\check x_{ij}=sgn(x_{ij})(|x_{ij}|\wedge \tau)$ for $1\le j\le d_1$ and $\check\bSigma_n= \sum\limits_{i=1}^n \check \bx_i\check\bx_i^T$. It is not hard to derive $\|\check \bSigma_{n}- \bSigma\|_{\max}= O_P(\sqrt{\log d/n})$ as in \cite{FLW16} with $\tau \asymp (n/\log d)^{\frac{1}{4}}$. Further regularization can be applied to $\check\bSigma_n$ if the true covariance is sparse. See, for example, \cite{MBu06}, \cite{BLe08}, \cite{LFa09}, \cite{CLi11}, \cite{CZh12}, \cite{FLM13}, among others.
		 		
		Theorem \ref{thm:6} is non-asymptotic as all the results in the previous sections and can be applied to both low-dimensional and high-dimensional regimes. In fact, $\widetilde\bSigma_n(\tau)$ can outperform the sample covariance $\overline\bSigma_n$ even with Gaussian samples if the dimension is high. The reason is that according to Theorem 5.39 in \cite{Ver10}, $\opnorm{\overline\bSigma_n- \bSigma}=O_P(\sqrt{d/n} \vee (d/n))$. When $d/n$ is large, the $d/n$ term will dominate $\sqrt{d/n}$, thus delivering statistical error of order $d/n$ for Gaussian sample covariance. However, our shrinkage sample covariance always retains the statistical error of order $\sqrt{d\log d/n}$ regardless of ratio between the dimension and sample size. Therefore, shrinkage overcomes not only heavy-tailed corruption, but also curse of dimensionality. In Section 5.4, we conduct simulations to further illustrate this point.
		
		

	\section{Simulation Study}
	
	In this section, we first compare the numerical performance of the robust procedure and standard procedure in compressed sensing, matrix completion and multi-tasking regression. For each setup, we investigate three noise settings: log-normal noise, truncated Cauchy noise and Gaussian noise. They represent heavy-tailed asymmetric distributions, heavy-tailed symmetric distributions and light-tailed symmetric distributions.  The results from the standard and robust methods are shown in the same color for each scenario in the following figures so that they can be compared more easily.

All the robust procedures proposed in our work are very easy to implement; we only need to truncate or shrink the data appropriately, and then apply the standard procedure to the transformed data. As for the parameter tuning, we refer to the rate developed in our theories for robust procedures and theories in \cite{NWa11} and \cite{NWa12}  for standard procedures. The constants before the rate are tuned for best performance. The main message is that the robust procedure outperforms the standard procedure under the setting with bounded moment noise, and it performs equally well as the standard procedure under the Gaussian noise. The simulations are based on $100$ independent Monte Carlo replications. There are no standard algorithms for solving penalized trace regression problems.  Hence, besides presenting the numerical performance, we also elucidate the algorithms that we use to solve the corresponding optimization problems, which might be of interest to readers concerned with implementations.  These algorithms are not necessarily the same as those used in the literature.

	We also compare the numerical performance of the regular sample covariance and shrinkage sample covariance as proposed in \eqref{eq:4.1} in estimating the true covariance. We choose $d/n=0.2, 0.5, 1$ and for each ratio, we let $n=100, 200, ..., 500$.  Simulation results show superiority of the shrinkage sample covariance over the regular sample covariance under both Gaussian noise and $t_3$ noise. Therefore, the shrinkage can not only overcome the heavy-tailed corruption, but also mitigate the curse of high dimensions.

	\subsection{Compressed Sensing}
	
	We first specify the parameters in the true model: $Y=\inn{\bX_i, \bTheta^*}+\epsilon_i$. We always set $d_1=d_2=d$, $\fnorm{\bTheta^*}=\sqrt{5}$ and $rank(\bTheta^*)=5$. In the simulation, we construct $\bTheta^*$ to be $\sum\limits_{i=1}^5 \bv_i\bv_i^T$, where $\bv_i$ is the $i$th top eigevector of the sample covariance of $100$ i.i.d. centered Gaussian random vectors with covariance $\bI_d$. The design matrix $\bX_i$ has i.i.d. standard Gaussian entries. The noise distributions are characterized as follows:
	\begin{itemize}
		\item Log-normal: $\epsilon_i \overset{d}{=} (Z-\E Z)/50$, where $Z \sim ln \cN(0, \sigma^2)$ and $\sigma^2=6.25$;
		\item Truncated Cauchy: $\epsilon_i \overset{d}{=} \min(Z, 10^3)/10$, where $Z$ follows Cauchy distribution;
		\item Gaussian: $\epsilon_i \sim N(0, \sigma^2)$, where $\sigma^2=0.25$.
	\end{itemize}
	The constants above are chosen to ensure appropriate signal-to-noise ratio for better presentation. We present the numerical results in Figure \ref{fig:1}. As we can observe from the plots, the robust estimator has much smaller statistical error than the standard estimator under the heavy-tailed noise, i.e., the log-normal and truncated Cauchy noise. When $d=40$ or $60$, robust procedures deliver sharper estimation as the sample size increases, while the standard procedure does not necessarily do so under the heavy-tailed noise. Under the setting of Gaussian noise, the robust estimator has nearly the same statistical performance as the standard one, which shows that it does not hurt to use the robust procedure under the light-tailed noise setting.
	
		\begin{figure}
		\begin{tabular}{ccc}
			\hspace{-.3in}\includegraphics[height=3in, width=2.4in]{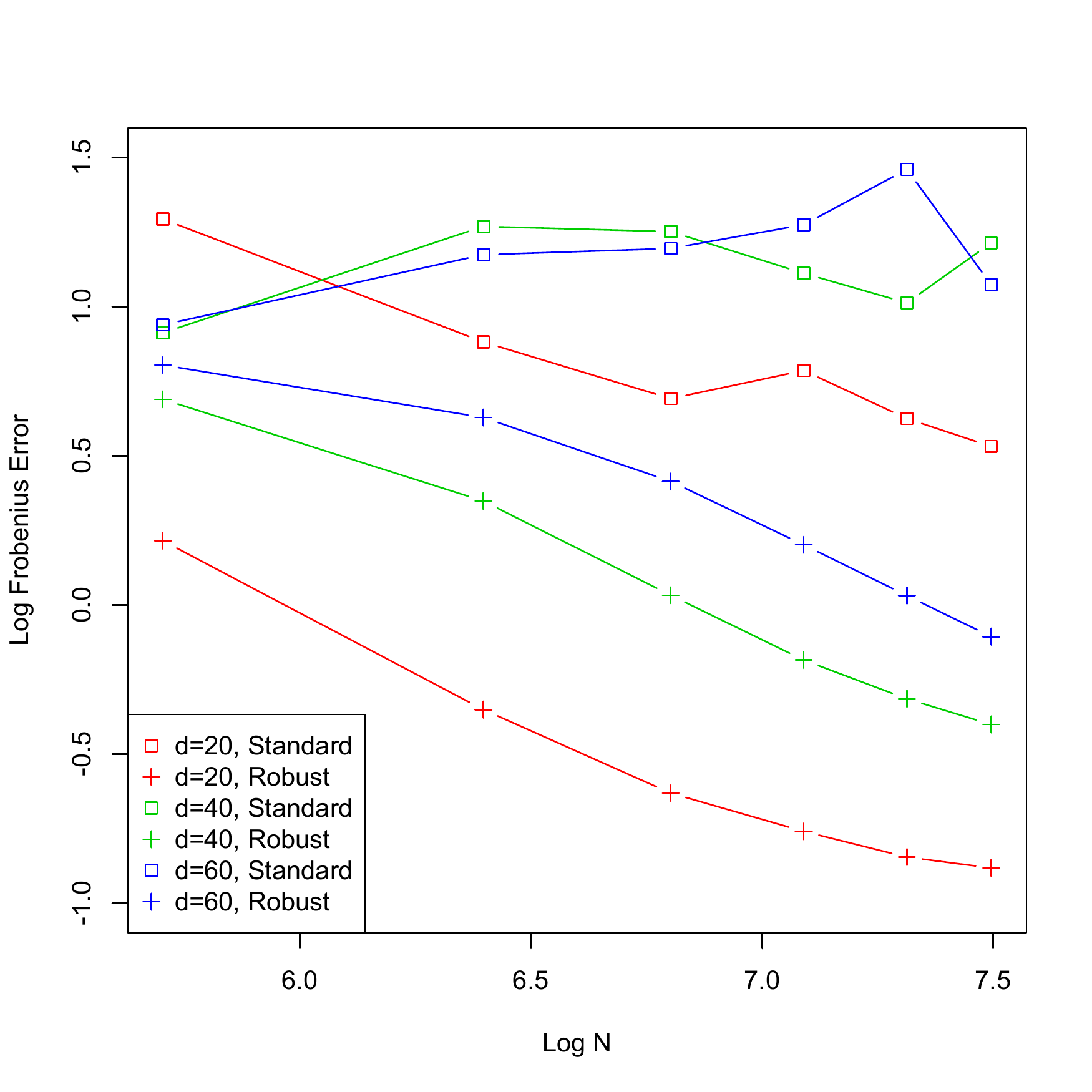} & \hspace{-.2in}\includegraphics[height=3in, width=2.4in]{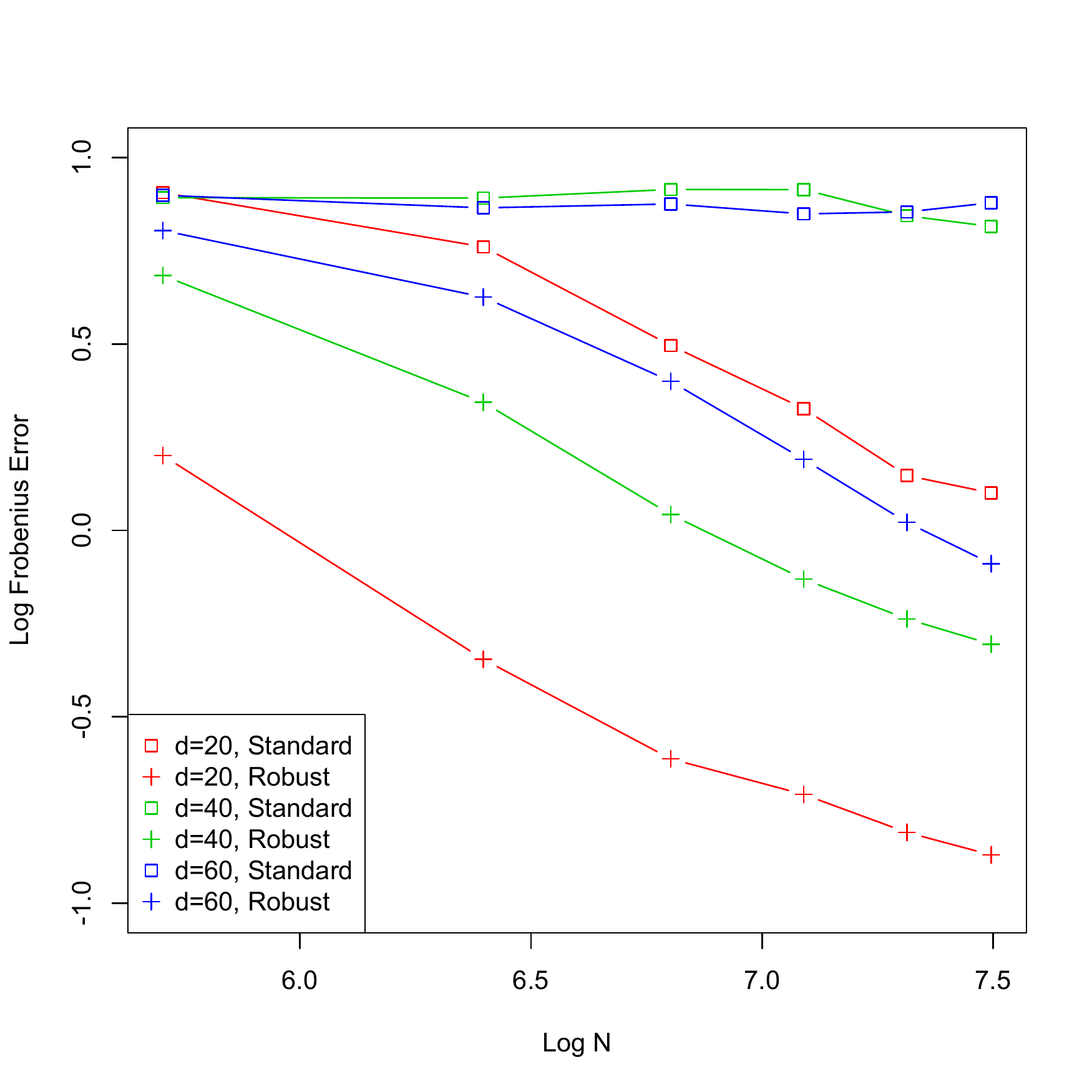} & \hspace{-.25in}\includegraphics[height=3in, width=2.4in]{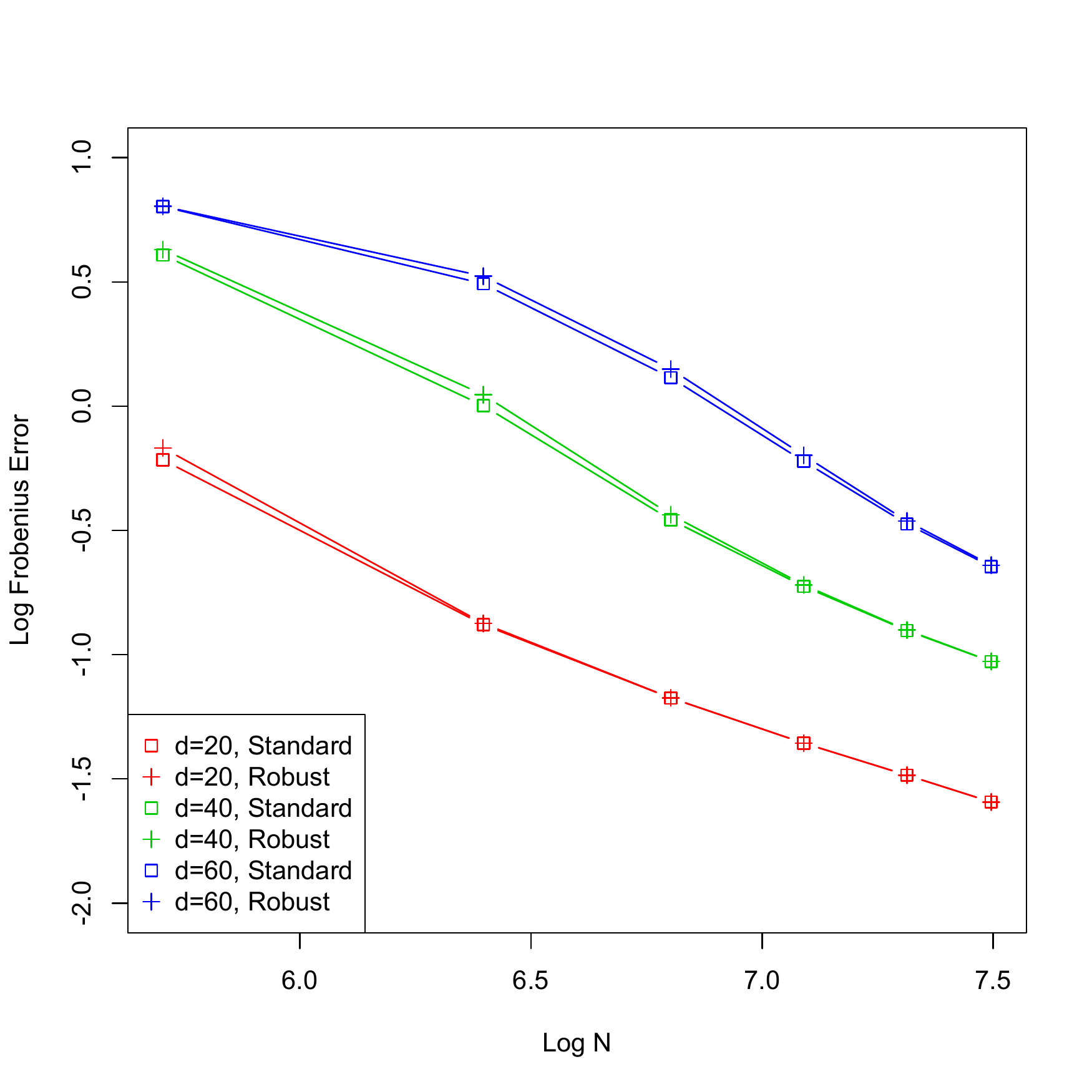} \\
			Log-normal Noise & Truncated Cauchy Noise & Gaussian Noise
		\end{tabular}
		\caption{Statistical errors of $\ln \fnorm{\widehat\bTheta- \bTheta^*}$ v.s. logarithmic sample size $\ln N$ for different dimensions $d$ in matrix compressed sensing.}
		\label{fig:1}
	\end{figure}
	
	As for the implementation, we exploit the contractive Peaceman-Rachford splitting method (PRSM) to solve the compressed sensing problem.
Here we briefly introduce the general scheme of the contractive PRSM for clarity. The contractive PRSM is for minimizing the summation of two convex functions under linear constraint:
	\beq
		\begin{aligned}
			\min_{\bx\in \RR^{p_1}, \by\in \RR^{p_2}} & \quad f_1(\bx)+f_2(\by), \\
			\text{subject to} & \quad\bC_1\bx+\bC_2\by-\bc=\bzero,
		\end{aligned}
	\eeq
	where $\bC_1\in \RR^{p_3\times p_1}$, $\bC_2\in \RR^{p_3\times p_2}$ and $\bc\in \RR^{p_3}$. The general iteration scheme of the contractive PRSM is
	\beq
		\label{eq:5.2}
		\left\{
		\begin{aligned}
		& \bx^{(k+1)}=\argmin_\bx \bigl\{ f_1(\bx)-(\brho^{(k)})^T(\bC_1\bx+\bC_2\by^{(k)}-\bc)+\frac{\beta}{2}\|\bC_1\bx+\bC_2\by^{(k)}-\bc\|_2^2\bigr\}, \\
		& \brho^{(k+\frac{1}{2})}= \brho^{(k)}-\alpha\beta(\bC_1\bx^{(k+1)}+\bC_2\by^{(k)}-\bc), \\
		& \by^{(k+1)}= \argmin_{\by}\bigl\{f_2(\by)-(\brho^{(k+\frac{1}{2})})^T(\bC_1\bx^{(k+1)}+\bC_2\by-\bc)+\frac{\beta}{2}\|\bC_1\bx^{(k+1)}+\bC_2\by-\bc\|_2^2\bigr\}, \\
		& \brho^{(k+1)}= \brho^{(k+\frac{1}{2})}-\alpha\beta(\bC_1\bx^{(k+1)}+\bC_2\by^{(k+1)}-\bc),
		\end{aligned}
		\right .
	\eeq
	where $\brho\in \RR^{p_3}$ is the Lagrangian multiplier, $\beta$ is the penalty parameter and $\alpha$ is the relaxation factor (\cite{EBe92}). Since the parameter of interest in our work is $\bTheta^*$, now we use $\btheta_x \in \RR^{d_1d_2}$ and  $\btheta_y \in \RR^{d_1d_2}$ to replace $\bx$ and $\by$ respectively in \eqref{eq:5.2}. By substituting
	\[
		f_1(\btheta_x)=\frac{1}{N}\sum\limits_{i=1}^N (Y_i-\inn{\vec(\bX_i), \btheta_x})^2,\ f_2(\btheta_y)=\lambda \nnorm{\mat(\btheta_y)},\ \bC_1=\bI,\ \bC_2=-\bI\ \text{and}\ \bc=\bzero
	\]
	into \eqref{eq:5.2}, we can obtain the PRSM algorithm for the compressed sensing problem. Let $\mathbb{X}$ be a $N$-by-$d_1d_2$ matrix whose rows are i.i.d. random designs $\{\vec(\bX_i)\}_{i=1}^N$ and $\mathbb{Y}$ be the $N$-dimensional response vector. Then we have the following iteration scheme specifically for compressed sensing.
	
	\beq
		\left\{
		\begin{aligned}
		& \btheta_x^{(k+1)}=(2\bbX^T\bbX/N+ \beta \cdot \bI )^{-1}(\beta\cdot \btheta_y^{(k)}+\brho^{(k)} + 2\bbX^T\bbY/N), \\
		& \brho^{(k+\frac{1}{2})}= \brho^{(k)}-\alpha\beta(\btheta_x^{(k+1)}-\btheta_y^{(k)}), \\
		& \btheta_y^{(k+1)}=\vec(S_{\lambda/\beta}(\mat(\btheta_x-\brho^{(k+\frac{1}{2})}))), \\
		& \brho^{(k+1)}= \brho^{(k+\frac{1}{2})}-\alpha\beta(\btheta_x^{(k+1)}-\btheta_y^{(k+1)}),
		\end{aligned}
		\right .
		\label{eq:csiter}
	\eeq
	where we choose $\alpha=0.9$ and $\beta=1$ according to \cite{EBe92} and \cite{HLW14}, $\brho \in \RR^{d_1d_2}$ is the Lagrangian multiplier and $S_\tau(\bz)$ is the singular value soft-thresholding function for matrix version of $\bz\in \RR^{d_1d_2}$. To be more specific, let $\bZ=\mat(\bz)\in \RR^{d_1\times d_2}$ and $\bZ=\bU\bLambda\bV^T= \bU\ diag(\lambda_1,..., \lambda_r)\ \bV^T$ be its singular value decomposition. Then $S_\tau(\bz)=\vec\bigl(\bU\  diag((\lambda_1-\tau)_+, (\lambda_2-\tau)_+, ..., (\lambda_r-\tau)_+) \ \bV^T\bigr)$, where $(x)_+=\max(x, 0)$.
	The algorithm stops if $\ltwonorm{\btheta_x- \btheta_y}$ is smaller than some predetermined threshold, and returns $\mat(\btheta_y)$ as the final estimator of $\bTheta^*$.
	

	\subsection{Matrix Completion}
	
We again set $d_1=d_2=d$ and construct $\bTheta^*$ to be $\sum\limits_{i=1}^5 \bv_i\bv_i^T/\sqrt{5}$, where $\bv_i$ is the $i$th top eigevector of the sample covariance of $100$ i.i.d. centered Gaussian random vectors with covariance $\bI_d$. 
Each design matrix $\bX_i$ takes the singleton form, which is uniformly sampled from $\{\be_j\be_k^T\}_{1\le j, k\le d}$.  The noise distributions are
	\begin{itemize}
		\item Log-normal: $\epsilon_i \overset{d}{=} (Z-\E Z)/250$, where $Z \sim \ln \cN(0, \sigma^2)$ and $\sigma^2=9$;
		\item Truncated Cauchy: $\epsilon_i \overset{d}{=} \min(Z, 10^3)/16$, where $Z$ follows Cauchy distribution;
		\item Gaussian: $\epsilon_i \sim N(0, \sigma^2)$, where $\sigma^2=0.25$.
	\end{itemize}
	Again, the constants above are set for an appropriate signal-to-noise ratio for better presentation. We present the numerical results in Figure \ref{fig:2}. Analogous to the matrix compressed sensing, we can observe from the figure that compared with the standard procedure, the robust procedure has significantly smaller statistical error in estimating $\bTheta^*$ under the log-normal and truncated Cauchy noise. Under Gaussian noise, the robust procedure has nearly the same statistical performance as the standard procedure.
	
		\begin{figure}
		\begin{tabular}{ccc}
			\hspace{-.3in}\includegraphics[height=3in, width=2.4in]{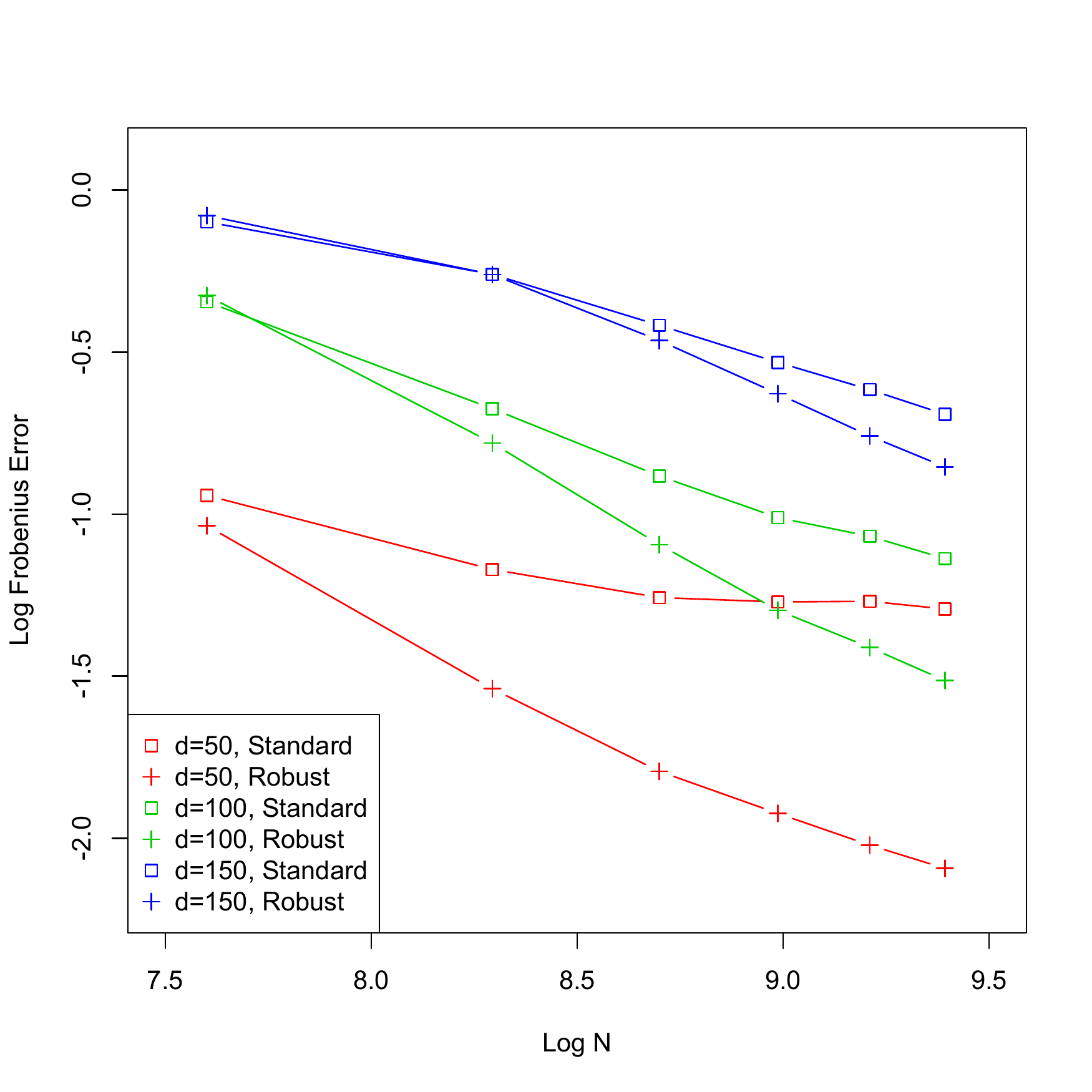} & \hspace{-.2in}\includegraphics[height=3in, width=2.4in]{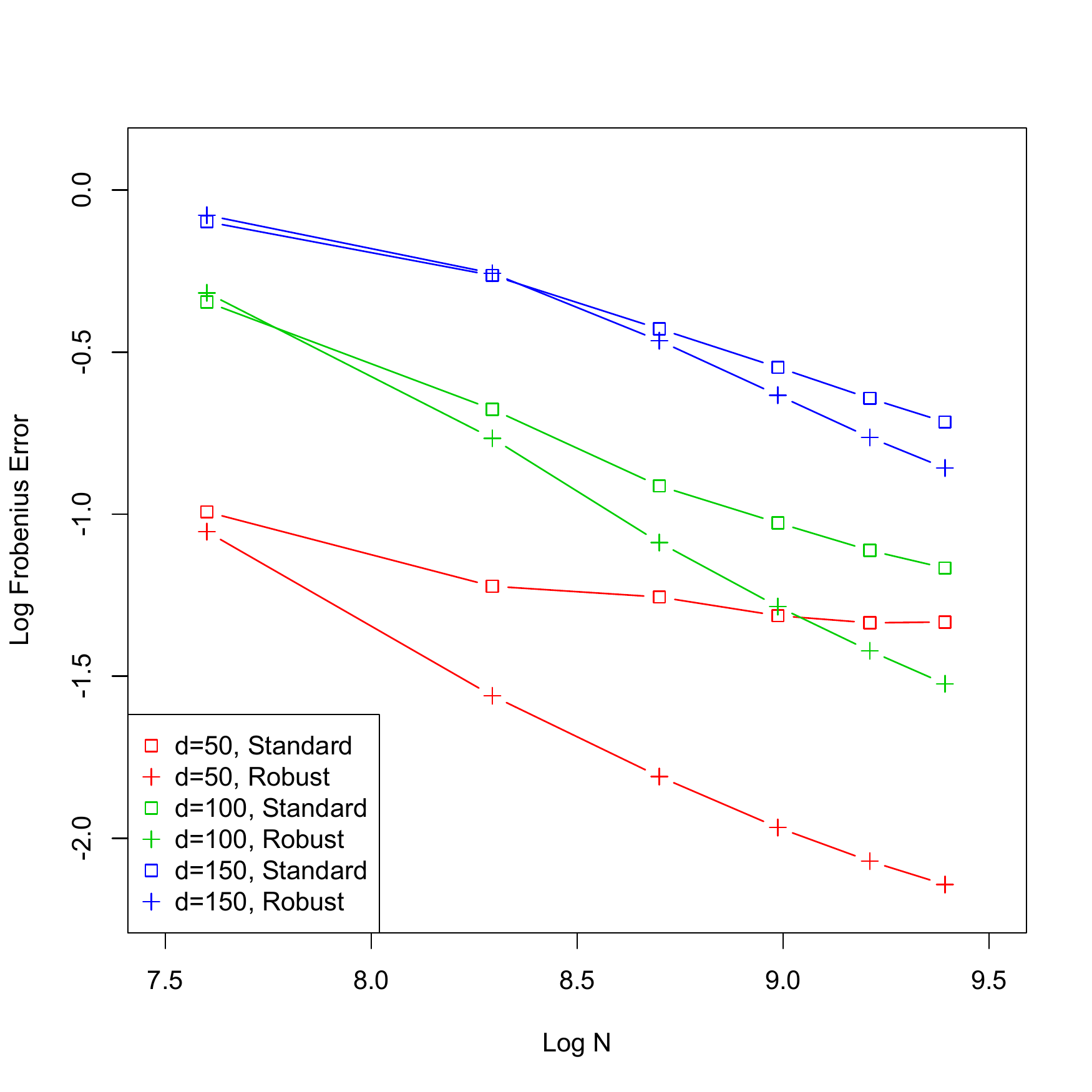} & \hspace{-.25in}\includegraphics[height=3in, width=2.4in]{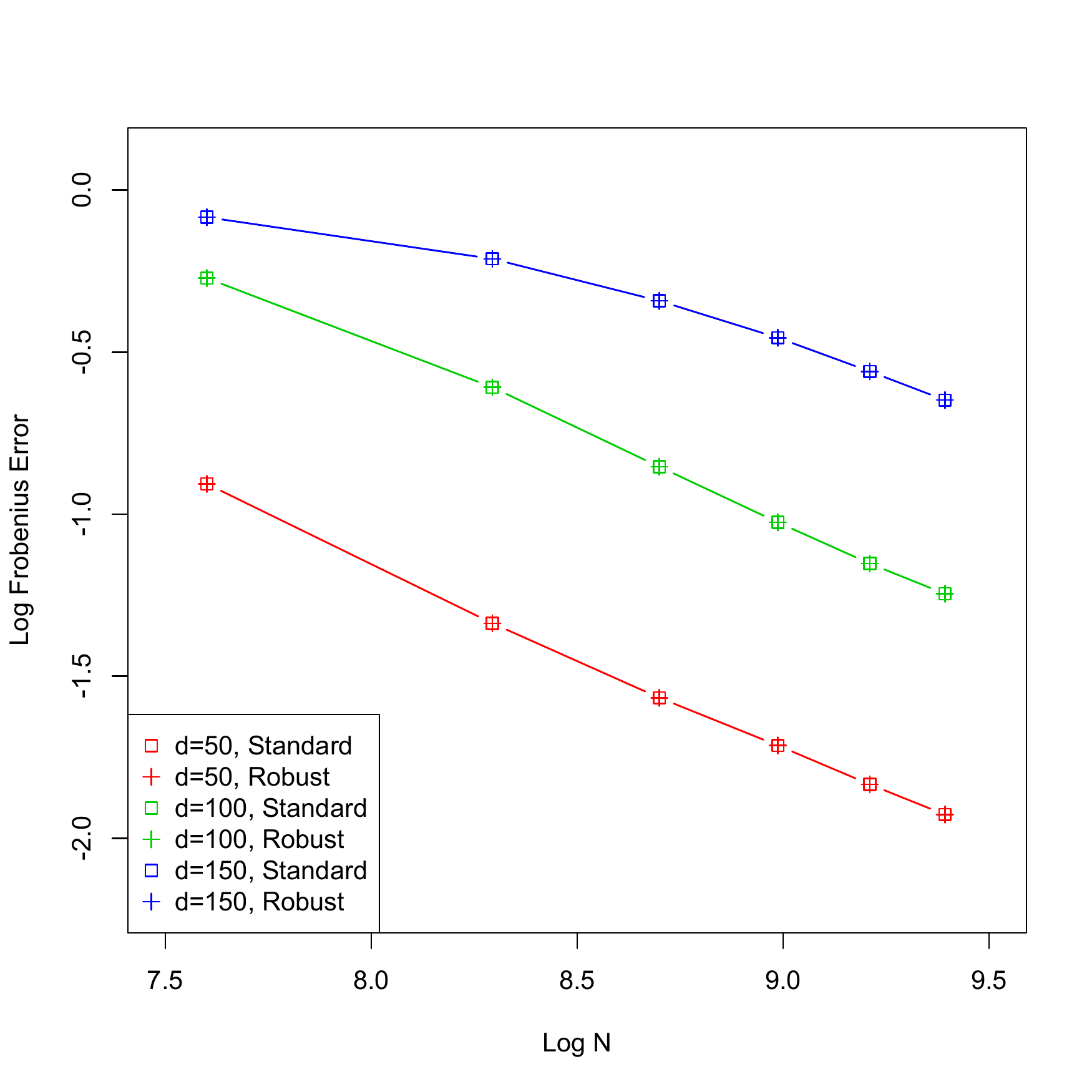} \\
			Log-normal Noise & Truncated Cauchy Noise & Gaussian Noise
		\end{tabular}
		\caption{Statistical errors of $\ln \fnorm{\widehat\bTheta- \bTheta^*}$ v.s. logarithmic sample size $\ln N$ for different dimensions $d$ in matrix completion.}
		\label{fig:2}
	\end{figure}

To solve the matrix completion problem in \eqref{eq:3.6}, we adapt the ADMM method inspired by \cite{FLT15}. They propose to recover the matrix by minimizing the square loss plus both nuclear norm and matrix max-norm penalizations under the entrywise max-norm constraint. By simply setting the penalization coefficient for the matrix max-norm to be zero, we can apply their algorithm to our problem. Let $\bL, \bR, \bW  \in \RR^{(d_1+d_2)\times (d_1+d_2)}$, which are variables in our algorithm. Define $\bTheta^n, \bTheta^s \in \RR^{d_1\times d_2}$ such that $\bTheta^n_{ij}=\sum\limits_{t=1}^N \ind_{\{\bX_t=\be_i\be_j^T\}}$ and $\bTheta^s_{ij}=\sum\limits_{t=1}^N Y_t \ind_{\{\bX_t=\be_i\be_j^T\}}$, so $\bTheta^n$ and $\bTheta^s$ are constants given the data. Below we present the iteration scheme of our algorithm to solve the matrix completion problem \eqref{eq:3.6}. Readers who are interested in the derivation of the algorithm can refer to \cite{FLT15} for the technical details.
	
	\beq
		\left \{
		\begin{aligned}
			& \bL^{(k+1)}=\Pi_{\cS^{d_1+d_2}_+}\{\bR^{(k)}-\rho^{-1}(\bW^{(k)}+\lambda_N \bI)\}, \\
			& \bC= \left(
				\begin{array}{cc}
					\bC^{11} & \bC^{12} \\
					\bC^{21} & \bC^{22}
				\end{array}
			\right)=\bL^{(k+1)}+\bW^{(k)}/\rho, \\
			& \bR^{12}_{ij}=\Pi_{[-R, R]} \{(\rho\bC^{12}_{ij}+2\bTheta^s_{ij}/N)/(\rho+2\bTheta^n_{ij}/N)\},\ 1\le i\le d_1, 1\le j\le d_2\\
			& \bR^{(k+1)}= \left(
				\begin{array}{cc}
					\bC^{11} & \bR^{12} \\
					(\bR^{12})^T & \bC^{22}
				\end{array}
			\right), \\
			& \bW^{(k+1)}=\bW^{(k)}+\gamma\rho(\bL^{(k+1)}-\bR^{(k+1)}).
		\end{aligned}
		\right.
	\eeq
	In the algorithm above, $\Pi_{\cS_+^{d_1+d_2}}(\cdot)$ is the projection operator onto the space of positive semidefinite matrices $\cS^{d_1+d_2}_+$, $\rho$ is the penalization parameter which we set to be $0.1$ in our simulation and $\gamma$ is the step length which is typically set to be $1.618$ according to \cite{FLT15}. We omit the stopping criteria for this algorithm here since it is complex. Readers who are interested in the implementation details of this algorithm can refer to \cite{FLT15} for detailed instruction. Once the stopping criteria is satisfied, the algorithm returns $\bR^{12}$ as the final estimator of $\bTheta^*$.

	\subsection{Multi-Tasking Regression}
	
 	We again set $d_1=d_2=d$ and construct $\bTheta^*$ to be $\sum\limits_{i=1}^5 \bv_i\bv_i^T$, where $\bv_i$ is the $i$th top eigevector of the sample covariance of $100$ i.i.d centered Gaussian random vectors with covariance $\bI_d$. The design vectors $\{\bx_i\}_{i=1}^n$ are i.i.d. Gaussian vectors with covariance matrix $\bI_d$. The noise distributions are characterized as follows:
	\begin{itemize}
		\item Log-normal: $\epsilon_i \overset{d}{=} (Z-\E Z)/50$, where $Z \sim ln \cN(0, \sigma^2)$ and $\sigma^2=4$;
		\item Truncated Cauchy: $\epsilon_i \overset{d}{=} \min(Z, 10^4)/10$, where $Z$ follows Cauchy distribution;
		\item Gaussian: $\epsilon_i \sim N(0, \sigma^2)$, where $\sigma^2=0.25$.
	\end{itemize}
	We present the numerical results in Figure \ref{fig:3}. Similar to the two examples before, the robust procedure has sharper statistical accuracy in estimating $\bTheta^*$ than the standard procedure under both heavy-tailed noises. Under Gaussian noise, the robust procedure has nearly the same statistical performance as the standard procedure.
	
		\begin{figure}
		\begin{tabular}{ccc}
			\hspace{-.3in}\includegraphics[height=3in, width=2.4in]{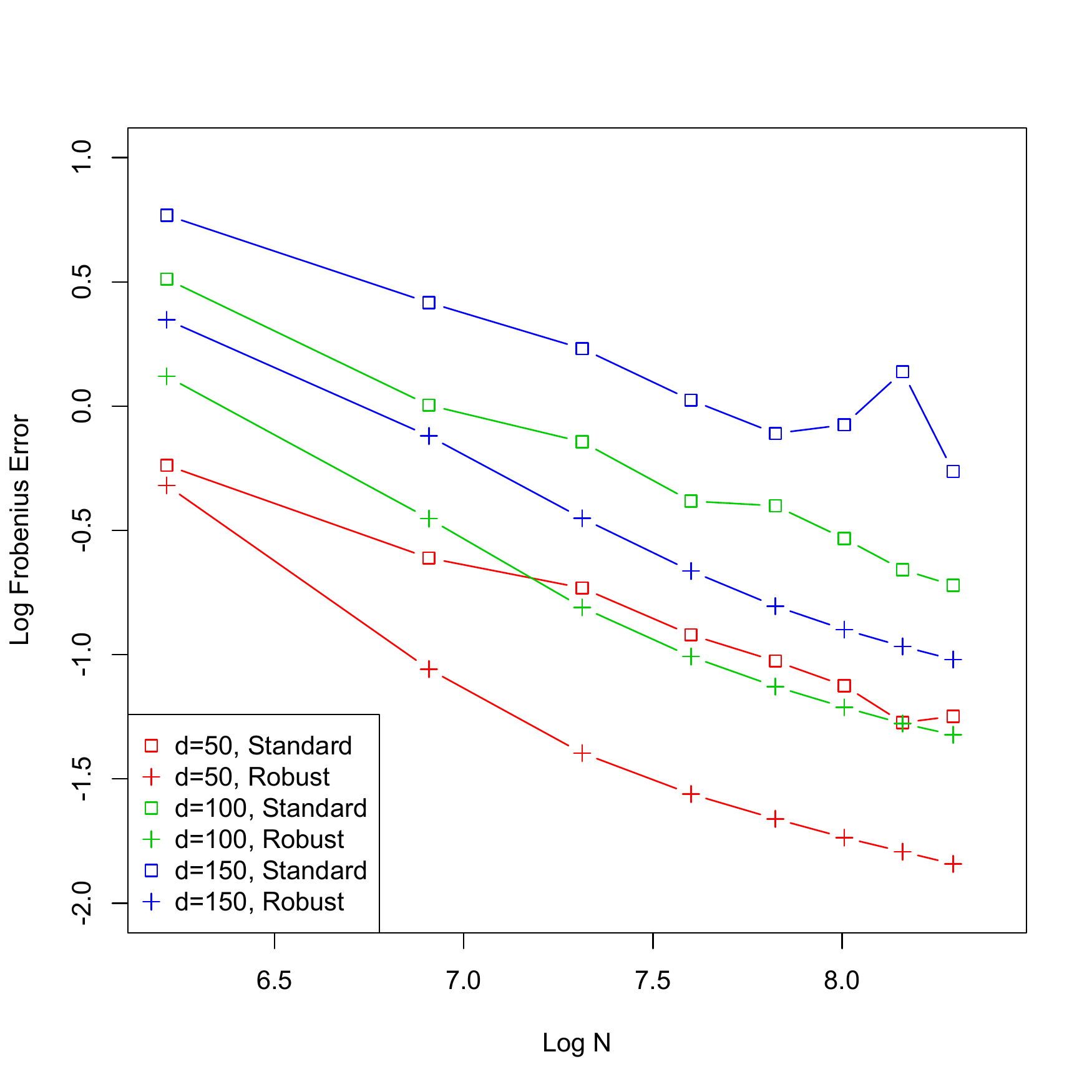} & \hspace{-.2in}\includegraphics[height=3in, width=2.4in]{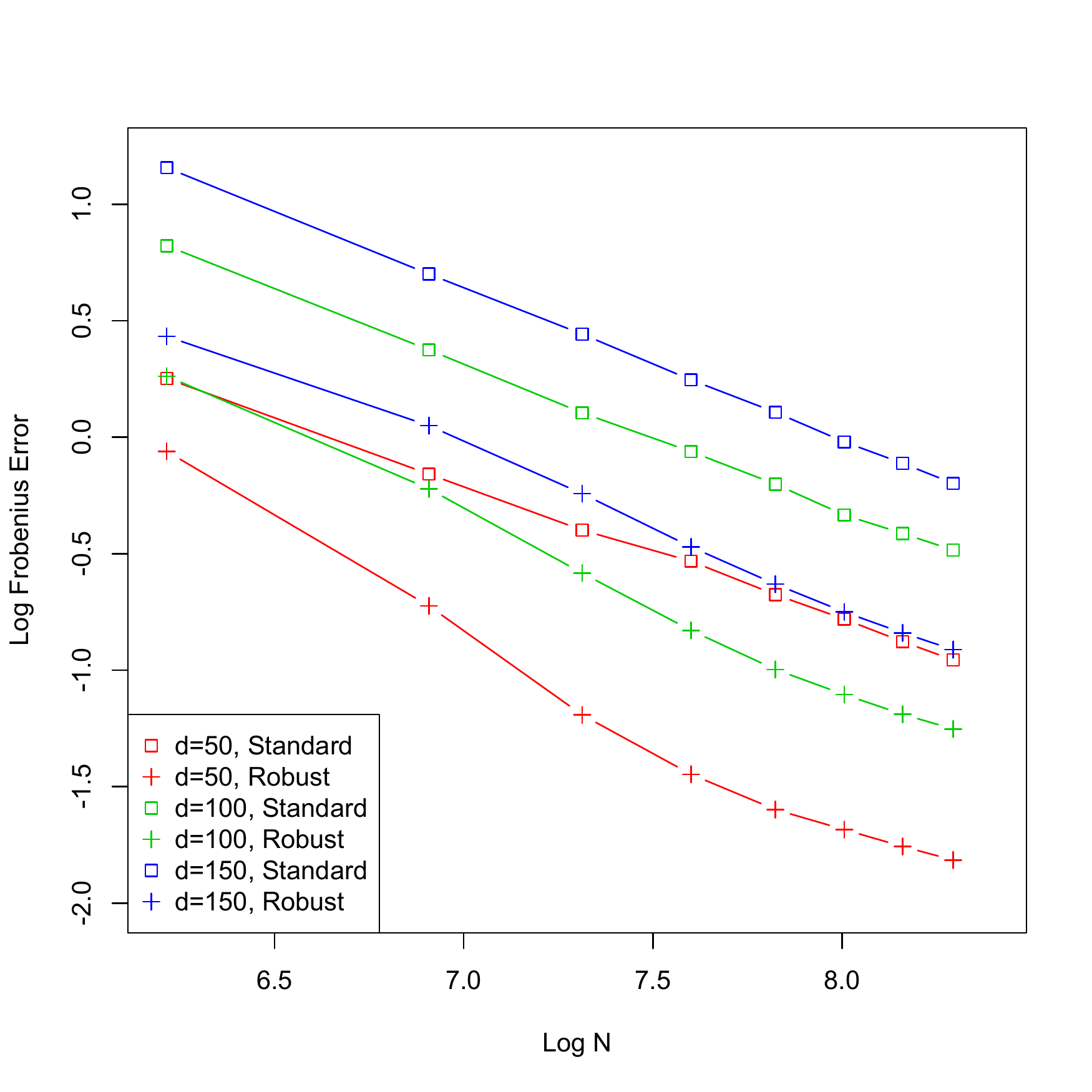} & \hspace{-.25in}\includegraphics[height=3in, width=2.4in]{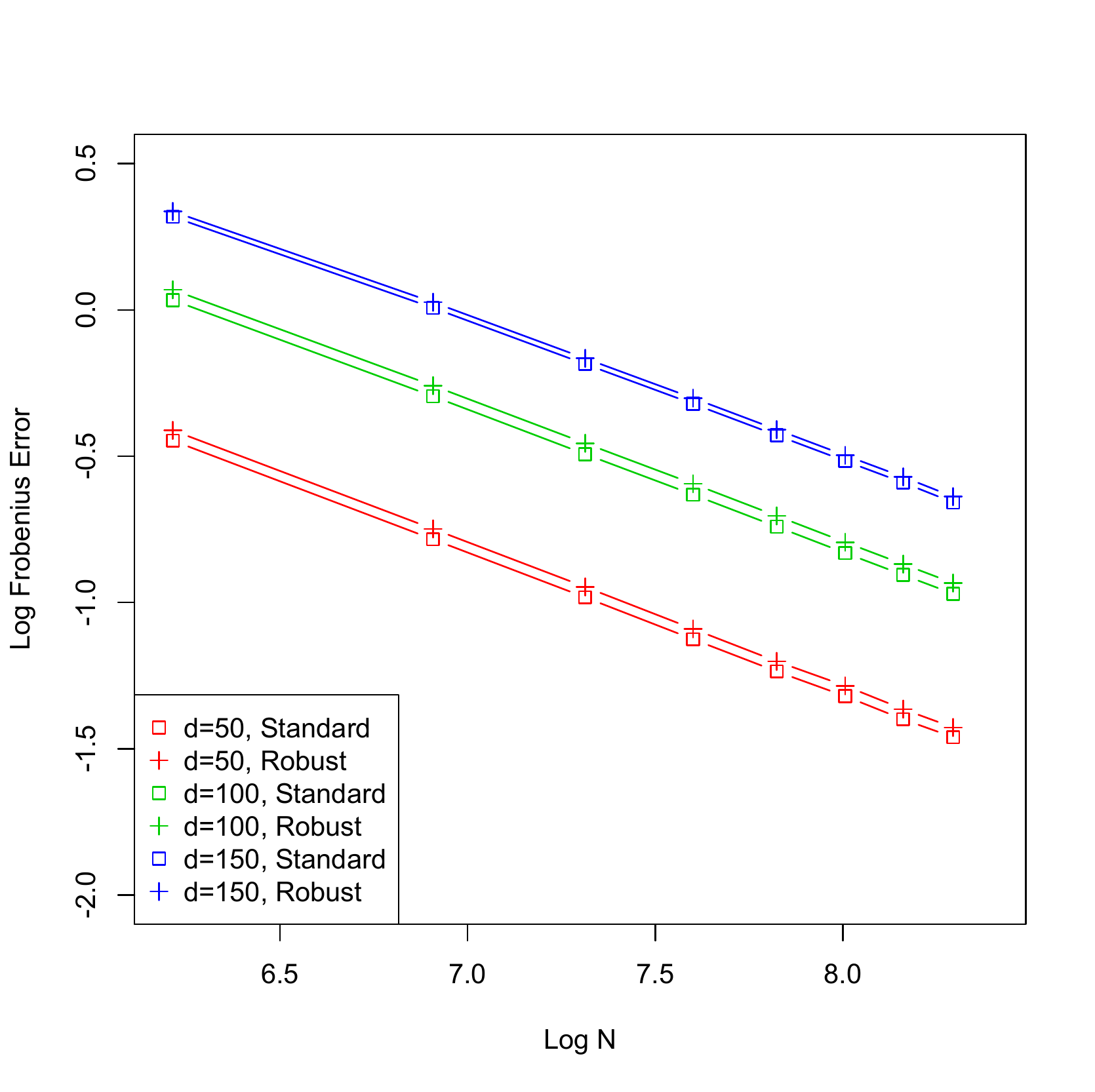} \\
			Log-normal Noise & Truncated Cauchy Noise & Gaussian Noise
		\end{tabular}
		\caption{Statistical errors of $\ln \fnorm{\widehat\bTheta- \bTheta^*}$ v.s. logarithmic sample size $\ln N$ for different dimensions $d$ in multi-tasking regression.}
		\label{fig:3}
	\end{figure}
	
For multi-tasking regression, we exploit the contractive PRSM method again. Let $\bbX$ be the $n$-by-$d_1$ design matrix and $\bbY$ be the $n$-by-$d_2$ response matrix. By following the general iteration scheme \eqref{eq:5.2}, we can develop the iteration steps for the multi-tasking regression.
	
	\beq
		\left\{
		\begin{aligned}
		& \bTheta_x^{(k+1)}=(2\bbX^T\bbX/n+ \beta \cdot \bI )^{-1}(\beta\cdot \bTheta_y^{(k)}+\brho^{(k)} + 2\bbX^T\bbY/n), \\
		& \brho^{(k+\frac{1}{2})}= \brho^{(k)}-\alpha\beta(\bTheta_x^{(k+1)}-\bTheta_y^{(k)}), \\
		& \bTheta_y^{(k+1)}=S_{\lambda/\beta}(\bTheta_x-\brho^{(k+\frac{1}{2})}), \\
		& \brho^{(k+1)}= \brho^{(k+\frac{1}{2})}-\alpha\beta(\bTheta_x^{(k+1)}-\bTheta_y^{(k+1)}),
		\end{aligned}
		\right.
	\eeq
	where $\bTheta_x, \bTheta_y, \brho \in \RR^{d_1\times d_2}$ and $S_\tau(\cdot)$ is the same singular value soft thresholding function as in \eqref{eq:csiter}. Analogous to the compressed sensing, we choose $\alpha=0.9$ and $\beta=1$. As long as $\fnorm{\bTheta_x- \bTheta_y}$ is smaller than some predetermined threshold, the iteration stops and the algorithm returns $\bTheta_y$ as the final estimator of $\bTheta^*$.

	\subsection{Covariance Estimation}
	
	In this subsection, we investigate the statistical error of the sample covariance $\overline\bSigma_n$ and shrinkage sample covariance $\widetilde\bSigma_n(\tau)$ proposed in Section \ref{sec4} when the random samples are heavy-tailed. We only consider two simple distributions: Gaussian and Student's $t_3$ random samples. The dimension is set to be proportional to sample size, i.e., $d/n = \alpha$ with $\alpha$ being $0.2,0.5,1$. $n$ will range from $100$ to $500$ for each case. Regardless of how large the dimension $d$ is, the true covariance $\bSigma$ is always set to be a diagonal matrix with the first diagonal element equal to $4$ and all the other diagonal elements equal to $1$. 	
	We present our results in Figure \ref{fig:4}. The statistical errors are measured in terms of the spectral norm gap between the estimator and true covariance, and our simulation is based on $1,000$ independent Monte Carlo replications.
		
	\begin{figure}[h]
		\centering
		\begin{tabular}{cc}
			\includegraphics[height=3in, width=3in]{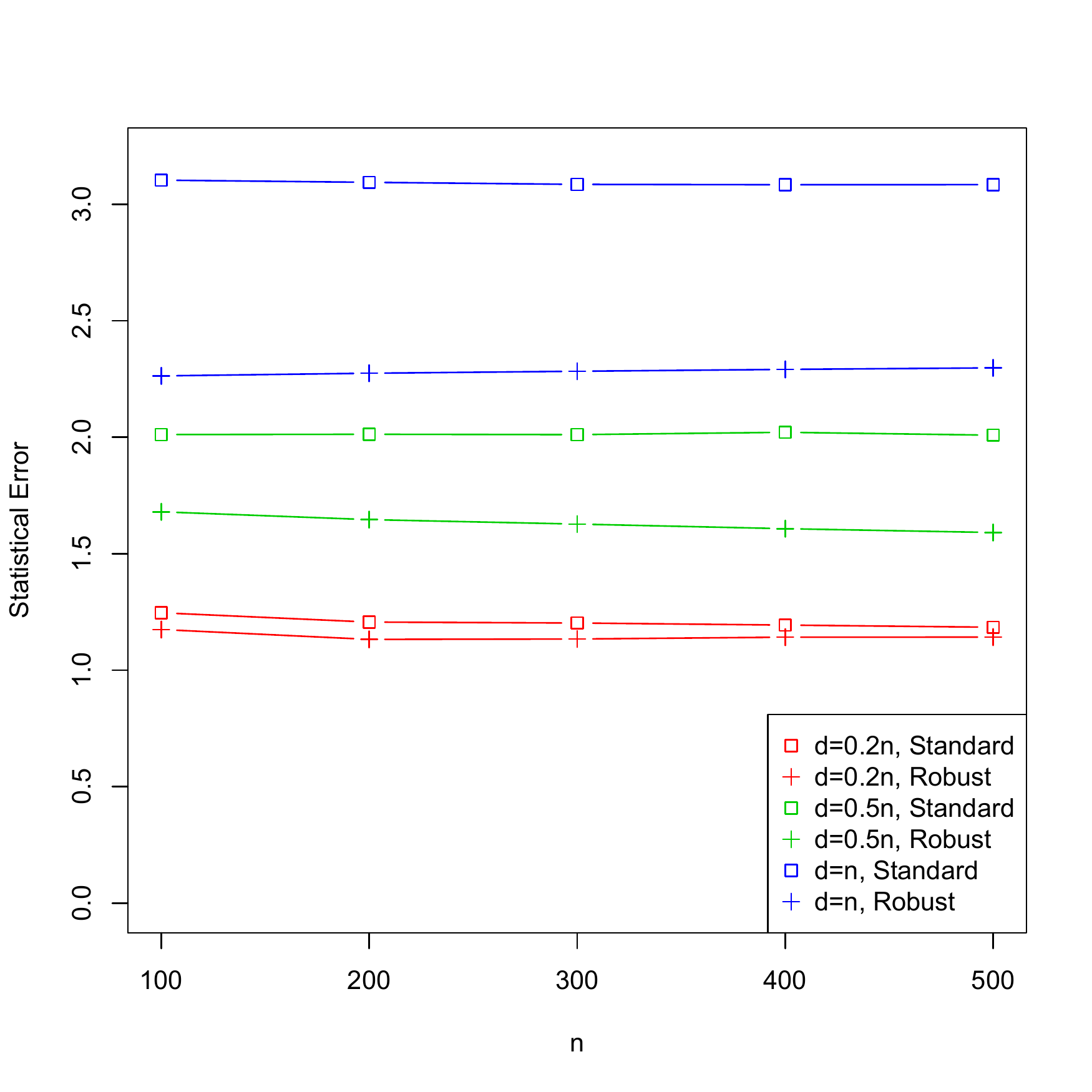} & \includegraphics[height=3in, width=3in]{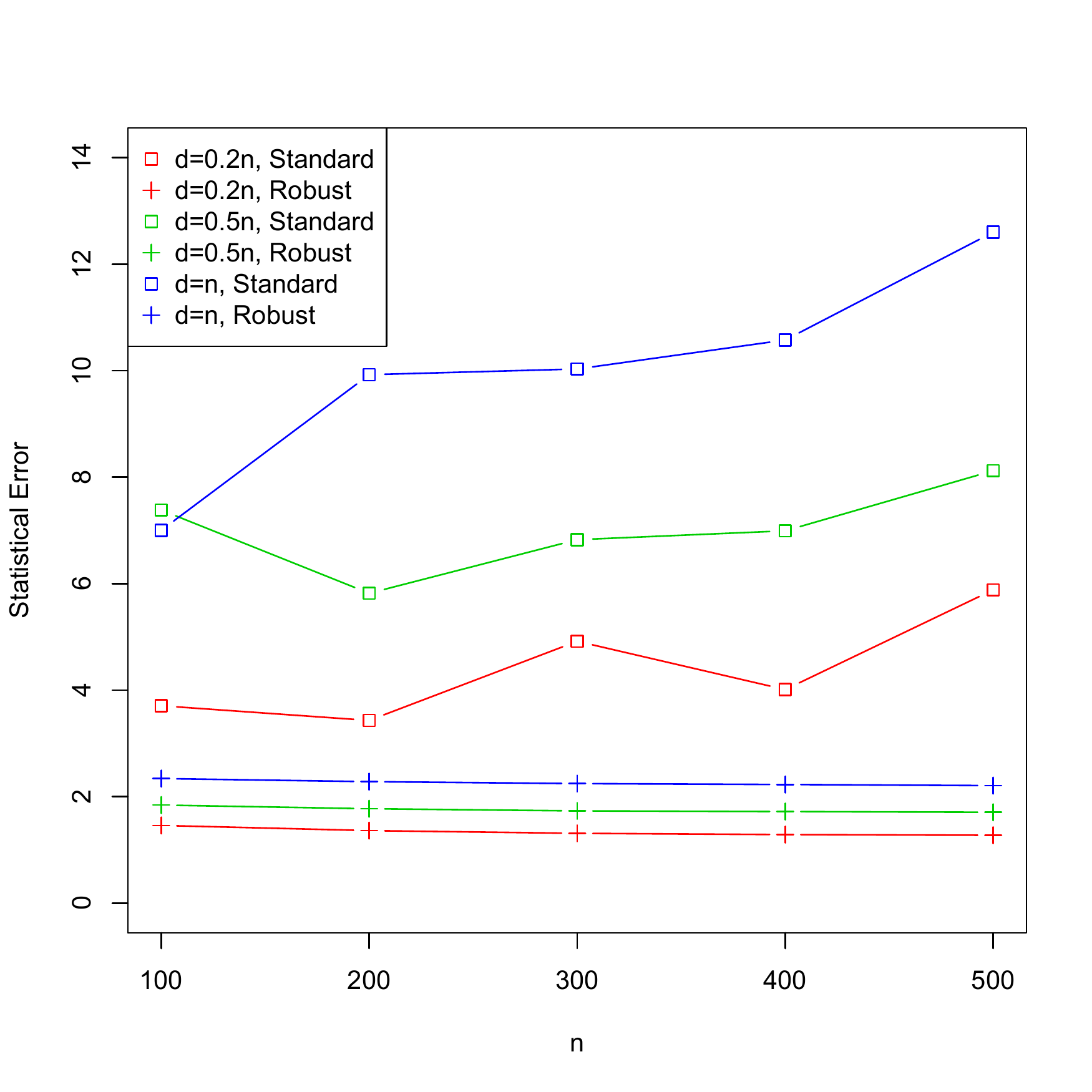} \\
			 Gaussian Samples & $t_3$ Samples
		\end{tabular}
		\caption{Statistical errors of $\opnorm{\widehat\bSigma-\bSigma}$ v.s. sample size $n$ for different dimensions $d$ in covariance estimation.}
		\label{fig:4}
	\end{figure}
	
	As we can see, for Gaussian samples, as long as we fix $d/n$, the statistical error of both $\overline\bSigma_n$ and $\widetilde\bSigma_n(\tau)$ does not change, which is consistent with Theorem 5.39 in \cite{Ver10} and Theorem 6 in our paper. Also, the higher the dimension is, the more significant the superiority of $\widetilde\bSigma_n(\tau)$ is over $\overline\bSigma_n$. This validates our remark after Theorem 6 that the shrinkage ameliorates the impact of dimensionality. Even for Gaussian data, shrinkage is meaningful and provides significant improvement.
	 For $t_3$ distribution, since it is heavy-tailed, the regular sample covariance does not maintain constant statistical error for a fixed $d/n$; instead the error increases as the sample size increases. In contrast, our shrinkage sample covariance still retains stable statistical error and enjoys much higher accuracy than the regular sample covariance. This strongly supports the sharp statistical error rate we derived for $\widetilde\bSigma_n(\tau)$ in Theorem 6.

	\bibliographystyle{ims}
	\bibliography{RobustBib}

\begin{thebibliography}{54}
\expandafter\ifx\csname natexlab\endcsname\relax\def\natexlab#1{#1}\fi
\expandafter\ifx\csname url\endcsname\relax
  \def\url#1{\texttt{#1}}\fi
\expandafter\ifx\csname urlprefix\endcsname\relax\def\urlprefix{URL }\fi

\bibitem[{Bickel and Levina(2008)}]{BLe08}
\textsc{Bickel, P.~J.} and \textsc{Levina, E.} (2008).
\newblock Covariance regularization by thresholding.
\newblock \textit{Annals of Statistics} \textbf{36} 2577--2604.

\bibitem[{Bickel et~al.(2009)Bickel, Ritov and Tsybakov}]{BRT09}
\textsc{Bickel, P.~J.}, \textsc{Ritov, Y.} and \textsc{Tsybakov, A.~B.} (2009).
\newblock Simultaneous analysis of lasso and dantzig selector.
\newblock \textit{Annals of Statistics} \textbf{37} 1705--1732.

\bibitem[{Boucheron et~al.(2013)Boucheron, Lugosi and Massart}]{BLM13}
\textsc{Boucheron, S.}, \textsc{Lugosi, G.} and \textsc{Massart, P.} (2013).
\newblock \textit{Concentration inequalities: A nonasymptotic theory of
  independence}.
\newblock OUP Oxford.

\bibitem[{Brownlees et~al.(2015)Brownlees, Joly and Lugosi}]{BJL15}
\textsc{Brownlees, C.}, \textsc{Joly, E.} and \textsc{Lugosi, G.} (2015).
\newblock Empirical risk minimization for heavy-tailed losses.
\newblock \textit{Annals of Statistics} \textbf{43} 2507--2536.

\bibitem[{Cai and Liu(2011)}]{CLi11}
\textsc{Cai, T.} and \textsc{Liu, W.} (2011).
\newblock Adaptive thresholding for sparse covariance matrix estimation.
\newblock \textit{Journal of the American Statistical Association} \textbf{106}
  672--684.

\bibitem[{Cai and Zhang(2014)}]{CZh14}
\textsc{Cai, T.~T.} and \textsc{Zhang, A.} (2014).
\newblock Sparse representation of a polytope and recovery of sparse signals
  and low-rank matrices.
\newblock \textit{IEEE Transactions on Information Theory} \textbf{60}
  122--132.

\bibitem[{Cai and Zhang(2015)}]{CZh15}
\textsc{Cai, T.~T.} and \textsc{Zhang, A.} (2015).
\newblock Rop: Matrix recovery via rank-one projections.
\newblock \textit{Annals of Statistics} \textbf{43} 102--138.

\bibitem[{Cai and Zhou(2012)}]{CZh12}
\textsc{Cai, T.~T.} and \textsc{Zhou, H.~H.} (2012).
\newblock Optimal rates of convergence for sparse covariance matrix estimation.
\newblock \textit{Annals of Statistics} \textbf{40} 2389--2420.

\bibitem[{Candes and Tao(2007)}]{CTa07}
\textsc{Candes, E.} and \textsc{Tao, T.} (2007).
\newblock The dantzig selector: Statistical estimation when p is much larger
  than n.
\newblock \textit{Annals of Statistics} \textbf{35} 2313--2351.

\bibitem[{Candes(2008)}]{candes2008restricted}
\textsc{Candes, E.~J.} (2008).
\newblock The restricted isometry property and its implications for compressed
  sensing.
\newblock \textit{Comptes Rendus Mathematique} \textbf{346} 589--592.

\bibitem[{Candes and Plan(2010)}]{CPl10}
\textsc{Candes, E.~J.} and \textsc{Plan, Y.} (2010).
\newblock Matrix completion with noise.
\newblock \textit{Proceedings of the IEEE} \textbf{98} 925--936.

\bibitem[{Candes and Plan(2011)}]{CPl11}
\textsc{Candes, E.~J.} and \textsc{Plan, Y.} (2011).
\newblock Tight oracle inequalities for low-rank matrix recovery from a minimal
  number of noisy random measurements.
\newblock \textit{IEEE Transactions on Information Theory} \textbf{57}
  2342--2359.

\bibitem[{Cand{\`e}s and Recht(2009)}]{candes2009exact}
\textsc{Cand{\`e}s, E.~J.} and \textsc{Recht, B.} (2009).
\newblock Exact matrix completion via convex optimization.
\newblock \textit{Foundations of Computational mathematics} \textbf{9}
  717--772.

\bibitem[{Candes and Tao(2006)}]{CTa06}
\textsc{Candes, E.~J.} and \textsc{Tao, T.} (2006).
\newblock Near-optimal signal recovery from random projections: Universal
  encoding strategies?
\newblock \textit{IEEE Transactions on Information Theory} \textbf{52}
  5406--5425.

\bibitem[{Catoni(2012)}]{Cat12}
\textsc{Catoni, O.} (2012).
\newblock Challenging the empirical mean and empirical variance: a deviation
  study.
\newblock \textit{Annales de l'Institut Henri Poincar{\'e}} \textbf{48}
  1148--1185.

\bibitem[{Chen et~al.(2001)Chen, Donoho and Saunders}]{CDS01}
\textsc{Chen, S.~S.}, \textsc{Donoho, D.~L.} and \textsc{Saunders, M.~A.}
  (2001).
\newblock Atomic decomposition by basis pursuit.
\newblock \textit{SIAM review} \textbf{43} 129--159.

\bibitem[{Donoho(2006)}]{Don06}
\textsc{Donoho, D.~L.} (2006).
\newblock Compressed sensing.
\newblock \textit{IEEE Transactions on Information Theory} \textbf{52}
  1289--1306.

\bibitem[{Donoho et~al.(2013)Donoho, Johnstone and
  Montanari}]{donoho2013accurate}
\textsc{Donoho, D.~L.}, \textsc{Johnstone, I.} and \textsc{Montanari, A.}
  (2013).
\newblock Accurate prediction of phase transitions in compressed sensing via a
  connection to minimax denoising.
\newblock \textit{IEEE Transactions on Information Theory} \textbf{59}
  3396--3433.

\bibitem[{Eckstein and Bertsekas(1992)}]{EBe92}
\textsc{Eckstein, J.} and \textsc{Bertsekas, D.~P.} (1992).
\newblock On the douglas---rachford splitting method and the proximal point
  algorithm for maximal monotone operators.
\newblock \textit{Mathematical Programming} \textbf{55} 293--318.

\bibitem[{Fan et~al.(2016)Fan, Li and Wang}]{FLW16}
\textsc{Fan, J.}, \textsc{Li, Q.} and \textsc{Wang, Y.} (2016).
\newblock Robust estimation of high-dimensional mean regression.
\newblock \textit{Journal of Royal Statistical Society, Series B} .

\bibitem[{Fan and Li(2001)}]{FLi01}
\textsc{Fan, J.} and \textsc{Li, R.} (2001).
\newblock Variable selection via nonconcave penalized likelihood and its oracle
  properties.
\newblock \textit{Journal of the American Statistical Association} \textbf{96}
  1348--1360.

\bibitem[{Fan et~al.(2013)Fan, Liao and Mincheva}]{FLM13}
\textsc{Fan, J.}, \textsc{Liao, Y.} and \textsc{Mincheva, M.} (2013).
\newblock Large covariance estimation by thresholding principal orthogonal
  complements.
\newblock \textit{Journal of the Royal Statistical Society: Series B
  (Statistical Methodology)} \textbf{75} 603--680.

\bibitem[{Fan and Lv(2008)}]{fan2008sure}
\textsc{Fan, J.} and \textsc{Lv, J.} (2008).
\newblock Sure independence screening for ultrahigh dimensional feature space.
\newblock \textit{Journal of the Royal Statistical Society: Series B
  (Statistical Methodology)} \textbf{70} 849--911.

\bibitem[{Fan and Yao(2015)}]{FY15}
\textsc{Fan, J.} and \textsc{Yao, Q.} (2015).
\newblock \textit{Elements of Financial Econometrics}.
\newblock Science Press, Beijing.

\bibitem[{Fang et~al.(2015)Fang, Liu, Toh and Zhou}]{FLT15}
\textsc{Fang, X.~E.}, \textsc{Liu, H.}, \textsc{Toh, K.~C.} and \textsc{Zhou,
  W.-X.} (2015).
\newblock Max-norm optimization for robust matrix recovery.
\newblock Tech. rep.

\bibitem[{Gupta et~al.(2014)Gupta, Ellis, Ashar, Moes, Bader, Zhan, West and
  Arking}]{gupta2014transcriptome}
\textsc{Gupta, S.}, \textsc{Ellis, S.~E.}, \textsc{Ashar, F.~N.}, \textsc{Moes,
  A.}, \textsc{Bader, J.~S.}, \textsc{Zhan, J.}, \textsc{West, A.~B.} and
  \textsc{Arking, D.~E.} (2014).
\newblock Transcriptome analysis reveals dysregulation of innate immune
  response genes and neuronal activity-dependent genes in autism.
\newblock \textit{Nature communications} \textbf{5}.

\bibitem[{He et~al.(2014)He, Liu, Wang and Yuan}]{HLW14}
\textsc{He, B.}, \textsc{Liu, H.}, \textsc{Wang, Z.} and \textsc{Yuan, X.}
  (2014).
\newblock A strictly contractive peaceman--rachford splitting method for convex
  programming.
\newblock \textit{SIAM Journal on Optimization} \textbf{24} 1011--1040.

\bibitem[{Hsu and Sabato(2016)}]{HSa16}
\textsc{Hsu, D.} and \textsc{Sabato, S.} (2016).
\newblock Loss minimization and parameter estimation with heavy tails.
\newblock \textit{Journal of Machine Learning Research} \textbf{17} 1--40.

\bibitem[{Ibragimov and Sharakhmetov(1998)}]{ISh98}
\textsc{Ibragimov, R.} and \textsc{Sharakhmetov, S.} (1998).
\newblock Short communications: On an exact constant for the rosenthal
  inequality.
\newblock \textit{Theory of Probability \& Its Applications} \textbf{42}
  294--302.

\bibitem[{Kim and Xing(2010)}]{kim2010tree}
\textsc{Kim, S.} and \textsc{Xing, E.~P.} (2010).
\newblock Tree-guided group lasso for multi-task regression with structured
  sparsity.
\newblock \textit{Annals of Applied Statistics} \textbf{6} 1095--1117.

\bibitem[{Koltchinskii et~al.(2011)Koltchinskii, Lounici and Tsybakov}]{KLT11}
\textsc{Koltchinskii, V.}, \textsc{Lounici, K.} and \textsc{Tsybakov, A.~B.}
  (2011).
\newblock Nuclear-norm penalization and optimal rates for noisy low-rank matrix
  completion.
\newblock \textit{Annals of Statistics} \textbf{39} 2302--2329.

\bibitem[{Lam and Fan(2009)}]{LFa09}
\textsc{Lam, C.} and \textsc{Fan, J.} (2009).
\newblock Sparsistency and rates of convergence in large covariance matrix
  estimation.
\newblock \textit{Annals of Statistics} \textbf{37} 4254--4278.

\bibitem[{Loh(2015)}]{Loh15}
\textsc{Loh, P.-L.} (2015).
\newblock Statistical consistency and asymptotic normality for high-dimensional
  robust m-estimators.
\newblock \textit{arXiv preprint arXiv:1501.00312} .

\bibitem[{Meinshausen and B{\"u}hlmann(2006)}]{MBu06}
\textsc{Meinshausen, N.} and \textsc{B{\"u}hlmann, P.} (2006).
\newblock High-dimensional graphs and variable selection with the lasso.
\newblock \textit{Annals of Statistics} \textbf{34} 1436--1462.

\bibitem[{Minsker(2015)}]{Min15}
\textsc{Minsker, S.} (2015).
\newblock Geometric median and robust estimation in banach spaces.
\newblock \textit{Bernoulli} \textbf{21} 2308--2335.

\bibitem[{Negahban and Wainwright(2011)}]{NWa11}
\textsc{Negahban, S.} and \textsc{Wainwright, M.~J.} (2011).
\newblock Estimation of (near) low-rank matrices with noise and
  high-dimensional scaling.
\newblock \textit{Annals of Statistics} \textbf{39} 1069--1097.

\bibitem[{Negahban and Wainwright(2012)}]{NWa12}
\textsc{Negahban, S.} and \textsc{Wainwright, M.~J.} (2012).
\newblock Restricted strong convexity and weighted matrix completion: Optimal
  bounds with noise.
\newblock \textit{Journal of Machine Learning Research} \textbf{13} 1665--1697.

\bibitem[{Negahban et~al.(2012)Negahban, Yu, Wainwright and Ravikumar}]{NRW12}
\textsc{Negahban, S.}, \textsc{Yu, B.}, \textsc{Wainwright, M.~J.} and
  \textsc{Ravikumar, P.~K.} (2012).
\newblock A unified framework for high-dimensional analysis of $ m $-estimators
  with decomposable regularizers.
\newblock \textit{Statistical Science} \textbf{24} 538--577.

\bibitem[{Nemirovsky et~al.(1982)Nemirovsky, Yudin and Dawson}]{NYD82}
\textsc{Nemirovsky, A.-S.}, \textsc{Yudin, D.-B.} and \textsc{Dawson, E.-R.}
  (1982).
\newblock Problem complexity and method efficiency in optimization .

\bibitem[{Nowak et~al.(2007)Nowak, Wright et~al.}]{nowak2007gradient}
\textsc{Nowak, R.~D.}, \textsc{Wright, S.~J.} \textsc{et~al.} (2007).
\newblock Gradient projection for sparse reconstruction: Application to
  compressed sensing and other inverse problems.
\newblock \textit{IEEE Journal of Selected Topics in Signal Processing}
  \textbf{1} 586--597.

\bibitem[{Oliveira(2016)}]{Oli16}
\textsc{Oliveira, R.~I.} (2016).
\newblock The lower tail of random quadratic forms with applications to
  ordinary least squares.
\newblock \textit{Probability Theory and Related Fields}  1--20.

\bibitem[{Raskutti et~al.(2011)Raskutti, Wainwright and Yu}]{RWY11}
\textsc{Raskutti, G.}, \textsc{Wainwright, M.~J.} and \textsc{Yu, B.} (2011).
\newblock Minimax rates of estimation for high-dimensional linear regression
  over-balls.
\newblock \textit{IEEE Transactions on Information Theory} \textbf{57}
  6976--6994.

\bibitem[{Recht(2011)}]{recht2011simpler}
\textsc{Recht, B.} (2011).
\newblock A simpler approach to matrix completion.
\newblock \textit{Journal of Machine Learning Research} \textbf{12} 3413--3430.

\bibitem[{Recht et~al.(2010)Recht, Fazel and Parrilo}]{RFP10}
\textsc{Recht, B.}, \textsc{Fazel, M.} and \textsc{Parrilo, P.~A.} (2010).
\newblock Guaranteed minimum-rank solutions of linear matrix equations via
  nuclear norm minimization.
\newblock \textit{SIAM review} \textbf{52} 471--501.

\bibitem[{Rohde and Tsybakov(2011)}]{RTs11}
\textsc{Rohde, A.} and \textsc{Tsybakov, A.~B.} (2011).
\newblock Estimation of high-dimensional low-rank matrices.
\newblock \textit{Annals of Statistics} \textbf{39} 887--930.

\bibitem[{Rudelson and Zhou(2013)}]{RZh13}
\textsc{Rudelson, M.} and \textsc{Zhou, S.} (2013).
\newblock Reconstruction from anisotropic random measurements.
\newblock \textit{IEEE Transactions on Information Theory} \textbf{59}
  3434--3447.

\bibitem[{Stock and Watson(2002)}]{stock2002macroeconomic}
\textsc{Stock, J.~H.} and \textsc{Watson, M.~W.} (2002).
\newblock Macroeconomic forecasting using diffusion indexes.
\newblock \textit{Journal of Business \& Economic Statistics} \textbf{20}
  147--162.

\bibitem[{Tibshirani(1996)}]{Tib96}
\textsc{Tibshirani, R.} (1996).
\newblock Regression shrinkage and selection via the lasso.
\newblock \textit{Journal of the Royal Statistical Society: Series B
  (Methodological)} \textbf{58} 267--288.

\bibitem[{Tropp(2012)}]{Tro12}
\textsc{Tropp, J.~A.} (2012).
\newblock User-friendly tools for random matrices: An introduction.
\newblock Tech. rep., DTIC Document.

\bibitem[{Tropp(2015)}]{Tro15}
\textsc{Tropp, J.~A.} (2015).
\newblock An introduction to matrix concentration inequalities.
\newblock \textit{arXiv preprint arXiv:1501.01571} .

\bibitem[{Velu and Reinsel(2013)}]{velu2013multivariate}
\textsc{Velu, R.} and \textsc{Reinsel, G.~C.} (2013).
\newblock \textit{Multivariate reduced-rank regression: theory and
  applications}, vol. 136.
\newblock Springer Science \& Business Media.

\bibitem[{Vershynin(2010)}]{Ver10}
\textsc{Vershynin, R.} (2010).
\newblock Introduction to the non-asymptotic analysis of random matrices.
\newblock \textit{arXiv preprint arXiv:1011.3027} .

\bibitem[{Zhang(2010)}]{zhang2010nearly}
\textsc{Zhang, C.-H.} (2010).
\newblock Nearly unbiased variable selection under minimax concave penalty.
\newblock \textit{Annals of statistics} \textbf{38} 894--942.

\bibitem[{Zou and Li(2008)}]{zou2008one}
\textsc{Zou, H.} and \textsc{Li, R.} (2008).
\newblock One-step sparse estimates in nonconcave penalized likelihood models.
\newblock \textit{Annals of Statistics} \textbf{36} 1509--1533.

\end{thebibliography}
	
	\appendix
	\section{Proofs}

	\begin{proof}[\underline{Proof of Theorem \ref{thm:1}}]
		This Lemma is just a simple application of the theoretical framework established in \cite{NRW12}, but for completeness and clarity, we present the whole proof here. We start from the optimality of $\widehat\bTheta$:
		\[
			-\inn{\widehat\bSigma_{Y\bX}, \widehat\bTheta}+\frac{1}{2}\vec(\widehat\bTheta)^T\widehat\bSigma_{\bX\bX}\vec(\widehat\bTheta)+\lambda_N\nnorm{\widehat\bTheta}\le -\inn{\widehat\bSigma_{Y\bX}, \bTheta^*}+\frac{1}{2}\vec(\bTheta^*)^T\widehat\bSigma_{\bX\bX}\vec(\bTheta^*)+\lambda_N\nnorm{\bTheta^*}.
		\]
		Note that $\widehat\bDelta=\widehat\bTheta-\bTheta^*$. Simple algebra delivers that
		\beq
			\label{eq:lm1start}
			\begin{aligned}
				\frac{1}{2}\vec(\widehat\bDelta)^T\widehat\bSigma_{\bX\bX}\vec(\widehat\bDelta) & \le \inn{\widehat\bSigma_{Y\bX} - \mat(\widehat\bSigma_{\bX\bX}\vec(\bTheta^*)), \widehat\bDelta}+\lambda_N\nnorm{\widehat\bDelta} \\
				& \le \opnorm{\widehat\bSigma_{Y\bX} - \mat(\widehat\bSigma_{\bX\bX}\vec(\bTheta^*))}\cdot\nnorm{\widehat\bDelta}+\lambda_N\nnorm{\widehat\bDelta} \le 2\lambda_N\nnorm{\widehat\bDelta},
			\end{aligned}
		\eeq
	if $\lambda_N\ge 2\opnorm{\widehat\bSigma_{Y\bX} - \mat(\widehat\bSigma_{\bX\bX}\vec(\bTheta^*))}$. To bound the RHS of (\ref{eq:lm1start}), we need to decompose $\widehat\bDelta$ as \cite{NWa11} did. Let $\bTheta^*=\bU\bD\bV^T$ be the SVD of $\bTheta^*$, where the diagonals of $\bD$ are in the decreasing order. Denote the first $r$ columns of $\bU$ and $\bV$ by $\bU^r$ and $\bV^r$ respectively, and define
	\beq
		\begin{aligned}
			& \cM:=\{\bTheta\in \Rdd\ |\ \text{row}(\bTheta)\subseteq col(\bV^r), \text{col}(\bTheta)\subseteq col(\bU^r)\}, \\
			& \overline{\cM}^{\perp}:=\{\bTheta\in \Rdd\ |\ \text{row}(\bTheta)\perp col(\bV^r), \text{col}(\bTheta)\perp col(\bU^r)\},
		\end{aligned}
	\eeq
	where $col(\cdot)$ and $row(\cdot)$ denote the column space and row space respectively. For any $\bDelta\in \Rdd$ and Hilbert space $\cW\subseteq\Rdd$, let $\bDelta_{\cW}$ be the projection of $\bDelta$ onto $\cW$. We first clarify here what $\bDelta_\cM$, $\bDelta_{\overline\cM}$ and $\bDelta_{\overline\cM^{\perp}}$ are. Write $\bDelta$ as
	\[
		\bDelta=[\bU^r, \bU^{r^\perp}]\left[
			\begin{array}{cc}
				\bGamma_{11} & \bGamma_{12} \\
				\bGamma_{21} & \bGamma_{22}
			\end{array} \right] [\bV^r, \bV^{r^\perp}]^T,
	\]
	 then the following equalities hold:
	 \beq
	 	\begin{aligned}
			\bDelta_{\cM} =\bU^r\bGamma_{11}{(\bV^r)}^T,\quad\bDelta_{\overline\cM^\perp} = \bU^{r^\perp} \bGamma_{22}(\bV^{r^\perp})^{T}, \quad \bDelta_{\overline\cM} = [\bU^r, \bU^{r^\perp}]\left[
			\begin{array}{cc}
				\bGamma_{11} & \bGamma_{12} \\
				\bGamma_{21} & \bzero
			\end{array} \right] [\bV^r, \bV^{r^\perp}]^T.
		\end{aligned}
	 \eeq
	Applying Lemma 1 in \cite{NRW12} to our new loss function implies that if $\lambda_N \ge 2\opnorm{\widehat\bSigma_{Y\bX} - \mat(\widehat\bSigma_{\bX\bX}\vec(\bTheta^*))}$, it holds that
		\beq
			\label{eq:cone}
			\nnorm{\widehat\bDelta_{\overline\cM^{\perp}}}\le 3\nnorm{\widehat\bDelta_{\overline\cM}}+ 4\sum\nolimits_{j\ge r+1} \sigma_j(\bTheta^*).
		\eeq
		Note that $\rank(\widehat\bDelta_{\overline\cM})\le 2r$; we thus have
		\beq
			\label{eq:6.5}
			\nnorm{\widehat\bDelta}\le \nnorm{\widehat\bDelta_{\overline\cM}}+\nnorm{\widehat\bDelta_{\overline\cM^{\perp}}}\le 4\nnorm{\widehat\bDelta_{\overline\cM}}+4\sum_{j\ge r+1}\sigma_j(\bTheta^*)\le 4\sqrt{2r}\fnorm{\widehat\bDelta}+4\sum_{j\ge r+1}\sigma_j(\bTheta^*).
		\eeq
		Following the proof of Corollary 2 in \cite{NWa11}, we determine the value of $r$ here. For a threshold $\tau>0$, we choose
		\[
			r=\# \{j \in \{1, 2, ..., (d_1\wedge d_2)\} | \sigma_j(\bTheta^*) \ge \tau\}.
		\]
		Then it follows that
		\beq
			\label{eq:6.6}
			\sum\limits_{j \ge r+1} \sigma_j(\bTheta^*)\le \tau \sum\limits_{j \ge r+1} \frac{\sigma_j(\bTheta^*)}{\tau} \le \tau \sum\limits_{j \ge r+1} \bigl(\frac{\sigma_j(\bTheta^*)}{\tau}\bigr)^q \le \tau^{1-q}\sum\limits_{j \ge r+1}\sigma_j(\bTheta^*)^q \le \tau^{1-q}\rho.
		\eeq
		On the other hand, $\rho \ge \sum\limits_{j \le r} \sigma_j(\bTheta^*)^q \ge r\tau^q$, so $r \le \rho\tau^{-q}$. Combining (\ref{eq:lm1start}), (\ref{eq:6.5}) and $\vec(\widehat\bDelta)^T\widehat\bSigma_{\bX\bX}\vec(\widehat\bDelta) \ge \kappa_{\cL}\fnorm{\widehat \bDelta}^2$, we have
		\[
			\frac{1}{2}\kappa_{\cL} \fnorm{\widehat\bDelta}^2 \le 2\lambda_N(4\sqrt{2r}\fnorm{\widehat\bDelta}+4\tau^{1-q}\rho),
		\]
		which implies that
		\[	
			\fnorm{\widehat\bDelta}\le 4\sqrt{\frac{\lambda_N\rho}{\kappa_{\cL}}}\Bigl(\sqrt{\frac{32\lambda_N\tau^{-q}}{\kappa_{\cL}}}+\sqrt{\tau^{1-q}}\Bigr).
		\]
		Choosing $\tau=\lambda_N/\kappa_{\cL}$, we have for some constant $C_1$,
		\[
			\fnorm{\widehat\bDelta} \le C_1\sqrt{\rho}\Bigl(\frac{\lambda_N}{\kappa_{\cL}}\Bigr)^{1-\frac{q}{2}}.
		\]
		Combining this result with \eqref{eq:6.5} and \eqref{eq:6.6},  we can further derive the statistical error rate in terms of the nuclear norm as follows.
		\[
			\nnorm{\widehat\bDelta} \le C_2\rho\Bigl(\frac{\lambda_N}{\kappa_{\cL}}\Bigr)^{1-q},
		\]
		where $C_2$ is certain positive constant.
		\end{proof}

	\begin{proof}[\underline{Proof of Lemma \ref{lem1}}]
	
	(a) We first prove for the case of the sub-Gaussian design. Recall that we use $\widetilde Y_i$ to denote $sgn(Y_i)(|Y_i| \wedge \tau)$ and $\widetilde\bx_i = \bx_i$ in this case. Let $\hat\sigma_{x_j\widetilde Y}(\tau) = \frac{1}{N}\sum\nolimits_{i=1}^N \widetilde Y_i x_{ij}$. Note that
	\beq
		\label{eq:lem1var}
		\begin{aligned}	
			\Var(\widetilde Y_i x_{ij}) & \le \E (\widetilde Y_i^2x_{ij}^2) \le \E (Y_i^2x_{ij}^2) \le (\E Y_i^{2k})^{\frac{1}{k}}(\E x_{ij}^{\frac{2k}{k-1}})^{\frac{k-1}{k}} \le 2M^{\frac{1}{k}}\kappa_0^2k/(k-1)< \infty\,,
		\end{aligned}
	\eeq
	which is a constant that we denote by $v_1$. In addition,  for $p>2$,
	\[
		\begin{aligned}
		 	\E |\widetilde Y_i x_{ij}|^p & \le \tau^{p-2}\E (Y_i^2|x_{ij}|^p) \le \tau^{p-2}M^{\frac{1}{k}}\bigl(\E |x_{ij}|^{\frac{kp}{k-1}}\bigr)^{\frac{k-1}{k}}  \le \tau^{p-2}M^{\frac{1}{k}}\Bigl(\kappa_0\sqrt{\frac{kp}{k-1}}\Bigr)^p.
		\end{aligned}
	\]
	By the Jensen's inequality and then Stirling approximation, it follows that for some constants $c_1$ and $c_2$,
	\[
		\begin{aligned}
			\E |\widetilde Y_ix_{ij}- \E \widetilde Y_ix_{ij}|^p & \le 2^{p-1}(\E |\widetilde Y_i x_{ij}|^p + |\E \widetilde Y_ix_{ij}|^{p}) \le 2^{p-1}(\E | Y_i x_{ij}|^p + (\E |Y_ix_{ij}|)^{p}) \\
			& \le \frac{c_1p!(c_2\tau)^{p-2}}{2}.
		\end{aligned}
	\]
	 Define $v:= c_1 \vee v_1$. According to Bernstein's Inequality (Theorem 2.10 in \cite{BLM13}), we have for $j=1,...,d$,
	\[
		P \Big(|\widehat\sigma_{x_j\widetilde Y}(\tau)-\E(\widetilde Y_ix_{ij})|\ge \sqrt{\frac{2vt}{N}}+\frac{c_2\tau t}{N}\Big)\le 2\exp(-t).
	\]
	Also note that by Markov's inequality,
	\[
		\E ((\widetilde Y_i-Y_i)x_{ij})  \le \E (|Y_ix_{ij}|\cdot \ind_{|Y_i|>\tau}) \le \sqrt{\E(Y_i^2 x_{ij}^2)P(|Y_i|>\tau)} \le \sqrt{\frac{v\E Y_i^2}{\tau^2}} \le \frac{\sqrt{vM^{\frac{1}{k}}}}{\tau}.
	\]
	Note that since $\Var(Y_i)= \bbeta^{*T}\bSigma_{\bx\bx}\bbeta^*+ \E \epsilon^2 \le M^{\frac{1}{k}}$ and $\lambda_{\min}(\bSigma_{\bx\bx})\ge \kappa_{\cL}$, $\ltwonorm{\bbeta^*} \le M^{\frac{1}{k}}/\kappa_{\cL}$. Therefore, $\|\bx_i^T\bbeta^*\|_{\psi_2}\le M^{\frac{1}{k}}\kappa_0/\kappa_{\cL}$ and $\|x_{ij}\bx_i^T\bbeta^*\|_{\psi_1}\le 2M^{\frac{1}{k}}\kappa_0^2/\kappa_{\cL}$. By Proposition 5.16 (Bernstein-type inequality) in \cite{Ver10}, we have for sufficiently small $t$, 
	\[
		P\Big(|(\hat\bSigma_{x_j\bx})^T\bbeta^*-\E(Y_ix_{ij})|\ge c_2\sqrt{\frac{t}{N}}\Big)\le \exp(-t)
	\]
	for some constant $c_2$. Choose $\tau \asymp \sqrt{N/\log d}$. An application of the triangle inequality and the union bound yields that as long as $\log d/N < \gamma_1$ for certain $\gamma_1$, we have for some constant $\nu_1>0$,
	\[
		P\Big(  \supnorm{\widehat\bSigma_{Y\bx}(\tau)-\frac{1}{N}\sum\limits_{i=1}^N (\bx_i^T\btheta^*)\bx_i} \ge \nu_1\sqrt{\frac{\delta\log d}{N}}\Big)\le 2d^{1-\delta}.
	\]
	
(b) Now we switch to the case where both the noise and the design have only bounded moments.  Note that
		\[
			\begin{aligned}
				\supnorm{\widehat\bSigma_{\widetilde Y\widetilde \bx}- \widehat\bSigma_{\widetilde\bx\widetilde\bx}\btheta^*} & \le \supnorm{\widehat\bSigma_{\widetilde Y\widetilde \bx}- \bSigma_{\widetilde Y \widetilde \bx}}+ \supnorm{\bSigma_{\widetilde Y \widetilde \bx}- \bSigma_{Y\bx}}+ \supnorm{\bSigma_{Y\bx}- \bSigma_{\widetilde\bx\widetilde\bx}\btheta^*} \\
				& + \supnorm{(\widehat\bSigma_{\widetilde\bx\widetilde\bx}- \bSigma_{\widetilde\bx\widetilde\bx})\btheta^*}= T_1+T_2+T_3+T_4.
			\end{aligned}
		\]
		We bound the four terms one by one. For $1\le j\le d$, analogous to \eqref{eq:lem1var},
		\[	
			\begin{aligned}
				\Var(\widetilde Y_i\widetilde x_{ij}) & \le \E ( \widetilde Y_i \widetilde x_{ij})^2 \le \E (Y_ix_{ij})^2 \le \sqrt{\E Y_i^4 \E x_{ij}^4} =:v_1< \infty.
			\end{aligned}
		\]
		In addition, $\E |\widetilde Y_i\widetilde x_{ij}|^p \le (\tau_1\tau_2)^{p-2}v_1$. Therefore according to Bernstein's Inequality (Theorem 2.10 in \cite{BLM13}), we have
		\[
			P\Bigl(|\widehat\sigma_{\widetilde Y\widetilde x_j}- \sigma_{\widetilde Y\widetilde x_j}|\ge \sqrt{\frac{2v_1t}{N}}+\frac{c\tau_1\tau_2 t}{N} \Bigr)\le \exp(-t),
		\]
		where $\widehat\sigma_{\widetilde Y\widetilde x_j}= \frac{1}{N} \sum\nolimits_{i=1}^N \widetilde Y_i\widetilde x_{ij}$, $\sigma_{\widetilde Y\widetilde x_j}= \E \widetilde Y_i\widetilde x_{ij}$ and $c$ is certain constant. Then by the union bound, we have
		\[
			P \Bigl(|T_1|> \sqrt{\frac{2v_1t}{N}}+ \frac{c\tau_1\tau_2t}{N} \Bigr) \le d\exp(-t).
		\]
		Next we bound $T_2$. Note that for $1 \le j \le d$,
		\[
			\begin{aligned}
				\E \widetilde Y_i\widetilde x_{ij}- \E Y_ix_{ij} & = \E \widetilde Y_i\widetilde x_{ij}- \E \widetilde Y_i x_{ij}+ \E \widetilde Y_i x_{ij}- \E Y_i x_{ij}= \E \widetilde Y_i(\widetilde x_{ij}- x_{ij})+ \E (\widetilde Y_i-Y_i)x_{ij} \\
				& \le  \sqrt{\E \bigl(Y_i^2(\widetilde x_{ij}- x_{ij})^2\bigr) P(|x_{ij}|\ge \tau_2) }+ \sqrt{\E \bigl((\widetilde Y_i- Y_i)^2 x_{ij}^2\bigr) P(|Y_i|\ge \tau_1)} \\
				& \le \sqrt{Mv_1}\Big(\frac{1}{\tau_2^2}+ \frac{1}{\tau_1^2}\Big),
			\end{aligned}
		\]
		which delivers that $T_2\le \sqrt{Mv_1}(1/\tau_1^2+ 1/\tau_2^2)$. Then we bound $T_3$. For $1\le j \le d$,
		\[
					\begin{aligned}
			\E & \bigl((\widetilde \bx_i^T\btheta^*)\widetilde x_{ij}\bigr)- \E\bigl((\bx_i^T\btheta^*)x_{ij}\bigr) = \sum\limits_{k=1}^d \E \bigl(\theta^*_k(\widetilde x_{ik}\widetilde x_{ij}- x_{ik}x_{ij})\bigr )\le \sum\limits_{k=1}^d |\theta^*_k|\E |\widetilde x_{ij}\widetilde x_{ik}- x_{ij}x_{ik}| \\
			& \le \sum\limits_{k=1}^d |\theta^*_k| \bigl(\E|  x_{ij}(\widetilde x_{ik}- x_{ik})|+ \E |(\widetilde x_{ij}- x_{ij})x_{ik}|\bigr )  \\
			& \le \sum\limits_{k=1}^d |\theta^*_k| \Bigl(\E \bigl(|x_{ij}(\widetilde x_{ik}- x_{ik})|\ind_{\{|x_{ik}|>\tau_2\}}\bigr)+ \E \bigl(|(\widetilde x_{ij}- x_{ij})x_{ik}|\ind_{\{|x_{ij}|>\tau_2\}}\bigr)\Bigr ) \le \frac{CR}{\tau_2^2}.
			\end{aligned}
		\]
		Finally we bound $T_4$. For $1\le j,k\le d$, we have $|\widetilde x_{ik}\widetilde x_{ij}| \le \tau_2^2$ and $\Var( \widetilde x_{ij}\widetilde x_{ik})\le M=: v_2$. Therefore according to Bernstein's inequality (Theorem 2.10 in \cite{BLM13}),
		\beq
			\label{eq:6.8}
			P\Bigl(|\widehat\sigma_{\widetilde x_j\widetilde x_k}- \sigma_{\widetilde x_j \widetilde x_k}|\ge \sqrt{\frac{2v_2t}{N}}+\frac{c\tau_2^2 t}{N} \Bigr)\le \exp(-t),
		\eeq
		where $\widehat\sigma_{\widetilde x_j\widetilde x_k}= (1/N)\sum\nolimits_{i=1}^N \widetilde x_{ij}\widetilde x_{ik}$, $\sigma_{\widetilde x_j\widetilde x_k}= \E \widetilde x_{ij}\widetilde x_{ik}$ and $c$ is certain constant. We therefore have
		\[
			P\Bigl(|T_4|\ge R\bigl(\sqrt{\frac{2v_2t}{N}}+\frac{c\tau_2^2 t}{N}\bigr) \Bigr)\le d^2\exp(-t).
		\]
		Now we choose $\tau_1, \tau_2\asymp (N / \log d)^{1/4}$, then it follows that for some constant $\nu_2>0$,
		\[
			P\Bigl(\|\widehat\bSigma_{\widetilde Y\widetilde \bx}- \widehat\bSigma_{\widetilde \bx\widetilde \bx}\btheta^*\|_{\max}>\nu_2\sqrt{\frac{\delta\log d}{N}} \Bigr)\le 2d^{1-\delta}.
		\]
	\end{proof}

	\begin{proof}[\underline{Proof of Lemma \ref{lem:2}}]
		{\bf (a)} We first consider the sub-Gaussian design. Again, since $\widetilde \bx_i=\bx_i$ in this case, we do not add any tilde above $\bx_i$ in this proof. Define $\widehat \bD_{\bx\bx}$ to be the diagonal of $\widehat\bSigma_{\bx\bx}$. Note that given $\|\bx_i\|_{\psi_2} \le \kappa_0$ and $\lambda_{\min}(\bSigma_{\bx\bx})\ge \kappa_{\cL}>0$, it follows that $$\forall \bv\in \cS^{d-1}, \sqrt{\E (\bv^T\bx_i)^4} \le 4\kappa_0^2 \le (4\kappa_0^2/\kappa_{\cL})\bv^T\bSigma_{\bx\bx}\bv.$$ In addition, since $\|x_{ij}\|_{\psi_2}\le \kappa_0$, $(\E x_{ij}^8)^\frac{1}{4}\le 8\kappa_0^2\le (8\kappa_0^2/\kappa_{\cL})\E x_{ij}^2 $. According to Lemma 5.2 in \cite{Oli16}, for any $\eta_1>1$,
		\beq
			\label{eq:6.28}
			P\Bigl( \forall\ \bv \in \RR^d: \bv^T\widehat \bSigma_{\bx\bx}\bv\ge \frac{1}{2}\bv^T\bSigma_{\bx\bx}\bv-\frac{c_1(1+2\eta_1\log d)}{N}\lonenorm{\widehat\bD^{1/2}_{\bx\bx}\bv}^2  \Bigr)\ge 1-\frac{d^{1-\eta_1}}{3},
		\eeq
		where $c_1$ is some universal constant. For any $1\le j \le d$, we know that $\|x_{ij}\|_{\psi_1}\le 2\|x_{ij}\|_{\psi_2}^2= 2\kappa_0^2$, therefore by the Bernstein-type inequality we have for sufficiently small $t$,
		\[
			P\Bigl(\bigl | \frac{1}{N}\sum\limits_{i=1}^n x_{ij}^2- \E x_{ij}^2 \bigr |\ge t\Bigr)\le 2\exp( -cNt^2),
		\]
		where $c$ depends on $\kappa_0$. Note that $\E x_{ij}^2\le 2\kappa_0^2$. An application of the union bound delivers that when $t$ is sufficiently small,
		\[
			P(\opnorm{\widehat\bD_{\bx\bx}}\ge 2\kappa_0^2+t)\le 2d\exp(-cNt^2).
		\]
		Let $t<\kappa_0^2$, then the inequality above yields that for a new constant $C_2>0$,
		\[
			P(\opnorm{\widehat\bD_{\bx\bx}}\ge 3\kappa_0^2)\le 2d\exp(-C_2N).
		\]
		Combining the inequality above with \eqref{eq:6.28}, it follows for a new constant $C_1>0$ that
		\[
			P\Bigl( \forall\ \bv \in \RR^d: \bv^T\widehat \bSigma_{\bx\bx}\bv\ge \frac{1}{2}\bv^T\bSigma_{\bx\bx}\bv-\frac{C_1(1+2\eta_1\log d)}{N}\lonenorm{\bv}^2  \Bigr)\ge 1-\frac{d^{1-\eta_1}}{3}- 2d\exp(-C_2N).
		\]
		
		{\bf (b)} Now we switch to the case of designs with only bounded moments. We will show that the sample covariance of the truncated design has the restricted eigenvalue property. Recall that for $i=1,...,N, j=1,...,d$, $\widetilde x_{ij}=sgn(x_{ij})(|x_{ij}|\wedge \tau_2)$. Note that
		\[
			\bv^T\widehat \bSigma_{\widetilde \bx\widetilde \bx} \bv= \bv^T (\widehat\bSigma_{\widetilde \bx\widetilde \bx}- \bSigma_{\widetilde\bx\widetilde\bx}) \bv+ \bv^T(\bSigma_{\widetilde \bx\widetilde \bx}- \bSigma_{\bx\bx}) \bv+ \bv^T\bSigma_{\bx\bx} \bv.
		\]
		According to \eqref{eq:6.8}, we have for some $c_1>0$,
		\[
			P\Bigl(\supnorm{\widehat\bSigma_{\widetilde \bx\widetilde \bx}- \bSigma_{\widetilde\bx \widetilde \bx}}\ge \sqrt{\frac{2Mt}{N}}+\frac{c_1\tau_2^2 t}{N} \Bigr)\le d^2\exp(-t).
		\]
		Given $\tau_2\asymp (N/\log d)^{1/4}$, we have for some constant $c_2>0$ and any $\eta_2>2$,
		\[
			P\Bigl( \supnorm{\widehat\bSigma_{\widetilde \bx\widetilde \bx}- \bSigma_{\widetilde \bx\widetilde \bx}} \ge c_2\sqrt{\frac{\eta_2\log d}{N}}\Bigr) \le d^{2-\eta_2}.
		\]
		In addition, for any $1 \le j_1, j_2 \le d$, we have
		\[	
			\begin{aligned}
				\E x_{ij_1}x_{ij_2}-\E \widetilde x_{ij_1}\widetilde x_{ij_2} & \le \E \bigl(|x_{ij_1}x_{ij_2}|(\ind_{\{|x_{ij_1}| \ge \tau_2\}}+ \ind_{\{|x_{ij_2}| \ge \tau_2\}})\bigr) \\
				& \le \sqrt{M}(\sqrt{P(|x_{ij_1}|\ge \tau_2)}+ \sqrt{P(|x_{ij_2}|\ge \tau_2})) \\
				& \le \sqrt{M}\Bigl(\sqrt{\frac{\E x_{ij_1}^4}{\tau_2^4}}+ \sqrt{\frac{\E x_{ij_2}^4}{\tau_2^4}} \Bigr) \le c_3\sqrt{\frac{ \log d}{N}} \\
			\end{aligned}
		\]
		for some $c_3>0$, which implies that $\supnorm{\bSigma_{\bx\bx}- \bSigma_{\widetilde \bx\widetilde \bx}} = O_P(\sqrt{\log d/N})$. Therefore we have
		\[
			P\Bigl(\forall \bv\in \RR^d, \bv^T\widehat\bSigma_{\widetilde\bx \widetilde\bx}\bv \ge \bv^T\bSigma_{\bx\bx}\bv - (c_2\eta_2+ c_3)\sqrt{\frac{\log d}{N}}\lonenorm{\bv}^2 \Bigr) \le d^{2-\eta_2}.
		\]
		
	\end{proof}

	\begin{proof}[\underline{Proof of Theorem \ref{thm:lm}}]
		 Suppose $\ltwonorm{\widehat\bDelta} \ge c_1\sqrt{\rho}\bigl(\frac{\lambda_N}{\kappa_{\cL}}\bigr)^{1-\frac{q}{2}}$ for some $c_1>0$, then by (\ref{eq:6.5}) and (\ref{eq:6.6}),
		\beq	
			\label{eq:6.33}		
			\begin{aligned}
				\lonenorm{\widehat\bDelta} & \le  4\sqrt{2r}\ltwonorm{\widehat\bDelta}+4\tau^{1-q}\rho \le 4\sqrt{2}\kappa_{\cL}^{\frac{q}{2}}\sqrt{\rho}\lambda_N^{-\frac{q}{2}}\ltwonorm{\widehat\bDelta}+ 4\kappa_{\cL}^{q-1}\rho\lambda_N^{1-q} \\
				& \le (4\sqrt{2}+4c_1^{-1})\sqrt{\rho}\kappa_{\cL}^{\frac{q}{2}}\lambda_N^{-\frac{q}{2}}\ltwonorm{\widehat\bDelta}	.		
			\end{aligned}
		\eeq
		For any $\delta>1$, since $\lambda_N=2\nu \sqrt{\delta\log d/N}$, it is easy to verify that as long as $\rho(\log d/N)^{1-\frac{q}{2}}\le C_1$ for some constant $C_1$, we will have by \eqref{eq:6.33} and \eqref{eq:3.3} in Lemma \ref{lem:2}
		\[
			P\Bigl({\widehat\bDelta}^T\widehat \bSigma_{\bx\bx}{\widehat\bDelta}\ge \frac{\kappa_{\cL}}{4}\ltwonorm{\widehat\bDelta}^2\Bigr)\ge 1-d^{1-\delta}.
		\]
		An application of Theorem \ref{thm:1} delivers that for constants $c_2, c_3>0$,
		\[
			\ltwonorm{\widehat\bDelta}\le c_2\sqrt{\rho}\Bigl(\frac{\lambda_N}{\kappa_{\cL}}\Bigr)^{1-\frac{q}{2}} \quad \text{and} \quad \lonenorm{\widehat\bDelta} \le c_3\rho\Bigl(\frac{\lambda_N}{\kappa_{\cL}}\Bigr)^{\frac{1-q}{2}}.
		\]	
		When $\ltwonorm{\widehat\bDelta} \le c_1\sqrt{\rho}\bigl(\frac{\lambda_N}{\kappa_{\cL}}\bigr)^{1-\frac{q}{2}}$, we can still obtain the $\ell_1$ norm bound that $\lonenorm{\widehat\bDelta}\le c_4\rho\Bigl(\frac{\lambda_N}{\kappa_{\cL}}\Bigr)^{\frac{1-q}{2}}$ for some constant $c_4$ through \eqref{eq:6.5} and \eqref{eq:6.6}. Overall, we can achieve the conclusion for some constants $C_2$ and $C_3$ that with probability at least $1-3d^{1-\delta}$,
		\[
			\ltwonorm{\widehat\bDelta}^2 \le C_2\rho \Bigl(\frac{\delta \log d}{N}\Bigr)^{1-\frac{q}{2}} \quad \text{and} \quad \lonenorm{\widehat\bDelta} \le C_3\rho \Bigl(\frac{\delta \log d}{N}\Bigr)^{\frac{1-q}{2}}.
		\]
		
		When the design only satisfies bounded moment conditions, again we first assume that $\ltwonorm{\widehat\bDelta}\ge c_5\sqrt{\rho}(\lambda_N/\kappa_{\cL})^{1-\frac{q}{2}}$ for some constant $c_5$. Analogous to the case of the sub-Gaussian design, it is easy to verify that for any $\delta>1$, as long as $\rho(\log d/N)^{\frac{1-q}{2}}\le C_4$ for some constant $C_4$, we will have by \eqref{eq:6.33} and \eqref{eq:3.3} in Lemma \ref{lem:2}
		\[
			P\Bigl({\widehat\bDelta}^T\widehat \bSigma_{\bx\bx}{\widehat\bDelta}\ge \frac{\kappa_{\cL}}{2}\ltwonorm{\widehat\bDelta}^2\Bigr)\ge 1-d^{1-\delta},
		\]
		An application of Lemma \ref{thm:1} delivers that for some constants $c_6, c_7>0$,
		\[
			\ltwonorm{\widehat\bDelta}\le c_7\sqrt{\rho}\Bigl(\frac{\lambda_N}{\kappa_{\cL}}\Bigr)^{1-\frac{q}{2}} \quad \text{and} \quad \lonenorm{\widehat\bDelta}\le c_8\sqrt{\rho}\Bigl(\frac{\lambda_N}{\kappa_{\cL}}\Bigr)^{\frac{1-q}{2}}.
		\]
		When $\ltwonorm{\widehat\bDelta}\le c_5\sqrt{\rho}(\lambda_N/\kappa_{\cL})^{1-\frac{q}{2}}$, we have exactly the same steps as those in the case of the sub-Gaussian design.
	\end{proof}

	\begin{proof}[\underline{Proof of Lemma \ref{lem:3}}]
	
	Recall that $\widetilde Y_i=\truncate{Y_i}$, then
	\begin{equation}
		\begin{aligned}
			\label{eq:6.10}
			& \opnorm{\widehat\bSigma_{Y\bX}(\tau) - \frac{1}{N}\sum\limits_{i=1}^N \inn{\bX_i, \bTheta^*}\bX_i} \\
			& \le \opnorm{\hat\bSigma_{Y\bX}(\tau)-\E(\widetilde Y\bX)}+\opnorm{\E((\widetilde Y-Y)\bX)}+\opnorm{\frac{1}{n}\sum\limits_{i=1}^n\inn{\bX_i, \bTheta^*}\bX_i- \E Y\bX}.
		\end{aligned}
	\end{equation}
	Next we use the covering argument to bound each term of the RHS. Let $\cS^{d-1}=\{\bu \in \Rd: \ltwonorm{u}=1\}$, $\cN^{d-1}$ be the $1/4-$net on $\cS^{d-1}$ and $\Phi(\bA)=\sup_{\bu \in \cN^{d_1-1}, \bv \in \cN^{d_2-1}} \bu^T\bA\bv$ for any matrix $\bA\in \Rdd$, then we claim
	\begin{equation}
		\label{eq:covering}
		\opnorm{\bA} \le \frac{16}{7}\Phi(\bA).
	\end{equation}
	 To establish the claim above, note that by the definition of the $1/4-$net, for any $\bu\in \cS^{d_1-1}$ and $\bv\in \cS^{d_2-1}$, there exist $\bu_1\in \cN^{d_1-1}$ and $\bv_1\in \cN^{d_2-1}$ such that $\ltwonorm{\bu-\bu_1}\le 1/4$ and $\ltwonorm{\bv-\bv_1}\le 1/4$. Then it follows that
	 \[
	 	\begin{aligned}
	 		\bu^T\bA\bv&=\bu_1^T\bA\bv_1+(\bu-\bu_1)^T\bA\bv_1+\bu_1^T\bA(\bv-\bv_1)+(\bu-\bu_1)^T\bA(\bv-\bv_1)\\
			&\le \Phi(\bA)+(\frac{1}{4}+\frac{1}{4}+\frac{1}{16})\sup_{\bu\in \cS^{d_1-1}, \bv \in \cS^{d_2-1}} \bu^T\bA\bv.
		\end{aligned}
	\]
	Taking the superlative over $\bu\in \cS^{d_1-1}$ and $\bv \in \cS^{d_2-1}$ on the LHS yields (\ref{eq:covering}).
	
	Now fix $\bu\in \cS^{d_1-1}$ and $\bv\in \cS^{d_2-1}$. Later on we always write $\bu^T\bX_i\bv$ as $Z_i$ and $\bu^T\bX\bv$ as $Z$ for convenience. Consider
	\[
		\bu^T(\hat\bSigma_{Y\bX}(\tau)-\E(\widetilde Y\bX))\bv=\frac{1}{n}\sum\limits_{i=1}^n\widetilde Y_iZ_i-\E(\widetilde YZ).
	\]
	Note that
	\beq
		\begin{aligned}
		\E (\widetilde Y_iZ_i- \E \widetilde Y_iZ_i)^2 & \le \E (\widetilde Y_iZ_i)^2 \le \E (Y_i Z_i)^2= \E (Z_i^2\E (Y_i^2|\bX_i)) = \E \Bigl (Z_i^2(\inn{\bX_i, \bTheta^*})^2 \Bigr) +\E\Bigl ( Z_i^2 \E(\epsilon_{i}^2|\bX_i) \Bigr )\\
		& \le 16R^2\kappa^4_0+ (2k\kappa_0^2/(k-1))\left(\E\bigl(\E(\epsilon_{i}^2|\bX_i)\bigr)^{k}\right)^{\frac{1}{k}} \le 16R^2\kappa^4_0+2k\kappa_0^2M^{\frac{1}{k}}/(k-1) <\infty,
		\end{aligned}
	\eeq
	which we denote by $v_1$ for convenience. Also we have
	\[
		\begin{aligned}
		 	\E |\widetilde Y_i Z_i|^p & \le \tau^{p-2}\E (Y_i^2|Z_i|^p)= \tau^{p-2}\E \left(\inn{\bX_i, \bTheta^*}^2|Z_i|^p + \epsilon_i^2 |Z_i|^p\right) \\
			& \le \tau^{p-2}\left (\sqrt{\E \inn{\bX_i, \bTheta^*}^4 \E |Z_i|^{2p}} + \Bigl (\E \bigl (\E (\epsilon_i^2 | \bX_i)^k \bigr) \Bigr)^{\frac{1}{k}}\Bigl( \E x_{ij}^{\frac{pk}{k-1}}\Bigr)^{1-\frac{1}{k}}\right) \\
			& \le \tau^{p-2} \Bigl(4R^2\kappa_0^2(\kappa_0\sqrt{2p})^p+M^{\frac{1}{k}}\bigl(\kappa_0 \sqrt{\frac{pk}{k-1}}\bigr)^p \Bigr)
		\end{aligned}
	\]
	and it holds for constants $c_1$ and $c_2$ that
	\begin{equation}
		\begin{aligned}
			\E |\widetilde Y_iZ_i- \E \widetilde Y_iZ_{i}|^p & \le 2^{p-1}(\E |\widetilde Y_i Z_{i}|^p + |\E \widetilde Y_iZ_{i}|^p) \le 2^{p-1}(\E |\widetilde Y_i Z_{i}|^p + (\E |Y_iZ_{i}|)^{p}) \\
			& \le c_1p!(c_2\tau)^{p-2}.
		\end{aligned}
	\end{equation}
	where the last inequality uses the Stirling approximation of $p!$. Define $v:= c_1 \vee v_1$. Then an application of Bernstein's Inequality (Theorem 2.10 in \cite{BLM13}) to $\Bigl\{\widetilde Y_i Z_i\Bigr\}_{i=1}^N$ delivers
	\[
		P\Big(|\frac{1}{N}\sum\limits_{i=1}^N \widetilde Y_i Z_i-\E \widetilde Y_iZ_i|>\sqrt{\frac{2vt}{N}}+\frac{c_2\tau t}{N}\Big)\le \exp(-t).
	\]
	By taking the union bound over all $(\bu, \bv)\in \cN^{d_1-1}\times \cN^{d_2-1}$ and (\ref{eq:covering}) it follows that
	\begin{equation}
		\label{eq:term1}
		P\Big(\opnorm{\hat\bSigma_{\widetilde Y\bX}(\tau)-\E(\widetilde Y\bX)}\ge \frac{16}{7}\Bigl(\sqrt{\frac{2vt}{N}}+\frac{c_2\tau t}{N}\Bigr)\Big)\le \exp\left((d_1+d_2)\log 8-t\right),
	\end{equation}
	where $c_1$ is a constant.
	
	Next we aim to bound $\opnorm{\E ((\widetilde Y-Y)\bX)}$. For any $\bu\in \cS^{d_1-1}$ and $\bv\in \cS^{d_2-1}$, by the Cauchy-Schwartz inequality and the Markov inequality,
	\[
		\begin{aligned}
			\E ((\widetilde Y-Y)Z) & \le \E (|YZ|\cdot \ind_{|Y|>\tau}) \le \sqrt{\E(Y^2 Z^2)P(|Y|>\tau)} \le \sqrt{\frac{v\E Y^2}{\tau^2}} \le \frac{\sqrt{v(2R^2\kappa_0^2+ M^{\frac{1}{k}})}}{\tau}.
		\end{aligned}
	\]
	Note that the inequality above holds for any $(\bu, \bv)$. Therefore
	\begin{equation}
		\label{eq:term2}
		\opnorm{\E ((\widetilde Y-Y)\bX)} \le  \frac{\sqrt{v(2R^2\kappa_0^2+ M^{\frac{1}{k}})}}{\tau}.
	\end{equation}
	
	Now we give an upper bound of the third term on the RHS of (\ref{eq:6.10}). For any $\bu\in \cS^{d_1-1}$ and $\bv\in \cS^{d_2-1}$, $\|\inn{\bX_i, \bTheta^*}Z_i\|_{\psi_1}\le R\kappa_0^2$, so by Proposition  5.16 (Bernstein-type inequality) in \cite{Ver10} it follows that for sufficiently small $t$,
	\[
		P\Big(\Big |\frac{1}{N}\sum\limits_{i=1}^N \inn{\bX_i, \bTheta^*}Z_i-\E YZ \Big|>t \Big)\le 2\exp\left(-\frac{c_3Nt^2}{R^2\kappa_0^4}\right)\,,
	\]
	where $c_3$ is a constant. Then an combination of the union bound over all points on $\cN^{d_1-1}\times \cN^{d_2-1}$ and (\ref{eq:covering}) delivers
	\begin{equation}
		\label{eq:term3}
		P\Big(\opnorm{\frac{1}{N}\sum\limits_{i=1}^N \inn{\bX_i, \bTheta^*}\bX_i- \E Y\bX}\ge \frac{16}{7}t\Big)\le 2\exp\Big((d_1+d_2)\log 8-\frac{c_3Nt^2}{R^2\kappa_0^4}\Big)\,,
	\end{equation}	
	
	Finally we choose $\tau\asymp \sqrt{N/(d_1+d_2)}$. Combining (\ref{eq:term1}), (\ref{eq:term2}) and (\ref{eq:term3}), we can find a constant $\gamma>0$ such that as long as $(d_1+d_2)/ N< \gamma$, it holds that
	\[
		P\Big(\opnorm{\widehat\bSigma_{Y\bX}(\tau) - \frac{1}{N}\sum\limits_{i=1}^N \inn{\bX_i, \bTheta^*}\bX_i}>\nu\sqrt{\frac{d_1+d_2}{N}}\Big)\le \eta\exp(-(d_1+d_2)),
	\]
	where $\nu$ and $\eta$ are constants.

	\end{proof}

	\begin{proof}[\underline{Proof of Theorem \ref{thm:cs}}]
		We first verify the RSC property. According to Proposition 1 in \cite{NWa11}, the following inequality holds for all $\bDelta\in \RR^{d_1\times d_2}$ with probability at least $1-2\exp(-N/32)$,
		\beq
			\label{eq:6.34}
			\sqrt{\vec(\bDelta)^T\overline\bSigma_{\bX\bX}\vec(\bDelta)} \ge \frac{1}{4}\ltwonorm{\sqrt{\bSigma_{\bX\bX}} \vec(\bDelta)}-c\Bigl(\sqrt{\frac{d_1}{N}}+ \sqrt{\frac{d_2}{N}} \Bigr)\nnorm{\bDelta}.
		\eeq
		Let $\kappa_{\cL}=(1/32)\lambda_{\min}(\bSigma_{\bX\bX})>0$. For ease of notations, write $\widehat\bTheta-\bTheta^*$ as $\widehat\bDelta$. Suppose $\fnorm{\widehat\bDelta} \ge c_1\sqrt{\rho}\Bigl(\frac{\lambda_N}{\kappa_{\cL}}\Bigr)^{1-\frac{q}{2}}$ for some $c_1>0$, then by (\ref{eq:6.5}) and (\ref{eq:6.6}),
		\beq			
			\label{eq:6.35}
			\begin{aligned}
				\nnorm{\widehat\bDelta} & \le  4\sqrt{2r}\fnorm{\widehat\bDelta}+4\tau^{1-q}\rho \le 4\sqrt{2\rho}\kappa_{\cL}^{\frac{q}{2}}\lambda_N^{-\frac{q}{2}}\fnorm{\widehat\bDelta}+ 4\kappa_{\cL}^{q-1}\rho\lambda_N^{1-q} \\
				& \le (4\sqrt{2}+4c_1^{-1})\sqrt{\rho}\kappa_{\cL}^{\frac{q}{2}}\lambda_N^{-\frac{q}{2}}\fnorm{\widehat\bDelta}	.		
			\end{aligned}
		\eeq
		Since we choose $\lambda_N= 2\nu \sqrt{(d_1+d_2)/N}$, there exists a constant $C_1$ such that as long as $\rho \bigl((d_1+d_2)/N\bigr)^{1-\frac{q}{2}} \le C_1$, we have by \eqref{eq:6.34} and \eqref{eq:6.35}
		\[
			\vec(\widehat\bDelta)^T\widehat\bSigma_{\bX\bX}\vec(\widehat\bDelta)\ge \kappa_{\cL}\fnorm{\widehat\bDelta}^2
		\]
		with probability at least $1-2\exp(-N/32)$. An application of Theorem \ref{thm:1} delivers that for some constants $c_2, c_3>0$, it holds with high probability that
		\[
			\fnorm{\widehat\bDelta}\le c_2\sqrt{\rho}\Bigl(\frac{\lambda_N}{\kappa_{\cL}}\Bigr)^{1-\frac{q}{2}} \quad \text{and} \quad  \nnorm{\widehat\bDelta} \le c_3\rho\Bigl(\frac{\lambda_N}{\kappa_{\cL}}\Bigr)^{\frac{1-q}{2}}.
		\]	
		When $\fnorm{\widehat\bDelta} \le c_1\sqrt{\rho}\bigl(\frac{\lambda_N}{\kappa_{\cL}}\bigr)^{1-\frac{q}{2}}$, we can still obtain the $\ell_1$ norm bound that $\lonenorm{\widehat\bDelta}\le c_4\rho\Bigl(\frac{\lambda_N}{\kappa_{\cL}}\Bigr)^{\frac{1-q}{2}}$ through \eqref{eq:6.5} and \eqref{eq:6.6}, where $c_4$ is some constant. Overall, we can achieve the conclusion that,
		\[
			\fnorm{\widehat\bDelta}^2 \le C_2\rho\Bigl(\frac{d_1+d_2}{N}\Bigr)^{1-\frac{q}{2}} \quad \text{and} \quad \nnorm{\widehat\bDelta} \le C_3\rho\Bigl(\frac{d_1+d_2}{N}\Bigr)^{\frac{1-q}{2}}
		\]
		 with probability at least $1-C_4\exp(-(d_1+d_2))$ for constants $C_2, C_3$ and $C_4$.
		
	\end{proof}

	\begin{proof}[\underline{Proof of Lemma \ref{lem:4}}]
	
		We follow essentially the same strategies as in proof of Lemma \ref{lem:3}. The only difference is that we do not use the covering argument to bound the first term in (\ref{eq:6.10}). Instead we apply the Matrix Bernstein inequality (Theorem 6.1.1 in \cite{Tro15}) to take advantage of the singleton design under the matrix completion setting.
		
		 For any fixed $\bu\in \cS^{d_1-1}$ and $\bv\in \cS^{d_2-1}$, write $\bu^T\bX_i\bv$ as $Z_i$. Then we have
		\[
			\begin{aligned}
			\E (\widetilde Y_iZ_i) & = \sqrt{d_1d_2}\E ( \widetilde Y_i u_{j(i)}v_{k(i)})=\frac{1}{\sqrt{d_1d_2}}\sum\limits_{j_0=1}^{d_1}\sum\limits_{k_0=1}^{d_2}\E \bigl (|\widetilde Y_i| \big | j(i)=j_0, k(i)=k_0 \bigr )\cdot |u_{j_0}v_{k_0}| \\
			& \le \frac{1}{\sqrt{d_1d_2}}\sum\limits_{j_0=1}^{d_1}\sum\limits_{k_0=1}^{d_2}\E \bigl (|Y_i|  \big | j(i)=j_0, k(i)=k_0 \bigr )\cdot |u_{j_0}v_{k_0}| \\
			& \le \sqrt{R^2+M^{\frac{1}{k}}}\cdot \frac{1}{\sqrt{d_1d_2}}\sum\limits_{j_0=1}^{d_1}\sum\limits_{k_0=1}^{d_2}|u_{j_0}v_{k_0}|\le \sqrt{R^2+M^{\frac{1}{k}}}.
			\end{aligned}
		\]
		Since the above argument holds for all $\bu\in \cS^{d_1-1}$ and $\bv \in \cS^{d_2-1}$, we have $\opnorm{\E (\widetilde Y_i \bX_i)} \le \sqrt{R^2+M^{\frac{1}{k}}}$. In addition,
		\[
			\begin{aligned}
				\opnorm{\E \widetilde Y_i^2 \bX_i^T\bX_i}&= d_1d_2\opnorm{\E \widetilde Y_i^2 \be_{k(i)}\be_{j(i)}^T\be_{j(i)}\be_{k(i)}^T}=d_1d_2\opnorm{\E \widetilde Y_i^2\be_{k(i)} \be_{k(i)}^T}\\
				&=d_1d_2\Bigl \|\E\Bigl( \E (\widetilde Y_i^2 | \bX_i) \be_{k(i)}\be_{k(i)}^T\Bigr)\Bigr \|_{op}=\Bigl \|\sum\limits_{k_0=1}^{d_2}\sum\limits_{j_0=1}^{d_1}\E\Bigl (\widetilde Y_i^2| k(i)=k_0, j(i)=j_0 \Bigr )\be_{k_0}\be_{k_0}^T \Bigr\|_{op} \\
				&=\max_{k_0=1,...,d_2} \sum\limits_{j_0=1}^{d_1}\E\Bigl (\widetilde Y_i^2| k(i)=k_0, j(i)=j_0 \Bigr ) \le d_1\E Y_i^2 \le d_1(R^2+M^{\frac{1}{k}})
			\end{aligned}
		\]
		Similarly we can get $\opnorm{\E \widetilde Y_i^2 \bX_i\bX_i^T} \le d_2(R^2+M^{\frac{1}{k}})$. Write $\widetilde Y_i \bX_i- \E \widetilde Y_i \bX_i$ as $\bA_i$. Therefore by the triangle inequality, $\max(\opnorm{\E \bA_i^T\bA_i}, \opnorm{\E \bA_i\bA_i^T}) \le  (d_1\vee d_2)(R^2+M^{\frac{1}{k}})$, which we denote by $v$ for convenience. Since $\opnorm{\bA_i} \le 2\sqrt{d_1d_2}\tau$, an application of the Matrix Bernstein inequality delivers,
		\beq
			\label{eq:mc-term1}
			P\left(\opnorm{\hat\bSigma_{\widetilde Y\bX}(\tau)-\E(\widetilde Y\bX)}\ge t \right)\le (d_1+d_2)\exp\Bigl(\frac{-Nt^2/2}{v+\sqrt{d_1d_2}\tau t/3}\Bigr),
		\eeq

		Next we aim to bound $\opnorm{\E ((\widetilde Y_i-Y_i)\bX_i)}$. Fix $\bu\in \cS^{d_1-1}$ and $\bv\in \cS^{d_2-1}$. By the Cauchy-Schwartz inequality and the Markov inequality,
	\[
		\begin{aligned}
			\E ((\widetilde Y_i-Y_i)Z_i) & = \sqrt{d_1d_2}\E ((\widetilde Y_i-Y_i)u_{j(i)}v_{k(i)}) \\
			& =\frac{1}{\sqrt{d_1d_2}}\sum\limits_{j_0=1}^{d_1}\sum\limits_{k_0=1}^{d_2}\E \bigl (|\widetilde Y_i-Y_i| \big | j(i)=j_0, k(i)=k_0 \bigr )\cdot |u_{j_0}v_{k_0}| \\
			& \le \frac{1}{\sqrt{d_1d_2}}\sum\limits_{j_0=1}^{d_1}\sum\limits_{k_0=1}^{d_2}\E \bigl (|Y_i|\cdot \ind_{|Y_i|>\tau} \big | j(i)=j_0, k(i)=k_0 \bigr )\cdot |u_{j_0}v_{k_0}| \\
			& \le \frac{R^2+M^{\frac{1}{k}}}{\tau}\cdot \frac{1}{\sqrt{d_1d_2}}\sum\limits_{j_0=1}^{d_1}\sum\limits_{k_0=1}^{d_2}|u_{j_0}v_{k_0}|\le \frac{R^2+M^{\frac{1}{k}}}{\tau}.
		\end{aligned}
	\]
	Note that the inequality above holds for any $(\bu, \bv)$. Therefore
	\begin{equation}
		\label{eq:mc-term2}
		\opnorm{\E ((\widetilde Y_i-Y_i)\bX_i)} \le \frac{R^2+M^{\frac{1}{k}}}{\tau} .
	\end{equation}
	
	Now we give an upper bound of the third term on the RHS of (\ref{eq:6.10}). Denote $\inn{\bX_i, \bTheta^*}\bX_i - \E\inn{\bX_i, \bTheta^*}\bX_i$ by $\bB_i$. It is not hard to verify that $\|\bB_i\|_{op} \le 2R\sqrt{d_1d_2}$ and $\max(\opnorm{\E \bB_i^T\bB_i}, \allowbreak \opnorm{\E \bB_i\bB_i^T}) \le  (d_1\vee d_2)R^2$, it follows that for any $t \in \RR$,
	\beq
		\label{eq:mc-term3}
		P\Big(\Bigl \|\frac{1}{N}\sum\limits_{i=1}^N \inn{\bX_i, \bTheta^*}\bX_i-\E Y_i\bX_i \Bigr \|_{op} \ge t\Big)\le  (d_1+d_2)\exp\Bigl(\frac{-Nt^2/2}{(d_1\vee d_2)R^2+2R\sqrt{d_1d_2}t/3}\Bigr).
	\eeq
	Finally, choose $\tau=\sqrt{N/((d_1 \vee d_2)\log (d_1+d_2))}$. Combining (\ref{eq:mc-term1}), (\ref{eq:mc-term2}) and (\ref{eq:mc-term3}), it is easy to verify that for any $\delta>0$, there exist constants $\nu$ and $\gamma$ such that the conclusion holds as long as $(d_1 \vee d_2)\log (d_1+d_2)/N< \gamma$.
	
	\end{proof}

	\begin{proof}[\underline{Proof of Theorem \ref{thm:mc}}]
		This proof essentially follows the proof of Lemma 1 in \cite{NWa12}. Write $(d_1+d_2)/2$ as $d$. Define a constraint set
		\[
		 	\cC(N;c_0)=\Bigl\{ \bDelta\in \RR^{d_1\times d_2}, \bDelta\neq \bzero | \sqrt{d_1d_2}\frac{\|\bDelta\|_{\max}\cdot \nnorm{\bDelta}}{\fnorm{\bDelta}^2} \le \frac{1}{c_0L}\sqrt{\frac{N}{(d_1+d_2)\log (d_1+d_2)}}\Bigr\}.
		\]
		According to Case 1 in the proof of Lemma 1 in \cite{NWa12}, if $\widehat\bDelta\notin\cC(N;c_0)$, we have
		\[
			\fnorm{\widehat\bDelta}^2 \le 2c_0R\sqrt{\frac{d\log d}{N}} \{ 8\sqrt{r}\fnorm{\widehat\bDelta}+4\sum\limits_{j=r+1}^{d_1\wedge d_2} \sigma_j(\bTheta^*) \},
		\]
		Following the same strategies in the proof of Theorem \ref{thm:1} in our work, we have
		\[
			\fnorm{\widehat\bDelta} \le c_1\sqrt{\rho}\Bigl(2c_0R\sqrt{\frac{d\log d}{N}}\Bigr)^{1-\frac{q}{2}}
		\]
		for some constant $c_1$. If $\widehat\bDelta\in \cC(N;c_0)$, according to Case 2 in proof of Lemma 1 in \cite{NWa12}, with probability at least $1-C_1\exp(-C_2d\log d)$, either $\fnorm{\widehat\bDelta} \le 512R/\sqrt{N}$, or $\vec(\widehat\bDelta)^T\widehat\bSigma_{\bX\bX}\vec(\widehat\bDelta) \ge (1/256)\fnorm{\widehat\bDelta}^2$, where $C_1$ and $C_2$ are certain constants. For the case where $\fnorm{\widehat\bDelta} \le 512R/\sqrt{N}$, combining this fact with \eqref{eq:6.5} and \eqref{eq:6.6} delivers that
		\[
			\nnorm{\widehat\bDelta}\le 4\sqrt{2\rho}\tau^{-\frac{q}{2}}\frac{512R}{\sqrt{N}}+4\tau^{1-q}\rho.
		\]
		We minimize the RHS of the inequality above by plugging in $\tau=\bigl(\frac{R^2}{\rho N}\bigr)^{\frac{1}{2-q}}$. Then we have for some constant $c_3>0$,
		\[
			\nnorm{\widehat\bDelta} \le c_3\Bigl(\rho\bigl(\frac{R^2}{N}\bigr)^{1-q}\Bigr)^{\frac{1}{2-q}}.
		\]
		For the case where $\vec(\widehat\bDelta)^T\widehat\bSigma_{\bX\bX}\vec(\widehat\bDelta) \ge (1/256)\fnorm{\widehat\bDelta}^2$, it is implied by Theorem \ref{thm:1} in our work that
		\[
			\fnorm{\widehat\bDelta} \le c_4\sqrt{\rho}\lambda_N^{1-\frac{q}{2}}\quad \text{and} \quad \nnorm{\widehat\bDelta} \le c_5\rho\lambda_N^{\frac{1-q}{2}}
		\]
		for some constants $c_4$ and $c_5$. Since $\lambda_N = \nu\sqrt{\delta(d_1 \vee d_2)\log(d_1+d_2)/N}$ and $R\ge 1$, by Lemma \ref{lem:4}, it holds with probability at least $1-C_1\exp(-C_2(d_1+d_2))-2(d_1+d_2)^{1-\delta}$ that
		\[
			\fnorm{\widehat\bDelta}^2 \le C_3\max\Bigl\{\rho\Bigl( \frac{\delta R^2(d_1+d_2)\log(d_1+d_2)}{N}\Bigr)^{1-\frac{q}{2}}, \frac{R^2}{N}\Bigr\}
		\]
		and
		\[
			\nnorm{\widehat\bDelta} \le C_4\max\Bigl\{\rho\Bigl(\frac{\delta R^2(d_1+d_2)\log(d_1+d_2)}{N}\Bigr)^{\frac{1-q}{2}}, \bigl(\frac{\rho R^{2-2q}}{N^{1-q}}\bigr)^{\frac{1}{2-q}}\Bigr\}
		\]
		for constants $C_3$ and $C_4$.
		
	\end{proof}

	\begin{proof}[\underline{Proof of Lemma \ref{lem:5}}]
	
		(a) First we study the sub-Gaussian design. Since $\widetilde\bx_i=\bx_i$, we do not add tildes above $\bx, \bx_i$ or $x_{ij}$ in this proof. Denote $\shrunk{\by_i}$ by $\widetilde\by_i$, then we have
		\beq
			\label{eq:2.2-3term}
			\begin{aligned}
			\opnorm{\widehat\bSigma_{\bx\widetilde \by}(\tau)&  - \frac{1}{n}\sum\limits_{j=1}^n \bx_j\bx_j^T\bTheta^*}= \opnorm{\widehat\bSigma_{\widetilde\by \bx}(\tau)  - \frac{1}{n}\sum\limits_{j=1}^n \bTheta^{*T}\bx_j\bx_j^T} \\
			& \le \opnorm{\hat\bSigma_{\widetilde\by \bx}(\tau)-\E(\widetilde\by \bx^T)}+\opnorm{\E(\widetilde\by \bx^T-\by\bx^T)}+\opnorm{\widehat{\bSigma}_{\bx\bx}\bTheta^* - \E (\by\bx^T)}
			\end{aligned}
		\eeq
		
		Let
		\[
			\bS_i=\left[
			\begin{array}{cc}
				\bzero & \widetilde\by_i\bx_i^T-\E \widetilde\by_i\bx_i^T  \\
				\bx_i\widetilde\by_i^T- \E \bx_i\widetilde\by_i^T & \bzero
			\end{array}
			\right].
		\]
		Now we bound $\opnorm{\E \bS_i^p}$ for $p>2$. When $p$ is even, i.e., $p=2l$, we have
		\[
			\bS_i^{2l}=\left[
				\begin{array}{cc}
					((\widetilde\by_i\bx_i^T- \E \widetilde\by_i\bx_i^T)(\bx_i\widetilde\by_i-\E \bx_i\widetilde\by_i))^l & \bzero \\
					\bzero & ((\bx_i\widetilde\by_i-\E \bx_i\widetilde\by_i)(\widetilde\by_i\bx_i^T- \E \widetilde\by_i\bx_i^T))^l
				\end{array}
			\right].
		\]
		For any $\bv\in \cS^{d_2-1}$,
		\beq
			\label{eq:evenmoment}
			\begin{aligned}
				\bv^T(\E (\widetilde\by_i\bx_i^T\bx_i \widetilde\by_i^T)^l )\bv & = \E ((\bx_i^T\bx_i)^l(\widetilde\by_i^T\widetilde\by_i)^{l-1}(\bv^T\widetilde\by_i)^2) \le \tau^{2l-2}\E ((\bx_i^T\bx_i)^l(\bv^T\widetilde\by_i)^2) \\
				& \le \tau^{2l-2}\E ((\bx_i^T\bx_i)^l(\bv^T\by_i)^2)= \tau^{2l-2} \E\left(\E [(\bv^T\by_i)^2|\bx_i](\bx_i^T\bx_i)^l\right) \\
				& = \tau^{2l-2}\E \left((\bv^T\bTheta^{*^T}\bx_i)^2(\bx_i^T\bx_i)^l+\E[(\bv^T\bepsilon_i)^2|\bx_i](\bx_i^T\bx_i)^l\right).
			\end{aligned}
		\eeq
		Also note that for any $\bv\in \cS^{d_2-1}$, we have
		\beq
			R\ltwonorm{\bv}^2 \ge \bv^T\E \by_i\by_i^T \bv \ge  (\bTheta^{*}\bv)^T\bSigma_{\bx\bx}(\bTheta^*\bv)\ge \kappa_{\cL}\ltwonorm{\bTheta^{*T}\bv}^2.
		\eeq
		Therefore it follows that $\opnorm{\bTheta^*} \le \sqrt{R/\kappa_{\cL}}$. Combining this with the fact that $\|\bx_i^T\bx_i\|_{\psi_1} \le 2d_1\kappa_0^2$, it follows that
		\[
			\E \left ((\bv^T\bTheta^{*^T}\bx_i)^2(\bx_i^T\bx_i)^l \right) \le \sqrt{\E (\bv^T\bTheta^{*^T}\bx_i)^4} \sqrt{\E (\bx_i^T\bx_i)^{2l}} \le (2\sqrt{R/\kappa_{\cL}}\kappa_0)^2(4d_1\kappa_0^2l)^l.
		\]
		Also by H\"{o}lder Inequality, we have
		\[
			\begin{aligned}
			 \E \left( \E[(\bv^T\bepsilon_i)^2|\bx_i](\bx_i^T\bx_i)^l\right) & \le \E \left(\E [(\bv^T\bx_i)^2 | \bx_i]^k\right)^{\frac{1}{k}}\left(\E [(\bx_i^T\bx_i)^{\frac{lk}{k-1}}]\right)^{1-\frac{1}{k}} \\
			 & \le  M^{\frac{1}{k}}(2d_1\kappa_0^2\frac{lk}{k-1})^l.
			 \end{aligned}
		\]
		Therefore we have
		\beq
			\label{eq:5.19}
			\opnorm{\E (\widetilde\by_i\bx_i^T\bx_i \widetilde\by_i^T)^l )} \le c'_1(c_1\tau)^{2l-2}(d_1l)^l
		\eeq
		for $l\ge 1$, where $c_1$ and $c'_1$ are some constants. Letting $l=1$ in the equation above implies that $\opnorm{\E \widetilde\by_i\bx_i^T \E \bx_i\widetilde\by_i^T} \le \opnorm{\E \widetilde\by_i\bx_i^T\bx_i\widetilde\by_i^T} \le c'_1d_1$. Therefore it holds that $(\widetilde\by_i\bx_i^T- \E \widetilde\by_i\bx_i^T)(\bx_i\widetilde\by_i-\E \bx_i\widetilde\by_i^T) \preceq 2 \widetilde\by_i\bx_i^T\bx_i\widetilde\by_i^T + 2\E \widetilde\by_i\bx_i^T \E \bx_i\widetilde\by_i^T \preceq 2\tilde \by_i\bx_i^T\bx_i\tilde \by_i^T +2c_1'd_1\bI$. In addition, for any commutable positive semi-definite (PSD) matrices $\bA$ and $\bB$, it is true that $(\bA+\bB)^l \preceq 2^{l-1} (\bA^l +\bB^l)$ for $l>1$. Then it follows that
		\beq
			\label{eq:matrixmoment}
			\begin{aligned}
				 ((\widetilde\by_i\bx_i^T- \E \widetilde\by_i\bx_i^T)(\bx_i\widetilde\by_i^T-\E \bx_i\widetilde\by_i^T))^l & \preceq (2 \widetilde\by_i\bx_i^T\bx_i\widetilde\by_i^T + 2c_1'd_1\bI)^l \\
				& \preceq 2^{2l-1}((\widetilde\by_i\bx_i^T\bx_i\widetilde\by_i^T)^l+(c_1'd_1)^l\bI).
			\end{aligned}
		\eeq
		Therefore we have $\opnorm{\E ((\widetilde\by_i\bx_i^T- \E \widetilde\by_i\bx_i^T)(\bx_i\widetilde\by_i^T-\E \bx_i\widetilde\by_i^T))^l} \le c'_2(c_2\tau)^{2l-2}(d_1l)^l$ for some constants $c_2$ and $c'_2$. Using similar methods, we can derive $\opnorm{\E ((\bx_i\widetilde\by_i^T-\E \bx_i\widetilde\by_i^T)(\widetilde\by_i\bx_i^T- \E \widetilde\by_i\bx_i^T))^l} \le c'_3(c_3\tau)^{2l-2}(d_2l)^l$ for some constants $c_3$ and $c'_3$. So we can achieve for some constants $c_4$ and $c'_4$,
		\beq
			\label{eq:moment1}	
			\opnorm{\E \bS_i^{2l}}\le c'_4(c_4\tau\sqrt{d_1 \vee d_2})^{2l-2}l^l(d_1 \vee d_2)
		\eeq
		
		When $p$ is an odd number, i.e., $p=2l+1$, by (\ref{eq:matrixmoment}) we have
		\[
			\begin{aligned}
				&\bS_i^{2l+1} =\sqrt{\bS_i^{4l+2}} \\
				& \preceq 2^{2l+1/2} \diag(\sqrt{(\widetilde\by_i\bx_i^T\bx_i\widetilde\by_i^T)^{2l+1}+(\E \widetilde\by_i\bx_i^T\E \bx_i\widetilde\by_i^T)^{2l+1}}, \sqrt{(\bx_i\widetilde\by_i^T\widetilde\by_i\bx_i^T)^{2l+1}+(\E \bx_i\widetilde\by_i^T\E \widetilde\by_i\bx_i\widetilde\by_i^T)^{2l+1}}) \\
				& \preceq 2^{2l+1/2} \diag \left((\widetilde\by_i\bx_i^T\bx_i\widetilde\by_i^T)^{l+1/2}+(\E \widetilde\by_i\bx_i^T\E \bx_i\widetilde\by_i^T)^{l+1/2}, (\bx_i\widetilde\by_i^T\widetilde\by_i\bx_i^T)^{l+1/2}+(\E \widetilde\by_i\bx_i^T\E \bx_i\widetilde\by_i^T)^{l+1/2} \right).
			\end{aligned}
		\]
		Note that $(\bx_i\widetilde\by_i^T\widetilde\by_i\bx_i^T)^{\frac{1}{2}} = \bx_i\widetilde\by_i^T\widetilde\by_i\bx_i^T/(\ltwonorm{\bx_i}\ltwonorm{\widetilde\by_i}) $ and $(\widetilde\by_i\bx_i^T\bx_i\widetilde\by_i^T)^{\frac{1}{2}} = \widetilde\by_i\bx_i^T\bx_i\widetilde\by_i^T/ (\ltwonorm{\bx_i}\ltwonorm{\widetilde\by_i})$, so for any $\bv\in \cS^{d_2-1}$, we have
		\[
			\bv^T\E (\widetilde\by_i\bx_i^T\bx_i\widetilde\by_i^T)^{l+1/2} \bv = \E ((\bx_i^T\bx_i)^{l+\frac{1}{2}}(\widetilde\by_i^T\widetilde\by_i)^{l-\frac{1}{2}}(\bv^T\widetilde\by_i)^2) \le \tau^{2l-1}\E ((\bx_i^T\bx_i)^{l+\frac{1}{2}}(\bv^T\widetilde\by_i)^2).
		\]
		Following the same steps as in the case of $l$ being even, we can derive $\opnorm{\E (\widetilde\by_i\bx_i^T\bx_i \widetilde\by_i^T)^{l+1/2} )} \le c_5'(c_5\tau)^{2l-1}(d_1(l+1/2))^{l+1/2}$, $\opnorm{\E (\bx_i\tilde \by_i^T\tilde \by_i \bx_i^T)^{l+1/2} )} \le c_6'(c_6\tau)^{2l-1}(d_2(l+1/2))^{l+1/2}$ and finally
		\beq
			\label{eq:moment2}
			\opnorm{\E \bS_i^{2l+1}} \le c'_7(c_7\tau \sqrt{d_1 \vee d_2})^{2l-1}(l+1/2)^{l+1/2}(d_1\vee d_2).
		\eeq
		Define $\sigma^2 = (c_7 \vee c_7')(d_1 \vee d_2)$. Then by (\ref{eq:moment1}), we have $\opnorm{\E \bS_i^2} \le \sigma^2$. Also by combining (\ref{eq:moment1}) and (\ref{eq:moment2}), we can verify the moment constraints of Lemma \ref{lem:bernsteinm}; for $p> 2$, we have
		\beq
			\label{eq: bmoment}
			\opnorm{\E \bS_i^p}\le p!((c_4 \vee c_7)\tau\sqrt{d_1 \vee d_2})^{p-2}\sigma^2.
		\eeq
		Therefore by 	Lemma \ref{lem:bernsteinm}, we have
		\beq
			\label{eq:mt-term1}
			\begin{aligned}
				P \Big( \opnorm{\hat\bSigma_{\widetilde\by \bx}(\tau)-\E(\widetilde\by \bx^T)} & \ge t \Big)  \le P \Big(\opnorm{\frac{1}{n}\sum\limits_{i=1}^n \bS_i} \ge t\Big) \\
				& \le (d_1+d_2)\cdot \exp\Big(-c \min\bigl(\frac{nt^2}{\sigma^2}, \frac{nt}{(c_4 \vee c_7)\tau \sqrt{d_1 \vee d_2}}\bigr) \Big).
			\end{aligned}
		\eeq
			
	Next we aim to bound $\opnorm{\E (\widetilde\by_i\bx_i^T- \by_i\bx_i^T)}$. Note that for any $\bu \in \cS^{d_1-1}$ and $\bv \in \cS^{d_2-1}$,
	\[
		\begin{aligned}
		\E ((\bv^T\by_i)^2(\bu^T \bx_i)^2) & = \E (\E((\bv^T\by_i)^2|\bx_i)(\bu^T\bx_i)^2)= \E \Bigl(\E \bigl((\bv^T\bTheta^* \bx_i)^2(\bu^T\bx_i)^2\bigr) + \E ((\bv^T\bepsilon_i)^2|\bx_i)(\bu^T\bx_i)^2 \Bigr) \\
		& := v^2< \infty.
		\end{aligned}
	\]
	For any $\bu\in \cS^{d_1-1}$ and $\bv\in \cS^{d_2-1}$, by the Cauchy-Schwartz inequality and the Markov inequality,
	\[
		\begin{aligned}
			\E ((\bv^T\widetilde\by_i)(\bu^T\bx_i)- (\bv^T\by_i)(\bu^T\bx_i)) & \le \E (|(\bv^T\by_i)(\bu^T\bx_i)|\cdot \ind_{\ltwonorm{\by}>\tau}) \le \sqrt{\E(|(\bv^T\by_i)(\bu^T\bx_i)|^2)P(\|\by\|_2>\tau)} \\
			& \le \sqrt{\frac{v^2 \E \ltwonorm{\by}^2}{\tau^2}}=\frac{v}{\tau}\sqrt{R^2\E \ltwonorm{\bx}^2+\E \ltwonorm{\bepsilon}^2}\le \frac{v \sqrt{ d_2}}{\tau}\sqrt{R^2\kappa_0^2+M^{\frac{1}{k}}}.
		\end{aligned}
	\]
	Since $\bu$ and $\bv$ are arbitrary, we have
	\begin{equation}
		\label{eq:mt-term2}
	       \opnorm{\E (\widetilde\by_i\bx_i^T- \by_i\bx_i^T)} \le \frac{v \sqrt{d_2}}{\tau}\sqrt{R^2\kappa_0^2+M^{\frac{1}{k}}}.
	\end{equation}
	
	Now we give an upper bound of the third term on the RHS of (\ref{eq:2.2-3term}). For this term, we still follow the covering argument. Let $\cN^{d-1}$ be the $1/4-$net on $\cS^{d-1}$. For any $\bu\in \cS^{d_1-1}$ and $\bv\in \cS^{d_2-1}$, $\|\bu^T\bx_i\bx_i^T\bTheta^*\bv\|_{\psi_1}\le \sqrt{R/\kappa_{\cL}}\kappa_0^2$, so by Proposition  5.16 (Bernstein-type inequality) in \cite{Ver10} it follows that for sufficiently small $t$,
	\[
		P\Big(\left |\bu^T\widehat{\bSigma}_{\bx\bx}\bTheta^*\bv-\E (\bu^T\bx_i)(\bv^T\by_i) \right|>t \Big)\le 2\exp\left(-c_8nt^2\right),
	\]
	where $c_8$ is some positive constant. Then an combination of the union bound over all points on $\cN^{d_1-1}\times \cN^{d_2-1}$ and (\ref{eq:covering}) delivers
	\begin{equation}
		\label{eq:mt-term3}
		P\Big(\ltwonorm{\widehat{\bSigma}_{\bx\bx}\bTheta^* - \E \by\bx^T}\ge \frac{16}{7}t\Big)\le 2\exp\left((d_1+d_2)\log 8-c_8nt^2\right),
	\end{equation}
	where $c_8$ is a constant.
	
	Finally we choose $\tau=O(\sqrt{n/\log (d_1+d_2)})$, and combining (\ref{eq:mt-term1}), (\ref{eq:mt-term2}) and (\ref{eq:mt-term3}) delivers that for any $\delta>0$, as long as $(d_1+d_2)\log(d_1+d_2)/n<\gamma$ for some constant $\gamma>0$,
	\[
		P\Big(\opnorm{\widehat\bSigma_{\bx\widetilde\by}(\tau_1) - \frac{1}{n}\sum\limits_{j=1}^n \bTheta^{*T}\bx_j\bx_j^T}\ge \sqrt{\frac{(\nu+\delta)(d_1+d_2)\log (d_1+d_2)}{n}}\Big)\le 2(d_1+d_2)^{1-\eta\delta},
	\]
	where $\nu$ and $\eta$ are universal constants.
	
	(b) Now we switch to the case of designs with only bounded moments. We first show that for two constants $c_3$ and $c_4$,
	\beq
		\label{eq:6.31}
		P\Bigl(\opnorm{\widehat\bSigma_{\widetilde\bx\widetilde\by} - \bSigma_{\bx\by}} \ge \sqrt{\frac{(c_3+\delta)(d_1+d_2)\log(d_1+d_2)}{n}}\Bigr)\le (d_1+d_2)^{1-c_4\delta}.
	\eeq
	Note that
	\[
		\E (\bv^T\by_i)^4 = \E (\bv^T\bTheta^{*T}\bx_i+ \bv^T\bepsilon_i)^4\le 16 \bigl(\E (\bv^T\bTheta^{*T}\bx_i)^4 + \E (\bv^T\bepsilon_i)^4 \bigr).
	\]
	By Rosenthal-type inequality in \cite{ISh98}, we have
	\[
		\E (\bv^T\by_i)^4 \le 16(16R^4\kappa_0^4 + M) =: V_1.
	\]
	In addition, $\E (\bv^T\bx_i)^4 \le 16\kappa_0^4=:V_2$. Let $V:=\max(V_1, V_2)$, $\bZ_i := \bx_i\by_i^T$and $\widetilde \bX_i:= \widetilde \bx_i\widetilde \by_i^T$. Following the proof strategy for Theorem \ref{thm:6}, we have
		\[
			\opnorm{\widetilde\bZ_i - \E \widetilde \bZ_i}\le \opnorm{\widetilde \bZ_i}+\opnorm{\E \bZ_i}=\ltwonorm{\widetilde\bx_i}\ltwonorm{\widetilde \by_i}+ \sqrt{R} \le (d_1d_2)^{\frac{1}{4}}\tau^2+\sqrt{V}.
		\]
		Also for any $\bv\in \cS^{d-1}$, we have
		\[	
			\begin{aligned}
				\E(\bv^T \widetilde \bZ_i^T\widetilde\bZ_i \bv) & =\E (\ltwonorm{\widetilde\bx_i}^2 (\bv^T\widetilde\by_i)^2)\le \E  (\ltwonorm{\bx_i}^2 (\bv^T\by_i)^2) \\
			 	& = \sum\limits_{j=1}^d \E (x_{ij}^2(\bv^T \by_i)^2) \le \sum\limits_{j=1}^d \sqrt{\E (x_{ij}^4) \E(\bv^T \by_i)^4 } \le Vd_1.
			\end{aligned}
		\]
		Then it follows that $\opnorm{\E \widetilde\bZ_i^T\widetilde\bZ_i}\le Vd_1$. Similarly we can obtain $\opnorm{\E \widetilde \bZ_i\widetilde \bZ_i^T} \le Vd_2$. Denote $d_1+d_2$ by $d$. Since $\max(\opnorm{(\E \widetilde \bX_i)^T\E \widetilde\bX_i}, \opnorm{(\E \widetilde \bZ_i)\E \widetilde\bZ_i^T}) \le \opnorm{\E \bZ_i}^2 \le V$, $$\max(\opnorm{\E((\widetilde\bZ_i- \E\widetilde\bZ_i)^T(\widetilde\bZ_i-\E \widetilde\bZ_i))}, \opnorm{\E((\widetilde\bZ_i- \E\widetilde\bZ_i)(\widetilde\bZ_i-\E \widetilde\bZ_i)^T)}) \le V(d+1).$$ By Corollary 6.2.1 in \cite{Tro12}, we have for some constant $c_1$,
		\beq
			\label{eq:6.32}
			P \left( \opnorm{\frac{1}{n}\sum\limits_{i=1}^n \widetilde\bZ_i-\E \widetilde\bZ_i}>t \right)\le d\exp\Bigl(-c_1\bigl(\frac{nt^2}{V(d+1)} \wedge \frac{nt}{\sqrt{d}\tau_1\tau_2+\sqrt{V}}\bigr)\Bigr).	
		\eeq
		Now we bound the bias $\E \opnorm{\bZ_i-\widetilde\bZ_i}$. For any $\bv\in \cS^{d-1}$, it holds that
		\[
			\begin{aligned}
				\E & (\bv^T(\bZ_i-\widetilde\bZ_i)\bv) =\E(|(\bv^T\bx_i)(\bv^T\by_i)-(\bv^T\widetilde\bx_i)(\bv^T\widetilde \by_i)|\ind_{\{\|\bx_i\|_4\ge \tau_1\ \text{or}\ \|\by_i\|_4 \ge \tau_2\}}) \\
				& \le \E (|(\bv^T\bx_i)(\bv^T\by_i)| (\ind_{\{\|\bx_i\|_4>\tau_1\}}+ \ind_{\|\by_i\|_4>\tau_2})) \\
				& \le \sqrt{\E ((\bv^T\bx_i)^2(\bv^T\by_i)^2) P(\|\bx_i\|_4>\tau_1)} + \sqrt{\E ((\bv^T\bx_i)^2(\bv^T\by_i)^2) P(\|\by_i\|_4>\tau_2)} \\
				& \le \frac{V\sqrt{d_1}}{\tau_1^2}+ \frac{V\sqrt{d_2}}{\tau_2^2}.
			\end{aligned}
		\]
		Therefore we have $\opnorm{\E(\bZ_i-\widetilde\bZ_i)}\le V\sqrt{d}(1/\tau_1^2+1/\tau_2^2)$. Choose $\tau_1\asymp (n/\log d_1)^{\frac{1}{4}}$, $\tau_2 \asymp (n/\log d_2)^{\frac{1}{4}}$ and substitute $t$ with $\sqrt{\delta d\log d/n}$. Then we obtain \eqref{eq:6.31} by combining the bound of bias and (\ref{eq:6.32}).
		
		Finally, note that
		\[
			\opnorm{\widehat\bSigma_{\widetilde\bx\widetilde\by} - \widehat\bSigma_{\widetilde\bx\widetilde\bx}\bTheta^*} \le \opnorm{\widehat\bSigma_{\widetilde\bx\widetilde \by} - \bSigma_{\bx\by}}+ \opnorm{(\widehat\bSigma_{\widetilde\bx\widetilde\bx}- \bSigma_{\bx\bx})\bTheta^*}+ \opnorm{\bSigma_{\bx\by}- \bSigma_{\bx\bx}\bTheta^*}.
		\]
		Combining consistency of $\widehat\bSigma_{\widetilde \bx\widetilde\bx}$ established in Theorem \ref{thm:6}, Condition (C1) and \eqref{eq:6.31}, we can reach the conclusion of the lemma.
	\end{proof}
		
	\begin{lem}[Matrix Bernstein Inequality with Moment Constraint]
		\label{lem:bernsteinm}
		Consider a finite sequence $\{\bS_i\}_{i=1}^n$ of independent, random, Hermitian matrices with dimensions $d\times d$. Assume that $\E \bS_i=0$ and $\opnorm{\E \bS_i^2} \le \sigma^2$. Also the following moment conditions hold for all $1\le i \le n$ and $p\ge 2$:
		\[
			\opnorm{ \E \bS_i^p} \le p!L^{p-2}\sigma^2 ,
		\]
		where $L$ is a constant. Then for every $t\ge 0$ we have
		\[
			P \Big(\opnorm{\frac{1}{n}\sum\limits_{i=1}^n \bS_i} \ge t\Big)\le d \exp\Big(\frac{-nt^2}{4\sigma^2+2Lt}\Big) \le d\cdot \exp\Big(-c \min\bigl(\frac{nt^2}{\sigma^2}, \frac{nt}{L}\bigr) \Big).
		\]
	\end{lem}
	
	\begin{proof}[\underline{Proof of Lemma \ref{lem:bernsteinm}}]
		Given the moment constraints, we have for $0<\theta<1/L$,
		\[
			\E e^{\theta\bS_i} = \bI+\sum\limits_{p=2}^{\infty} \frac{\theta^p \E \bS_i^p}{p!} \preceq \bI+ \sum\limits_{p=2}^{\infty} \sigma^2 L^{p-2}\theta^{p} \bI = \bI+ \frac{\theta^2\sigma^2}{1-\theta L} \bI \preceq \exp\left(\frac{\theta^2\sigma^2}{1-\theta L}\right) \bI.
		\]
		Let $g(\theta)= \theta^2/(1-\theta L)$. Owing to the master tail inequality (Theorem 3.6.1 in \cite{Tro15}), we have
		\[
			\begin{aligned}
				P \Big( \opnorm{\sum\limits_{i=1}^n \bS_i}\ge t \Big) & \le \inf\limits_{\theta>0} e^{-\theta t}\tr\exp\Big( \sum_{i} \log \E e^{\theta\bS_i}\Big) \\
				& \le \inf\limits_{0<\theta<1/L} e^{-\theta t}\tr \exp\Big( n\sigma^2 g(\theta) \bI \Big) \\
				& \le \inf\limits_{0<\theta<1/L} de^{-\theta t} \exp\left( n\sigma^2 g(\theta) \right).	
			\end{aligned}
		\]
		Choosing $\theta=t/(2n\sigma^2+Lt)$, we can reach the conclusion.
	\end{proof}

	\begin{proof}[\underline{Proof of Theorem \ref{thm:mt}}]
	
		(a) First consider the case of the sub-Gaussian design. Denote $\frac{1}{n}\sum\nolimits_{j=1}^n \bx_j\bx_j^T$ by $\overline\bSigma_{\bx\bx}$. Since $\bx_j$ is a sub-Gaussian vector, an application of Theorem 5.39 in \cite{Ver10} implies that with probability at least $1-\exp(-c_1d_1)$, $\opnorm{\overline\bSigma_{\bx\bx}-\bSigma_{\bx\bx}}\le c_2\sqrt{d_1/n}$ and furthermore $\lambda_{\min}({\overline\bSigma_{\bx\bx}})\ge \frac{1}{2}\lambda_{\min}(\bSigma_{\bx\bx})>0$ as long as $d_1/n$ is sufficiently small. Therefore
		\beq	
			\label{eq:6.38}
			\begin{aligned}
				\vec(\widehat\bDelta)^T\widehat\bSigma_{\bX\bX}\vec(\widehat\bDelta) & =\frac{1}{N}\sum\limits_{i=1}^N \inn{\widehat\bDelta, \bX_i}^2= \frac{1}{N}\sum\limits_{j=1}^n\sum\limits_{k=1}^{d_2}\inn{\widehat\bDelta, \bx_j\be_k^T}^2 \\
				& = \frac{1}{N} \sum\limits_{j=1}^n \sum\limits_{k=1}^{d_2} \Tr (\bx_j^T\widehat\bDelta \be_k)^2= \frac{1}{N} \sum\limits_{k=1}^{d_2}\sum\limits_{j=1}^n  (\bx_j^T\widehat\bDelta \be_k)^2 \\
				& = \frac{1}{d_2}\sum\limits_{k=1}^{d_2} (\widehat\bDelta \be_k)^T\overline\bSigma_{\bx\bx} (\widehat\bDelta \be_k) \ge \frac{1}{2d_2}\lambda_{\min}(\bSigma_{\bx\bx})\fnorm{\widehat\bDelta}^2\,,
			\end{aligned}
		\eeq
		with probability at least $1-\exp(-c_1d_1)$. In addition, by Lemma \ref{lem:5}, as long as we choose $\lambda_N= \frac{2}{d_2}\sqrt{(\nu_1+\delta)(d_1+d_2)\log (d_1+d_2)/n}$, $\lambda_N \ge 2\opnorm{\widehat\bSigma_{Y\bX} - \frac{1}{N}\sum\nolimits_{i=1}^N Y_i\bX_i}$ with probability at least $1-2(d_1+d_2)^{1-\eta_1\delta}$. Finally, by Theorem \ref{thm:1} we establish the statistical error rate: there exist $\gamma_1, \gamma_2>0$ that as long as $(d_1+d_2)\log(d_1+d_2)/n<\gamma_1$ and $d_1+d_2>\gamma_2$, we have
		\[
			\fnorm{\widehat\bTheta(\lambda_N, \tau)- \bTheta^*}^2\le C_1\rho \Bigl(\frac{(\nu_1+\delta)(d_1+d_2)\log{(d_1+d_2)}}{n}\Bigr)^{1-q/2}
		\]
		and
		\[
			\nnorm{\widehat\bTheta(\lambda_N, \tau)- \bTheta^*}\le C_2\rho \Bigl(\frac{(\nu_1+\delta)(d_1+d_2)\log{(d_1+d_2)}}{n}\Bigr)^{\frac{1-q}{2}}
		\]
		for some constants $C_1$ and $C_2$.
		
		(b) Now we switch to the designs with bounded moments. According to Theorem \ref{thm:6}, with probability at least $1-d_1^{1-C\delta}$,
		\[	
			\opnorm{\widehat\bSigma_{\widetilde\bx\widetilde\bx}- \bSigma_{\bx\bx}} \le \sqrt{\frac{\delta d_1\log d_1}{n}},
		\]
		which furthermore implies that $\lambda_{\min}(\widehat\bSigma_{\widetilde\bx \widetilde\bx}) \ge \frac{1}{2}\lambda_{\min}(\bSigma_{\bx\bx})>0$ as long as $d_1/n$ is sufficiently small. Analogous to \eqref{eq:6.38}, we therefore have
		\[
			\vec(\widehat\bDelta)^T\widehat\bSigma_{\widetilde\bX\widetilde\bX}\vec(\widehat\bDelta) \ge \frac{1}{2d_2}\lambda_{\min}(\bSigma_{\bx\bx})\fnorm{\widehat\bDelta}^2.
		\]
		Combining Theorem \ref{thm:1} with the inequality above, we can reach the final conclusion.
	\end{proof}
	
	\begin{proof}[\underline{Proof of Theorem \ref{thm:6}}]
	
		We denote $\bx_i\bx_i^T$ by $\bX_i$ and $\widetilde\bx_i\widetilde\bx_i^T$ by $\widetilde\bX_i$ for ease of notations. Note that
		\[
			\opnorm{\widetilde\bX_i- \E \widetilde\bX_i}\le \opnorm{\widetilde \bX_i}+\opnorm{\E \widetilde\bX_i}=\ltwonorm{\widetilde\bx_i}^2+ \sqrt{R} \le \sqrt{d}\tau^2+\sqrt{R}
		\]
		Also for any $\bv\in \cS^{d-1}$, we have
		\[	
			\begin{aligned}
				\E(\bv^T \widetilde\bX_i^T\widetilde\bX_i \bv) & =\E (\ltwonorm{\widetilde\bx_i}^2 (\bv^T\widetilde\bx_i)^2)\le \E  (\ltwonorm{\bx_i}^2 (\bv^T\bx_i)^2) \\
			 	& = \sum\limits_{j=1}^d \E (x_{ij}^2(\bv^T \bx_i)^2) \le \sum\limits_{j=1}^d \sqrt{\E (x_{ij}^4) \E(\bv^T \bx_i)^4 } \le Rd
			\end{aligned}
		\]
		Then it follows that $\opnorm{\E \widetilde\bX_i^T\widetilde\bX_i}\le Rd$. Since $\opnorm{(\E \widetilde \bX_i)^T\E \widetilde\bX_i}\le \opnorm{\E \bX_i}^2 \le R$, $\opnorm{\E((\widetilde\bX_i- \E\widetilde\bX_i)^T(\widetilde\bX_i-\E \widetilde\bX_i))} \le R(d+1)$. By Theorem 5.29 (Non-commutative Bernstein-type inequality) in \cite{Ver10}, we have for some constant $c$,
		\beq
			P \Big( \opnorm{\frac{1}{n}\sum\limits_{i=1}^n \widetilde\bX_i-\E \widetilde\bX_i}>t \Big)\le 2d\exp\Bigl(-c\bigl(\frac{nt^2}{R(d+1)} \wedge \frac{nt}{\sqrt{d}\tau^2+\sqrt{R}}\bigr)\Bigr).	
			\label{eq:cctt}
		\eeq
		Now we bound the bias $\E \opnorm{\bX_i-\widetilde\bX_i}$. For any $\bv\in \cS^{d-1}$, it holds that
		\[
			\begin{aligned}
				\E (\bv^T(\bX_i-\widetilde\bX_i)\bv) & =\E(((\bv^T\bx_i)^2-(\bv^T\widetilde\bx_i)^2)\ind_{\{\|\bx_i\|_4>\tau\}}) \\
				& \le \E ((\bv^T\bx_i)^2 \ind_{\{\|\bx_i\|_4>\tau\}}) \le \sqrt{\E (\bv^T\bx_i)^4 P(\|\bx_i\|_4>\tau)} \\
				& \le \sqrt{\frac{R^2d}{\tau^4}}=\frac{R\sqrt{d}}{\tau^2}.
			\end{aligned}
		\]
		Therefore we have $\opnorm{\E(\bX_i-\widetilde\bX_i)}\le R\sqrt{d}/\tau^2$. Choose $\tau\asymp (nR/\delta \log d)^{\frac{1}{4}}$ and substitute $t$ with $\sqrt{\delta Rd\log d/n}$. Then we reach the final conclusion by combining the bound of bias and (\ref{eq:cctt}).
		
	\end{proof}

\end{document}